\documentclass[11pt,reqno]{amsart}

\usepackage{amsmath, amsfonts, amsthm, amssymb, mathrsfs, amscd, graphicx, cite, color}
\usepackage{latexsym,hyperref}

\makeatletter
\@namedef{subjclassname@2020}{%
  \textup{2020} Mathematics Subject Classification}
\makeatother

\catcode`!=11
\let\!int\int \def\int{\displaystyle\!int}
\let\!lim\lim \def\lim{\displaystyle\!lim}
\let\!sum\sum \def\sum{\displaystyle\!sum}
\let\!sup\sup \def\sup{\displaystyle\!sup}
\let\!inf\inf \def\inf{\displaystyle\!inf}
\let\!cap\cap \def\cap{\displaystyle\!cap}
\let\!max\max \def\max{\displaystyle\!max}
\let\!min\min \def\min{\displaystyle\!min}
\let\!frac\frac \def\frac{\displaystyle\!frac}
\catcode`!=12

\let\oldsection\section
\renewcommand\section{\setcounter{equation}{0}\oldsection}

\allowdisplaybreaks

\newtheorem{thm}{Theorem}[section]
\newtheorem{pro}{Proposition}[section]
\newtheorem{lem}{Lemma}[section]

\theoremstyle{definition}

\theoremstyle{remark}
\newtheorem{re}{Remark}[section]

\begin{document}
\title[Inviscid Limit of Compressible Viscoelastic Equations]
{Inviscid Limit of Compressible Viscoelastic Equations with the No-Slip Boundary Condition}

\author[D. Wang]{Dehua Wang}
\address{Department of Mathematics, University of Pittsburgh, Pittsburgh, PA 15260, USA.}
\email{dwang@math.pitt.edu}

\author[F. Xie]{Feng Xie}
\address{School of Mathematical Sciences and LSC-MOE, Shanghai Jiao Tong University, Shanghai 200240, P. R. China}
\email{tzxief@sjtu.edu.cn}

%

\begin{abstract}  

The inviscid limit for the two-dimensional  compressible viscoelastic equations on the half plane is considered under the no-slip boundary condition.   When the initial deformation tensor   is   a perturbation of the identity matrix and  the initial density is near a positive constant, we establish the uniform estimates of solutions to the  compressible viscoelastic flows in the  conormal Sobolev spaces. It is well-known that for the  corresponding inviscid limit of the  compressible Navier-Stokes equations with the no-slip boundary condition, one does not expect the uniform energy estimates of solutions due to the appearance of strong boundary layers. However, when the deformation tensor effect is taken into account, our results show that the deformation tensor plays an important role in the vanishing viscosity process and can prevent the formation of strong boundary layers. As a result we are able to justify the inviscid limit of solutions for the compressible viscous flows under the no-slip boundary condition governed by the viscoelastic equations,  based on the uniform conormal regularity estimates achieved in this paper.

\end{abstract}

%
%



%

\date{\today}

\maketitle

\section{Introduction}
In this paper we consider the inviscid limit for the two-dimensional  compressible viscoelastic equations on the half plane:
\begin{align}
\label{1.1}
\left\{
\begin{array}{ll}
\partial_t\rho^\varepsilon+\nabla\cdot(\rho^\varepsilon {\bf u^\varepsilon})=0,\\
\rho^\varepsilon\partial_t{\bf u^\varepsilon}+(\rho^\varepsilon{\bf u^\varepsilon}\cdot\nabla) {\bf u^\varepsilon}-\varepsilon \mu\triangle {\bf u^\varepsilon}-\varepsilon(\mu+\lambda)\nabla\hbox{div}{\bf u^\varepsilon}+\nabla p(\rho^\varepsilon)=\hbox{div}(\rho^\varepsilon{\bf F}^\varepsilon {\bf F}^{\varepsilon \top}),\\
\partial_t {\bf F^\varepsilon}+({\bf u^\varepsilon}\cdot\nabla){\bf F^\varepsilon}=\nabla {\bf u^\varepsilon}\cdot {\bf F^\varepsilon}, \quad t>0,\quad {\bf x}=(x,y)\in\mathbb{R}^2_+:=\mathbb{R}\times\mathbb{R}_+,
\end{array}
\right.
\end{align}
where $\rho^\varepsilon$ denotes the density, ${\bf u^\varepsilon}=(u^\varepsilon, v^\varepsilon)$   the velocity,  and ${\bf F^\varepsilon}=(F^\varepsilon_1, F^\varepsilon_2)^\top$ the deformation tensor matrix with $F^\varepsilon_1=(1+f^\varepsilon_1, f^\varepsilon_2), F^\varepsilon_2=(f^\varepsilon_3, 1+f^\varepsilon_4)$;
 the viscosity coefficients $\mu\varepsilon$ and $\lambda\varepsilon$ satisfy $\mu>0$ and $(\mu+\lambda)>0$ with $\varepsilon\in (0,1)$ being a small parameter, and the pressure $p(\rho)$ is a function of the density $\rho$ that is given by the following formula in the isentropic case:
\begin{align}
\label{P}
p(\rho)=\rho^\gamma,\qquad \gamma\geq 1,
\end{align}
where $\gamma$ is the adiabatic constant. We refer the readers to \cite{DafermosBook2016,Joseph1990,RHN} for the discussions on the physical background of viscoelasticity.
The initial data of (\ref{1.1}) is given by
\begin{align}
\label{ID}
\rho^\varepsilon(0,x,y)=\rho_0(x,y),\quad  {\bf u^\varepsilon}(0,x,y)={\bf u_0}(x,y),\quad {\bf F^\varepsilon}(0,x,y)={\bf F_0}(x,y),
\end{align}
and the no-slip boundary condition is imposed on the velocity,
\begin{align}
\label{1.2}
{\bf u^\varepsilon}(t,x,0)={\bf 0}.
\end{align}
Since the equations of deformation tensor ${\bf F^\varepsilon}$ are a hyperbolic system, one does not need to impose any boundary condition for ${\bf F^\varepsilon}$ due to (\ref{1.2}), and the value of ${\bf F^\varepsilon}$ on the boundary is determined by its initial value. In this paper, we consider the case that
\begin{align} \label{F0}
{\bf F_0}(x,0)=\mathbb{I}_{2\times 2},
\end{align}
where $\mathbb{I}_{2\times 2}$ is a $2\times 2$ identity matrix, which implies
$$f_i^\varepsilon(0,x,0)=0, \quad i=1, 2, 3, 4.$$
A direct calculation from the third equation in (\ref{1.1}) and the boundary condition (\ref{1.2})  shows that
\begin{align}
\label{BCF}
f^\varepsilon_3(t,x,0)=0,
\end{align}
which can be seen from the third equation in \eqref{3.7}.
Formally, when $\varepsilon=0$, the equations in (\ref{1.1}) are reduced to the  following ideal compressible elastodynamic equations:
\begin{align}
\label{IIV}
\left\{
\begin{array}{ll}
\partial_t\rho^0+\nabla\cdot(\rho^0 {\bf u}^0)=0,\\
\rho^0\partial_t{\bf u}^0+(\rho^0{\bf u}^0\cdot\nabla) {\bf u}^0+\nabla p(\rho^0)=\hbox{div}(\rho^0{\bf F}^0{\bf F}^{0\top}),\\
\partial_t {\bf F}^0+({\bf u}^0\cdot\nabla){\bf F}^0=\nabla {\bf u}^0\cdot {\bf F}^0, \quad t>0,\quad {\bf x}=(x,y)\in  \mathbb{R}^2_+. 
\end{array}
\right.
\end{align}
The aim of this paper is to justify the vanishing viscosity limit from  the viscoelastic equations  \eqref{1.1} to the inviscid elastodynamic equations  $\eqref{IIV}$
as $\varepsilon\to 0$ under the no-slip boundary condition \eqref{1.2} on the half plane.

There have been extensive studies on the existence of solutions to both the incompressible and compressible viscoelastic equations; see \cite{LLZ1,LLZ2,L,LZ,HZ, HW2,Hu-Lin2016},  the survey paper \cite{Hu-Lin-Liu2018} and the references therein.
 The inviscid limit of solutions for the Cauchy problem was studied in many papers such as \cite{AD,CW,K,Mas,S} for the incompressible Navier-Stokes equations and    in \cite{CLLM} for the incompressible viscoelastic equations; see also \cite{CP2010,Feireisl2021,HL89,HWY-2,LWW2} and their references for other related vanishing viscosity limits of the Cauchy problem for the compressible Navier-Stokes equations.
%
When the inviscid limit problem is considered in a domain with a physical boundary, the vanishing viscosity limit problem is usually  more challenging due to the possible presence of boundary layers \cite{GGW,Ole,sch,von,WXY}.
In particular, if a strong boundary layer appears,  
  the inviscid limit usually becomes extremely difficult because of the uncontrollability of the vorticity of boundary layer corrector.
If the no-slip boundary condition \eqref{1.2} is replaced by the so-called Navier-slip boundary conditions, the strong boundary layer will disappear, and the inviscid limit  has been established in \cite{WW, WXY1} for the   compressible Navier-Stokes equations.  For the corresponding inviscid limit of the incompressible Navier-Stokes equations with the Navier-slip boundary conditions, we refer the readers to \cite{B, IS, MR, XX,CLQ18} and the references therein.
%

When the no-slip boundary condition is imposed, the inviscid limit problem in a domain with a boundary is more complicated and less developed in analysis.
  To the best of our knowledge, the inviscid limit of the incompressible Navier-Stokes equations with the no-slip boundary condition  was proved  only in the analytic function framework or in the Gevrey settings; see \cite{samm-caf1,samm-caf2, Mae, GMM} and the references therein.
  For the incompressible magnetohydrodynamic (MHD) equations with the no-slip boundary condition,     the well-posedness of solutions to the MHD boundary layer equations and the validity of Prandtl boundary layer expansion in the Sobolev spaces  were obtained in \cite{L-X-Y1, L-X-Y2}  provided that the tangential component of magnetic field does not degenerate near the physical boundary initially; and  it was proved in \cite{L-X-Y3} that    there are no strong boundary layers in the   inviscid limit for the incompressible non-resistive MHD system when the normal component of magnetic field does not degenerate near the physical boundary initially.
  However, the inviscid limit of the compressible Navier-Stokes equations with the no-slip boundary condition on the half plane is still open, except for the linearized Navier-Stokes equations \cite{XY1}, even in  the analytic function spaces or in the Gevrey class owing to the appearance of strong boundary layers \cite{GGW,WXY}.
  In this paper, we consider the inviscid limit for the compressible viscoelastic equations on the half plane with the no-slip boundary condition.
We find that the deformation tensor in viscoelasticity has a significant effect on 
 the vanishing viscosity process and can prevent the formation of strong boundary layers. For this reason we are able to justify 
 the inviscid limit of solutions for the compressible viscous flows governed by the viscoelastic equations \eqref{1.1} under the no-slip boundary condition.

To formulate our main results, let us define the conormal Sobolev spaces that will be used in this paper. Set the conormal derivative operators as  the following:
\begin{align*}
Z_0=\partial_t,\qquad Z_1=\partial_x,\qquad Z_2=\varphi(y)\partial_y,\qquad Z^{\alpha}=Z_0^{\alpha_0}Z_1^{\alpha_1}Z_2^{\alpha_2},
\end{align*}
with $\alpha=(\alpha_0, \alpha_1, \alpha_2)$ and $|\alpha|=\alpha_0+\alpha_1+\alpha_2$.  Here the weight function $\varphi(y)$ satisfies   $\varphi(0)=0,\ \varphi'(0)>0,$ $\|\partial_y^i\varphi\|_{L^\infty}\leq C$  ($i=0,...,m$ for some integer $m>0$),  and $\varphi(y)$ has  uniform lower and upper positive bounds away from the physical boundary, that is $C^{-1}\leq \varphi(y)\leq C$ for some $C>1$ when $y\geq \delta>0$ with some constant $\delta>0$. For example, $\varphi(y)=y/(1+y)$ may  be used as  a weight function.
Define the following two conormal Sobolev spaces:
\begin{align*}
H^m_{co}([0, t]\times \mathbb{R}^2_+)=\{f: Z^\alpha f\in L^2([0, t]\times\mathbb{R}^2_+), |\alpha|\leq m\},
\end{align*}
and
\begin{align*}
\mathcal{H}^m_{co}([0, t]\times \mathbb{R}^2_+)=\{f: Z^\alpha f\in L^\infty([0, t], L^2(\mathbb{R}^2_+)), |\alpha|\leq m\}.
\end{align*}
For a given $t>0$,
\begin{align*}
\|f(t)\|_m^2=\sum_{|\alpha|\leq m}\|Z^\alpha f(t,\cdot)\|_{L^2(\mathbb{R}^2_+)}^2,
\end{align*}
then
\begin{align*}
\|f\|_{H^m_{co}}^2=\int_0^t\|f(s)\|_m^2ds,\qquad \|f\|_{\mathcal{H}^m_{co}}^2=\sup\limits_{0\leq s\leq t}\|f(s)\|_m^2.
\end{align*}
As usual we use the notation:
\begin{align*}
W^{m,\infty}_{co}([0, t]\times\mathbb{R}^2_+)=\{f:Z^\alpha f\in L^\infty([0, t]\times\mathbb{R}^2_+), |\alpha|\leq m\},
\end{align*}
and
\begin{align*}
\|f(t)\|_{m,\infty}=\sum\limits_{|\alpha|\leq m}\|Z^\alpha f(t,\cdot)\|_{L^\infty}.
\end{align*}
Denote the energy by
\begin{align}
N_m(t)=&\|(\rho^\varepsilon-1, {\bf u^\varepsilon}, {\bf F^\varepsilon}-\mathbb{I}_{2\times 2})\|_{\mathcal{H}^m_{co}}^2
+\varepsilon\left(\|\partial_y(\rho^\varepsilon, f_2^\varepsilon)\|_{\mathcal{H}^{m-1}_{co}}^2
+\|\partial^2_y(\rho^\varepsilon, f_2^\varepsilon)\|_{\mathcal{H}^{m-2}_{co}}^2\right) \nonumber\\
&+\|\partial_y(\rho^\varepsilon, {\bf u^\varepsilon}, {\bf F^\varepsilon})\|_{H^{m-1}_{co}}^2
+\|\partial^2_y(\rho^\varepsilon, {\bf u^\varepsilon}, {\bf F^\varepsilon})\|_{H^{m-2}_{co}}^2+\varepsilon\|\nabla {\bf u^\varepsilon}\|_{H^{m}_{co}}^2\nonumber\\
&+\varepsilon^2\left(\|\partial^2_y {\bf u^\varepsilon}\|_{H^{m-1}_{co}}^2+\|\partial^3_y {\bf u^\varepsilon}\|_{H^{m-2}_{co}}^2\right).
\end{align}
We always take $0<\varepsilon<1$ and define
$$\Lambda^m(t)=\{(\rho-1, {\bf u}, {\bf F}-\mathbb{I}_{2\times 2})\in H^m_{co}, \; \partial_y(\rho, {\bf u}, {\bf F})\in H^{m-1}_{co},\; \partial_y^2(\rho, {\bf u}, {\bf F})\in H^{m-2}_{co}\}.$$

Now we state our main theorem as follows.
\begin{thm} \label{Thm1}
Let $m>8$ be an integer. Suppose that the initial data $(\rho_0, {\bf u_0}, {\bf F_0})$ satisfies
 \begin{align}
\label{CCCC}
\|(\rho_0-1, {\bf u_0}, {\bf F_0}-\mathbb{I}_{2\times 2})\|^2_m+\|\partial_y(\rho_0, {\bf u_0}, {\bf F_0})\|_{m-1}^2+\|\partial_y(\nabla\rho_0, \nabla{\bf u_0}, \nabla{\bf F_0})\|_{m-2}^2\leq \sigma_0,
\end{align}
for some sufficiently small positive constant $\sigma_0$, and
\begin{align} \label{rF}
\rho_0\det({\bf F_0})=1,
\qquad
\hbox{\rm div}(\rho_0 {\bf F}_0^T)=0.
\end{align}
Then, there exist a time $T>0$ that is independent of $\varepsilon$ and a unique solution $U^\varepsilon=(\rho^\varepsilon, {\bf u^\varepsilon}, {\bf F^\varepsilon})\in \Lambda^m(T)$ to (\ref{1.1})-(\ref{1.2}), 
such that

\noindent
{\rm (1)} the following estimate holds for $t\in [0, T]$,
\begin{align}
\label{THM1}
N_m(t)+\|(\rho^\varepsilon-1, {\bf u^\varepsilon}, {\bf F^\varepsilon}-\mathbb{I}_{2\times 2})(t)\|_{1,\infty}+\|\nabla (\rho^\varepsilon, {\bf u^\varepsilon}, {\bf F^\varepsilon})(t)\|_{1,\infty}\leq C\sigma_0,
\end{align}
  where $C>0$ is some constant independent of $\varepsilon$;

\noindent
{\rm (2)}  there exists a function $U^0=(\rho^0, {\bf u^0}, {\bf F^0})\in \Lambda^{m}(T)$ 
satisfying the following limit:
\begin{align}
\lim_{\varepsilon\to 0}\sup\limits_{t\in[0,T]}\left\|(U^\varepsilon-U^0, \partial_y(U^\varepsilon-U^0))(\cdot, t)\right\|_{L^\infty(\mathbb{R}^2_+)}= 0,
\end{align}
and   $U^0=(\rho^0, {\bf u^0}, {\bf F^0})$ is a unique solution to the ideal compressible elastodynamic equations (\ref{IIV}) with the same initial data $(\rho_0, {\bf u_0}, {\bf F_0})$ and the no-slip boundary condition.
\end{thm}

\begin{re}
Since we are considering the solutions to the  compressible flows of viscoelasticity in the conormal Sobolev spaces, we need to avoid the appearance of vacuum and degeneracy of deformation tensor matrix, which is guaranteed by the smallness condition \eqref{CCCC}.
\end{re}
\begin{re}
 The time regularity requirements on the initial data can be changed  to the spatial regularity requirements through the equations. We believe that the regularity requirements in Theorem \ref{Thm1} are not optimal.
\end{re}
\begin{re}
It is noted that the identity matrix $\mathbb{I}_{2\times 2}$ is not essential in the analysis. In fact, we only need to assume that the component $f_3^\varepsilon$ in the deformation tensor matrix is zero on the boundary, and the component $1+f_4^\varepsilon$ is not zero initially. We choose the initial data of the deformation tensor as a small perturbation of the identity matrix solely for the sake of simplicity of presentation. Moreover, the form of pressure is also not essential, and  our results can be extended to more general form of pressure without increasing any more difficulty.
\end{re}
\begin{re}
Based on the uniform conormal energy estimates (\ref{THM1}) achieved in the first part of Theorem \ref{Thm1}, the   inviscid limit in the second part of Theorem \ref{Thm1} can be regarded as a direct consequence of the first part  by using some compactness arguments as in \cite{MR}.
\end{re}

Next we shall explain the main difficulties and the strategy to prove the main theorem.
It is well known that when the inviscid limit is considered in a domain with a physical boundary, the uniform estimates of normal derivatives for solutions with respect to the small viscosity parameter are very difficult to obtain. Usually, it is impossible to achieve these uniform estimates due to the presence of strong boundary layers 
for the solutions to both the incompressible and compressible Navier-Stokes equations with the no-slip boundary condition. Surprisingly, if the deformation tensor in viscoelasticity is taken into account,  even though the no-slip boundary condition is imposed on  the velocity, the uniform estimates of normal derivatives for solutions to the compressible  viscoelastic fluid   equations can be achieved, which is the main finding of this paper. In other words, our results in Theorem \ref{Thm1}   show  that the deformation tensor can prevent the strong boundary layers from occurring.
These observations are obviously different from both the compressible and incompressible Navier-Stokes equations with the no-slip boundary condition.  The effect of  the deformation tensor is essentially used in deriving the conormal energy estimates.


We shall present below our strategy to establish the uniform estimates of normal derivatives for all components of ${\bf u^\varepsilon}, {\bf F^\varepsilon}$ and $p^\varepsilon$ in four main steps. \\[3pt]
%
{\bf Step I:} Estimates of $\partial_yv^\varepsilon$ and $\partial_yu^\varepsilon$. \;
From the second and third equations in \eqref{3.7} on the deformation tensor ${\bf F}^\varepsilon$, we can write the normal derivatives in terms of the components of ${\bf F}^\varepsilon$ as the following:
\begin{align*}
\partial_yv^\varepsilon=\frac{1}{1+f_4^\varepsilon}(\partial_tf^\varepsilon_4+u^\varepsilon\partial_xf^\varepsilon_4+v^\varepsilon\partial_yf^\varepsilon_4-f^\varepsilon_2\partial_xv^\varepsilon),
\end{align*}
and
\begin{align*}
\partial_yu^\varepsilon=\frac{1}{1+f^\varepsilon_4}\{\partial_tf^\varepsilon_2+u^\varepsilon\partial_xf^\varepsilon_2+v^\varepsilon\partial_yf^\varepsilon_2-f^\varepsilon_2\partial_xu^\varepsilon\}.
\end{align*}
Then, using  the estimate
\begin{align*}
&\|v^\varepsilon\partial_yf^\varepsilon_4\|_{m-1}=\left\|\frac{v^\varepsilon}{\varphi(y)}\varphi(y)\partial_yf^\varepsilon_4\right\|_{m-1}
\lesssim\left\|\frac{v^\varepsilon}{\varphi(y)}\right\|_{m-1}\|Z_2f^\varepsilon_4\|_{L^\infty}+\left\|\frac{v^\varepsilon}{\varphi(y)}\right\|_{L^\infty}\|f^\varepsilon_4\|_m\\
&\lesssim\left\|\partial_yv^\varepsilon\right\|_{m-1}\|Z_2f^\varepsilon_4\|_{L^\infty}+\left\|\frac{v^\varepsilon}{\varphi(y)}\right\|_{L^\infty}\|f^\varepsilon_4\|_m,
\end{align*}
we see that, at least for the case of  suitably small $\|Z_2f^\varepsilon_4\|_{L^\infty}$, we can control $\|\partial_yv^\varepsilon\|_{m-1}$ by the quantity $(1+P(Q(t)))\|(u^\varepsilon, v^\varepsilon, f^\varepsilon_2, f^\varepsilon_4)\|_{m}$, where $Q(t)$ denotes the $W^{1, \infty}_{co}$-norm of the solution and its first-order derivatives, and $P$ is a generic polynomial that will be frequently used in the estimates of the paper.  Using the similar arguments and the estimate of  $\|\partial_yv^\varepsilon\|_{m-1}$, one can derive the estimate of   $\|\partial_yu^\varepsilon\|_{m-1}$.
Here the {\it a priori} assumption that $1+f_4^\varepsilon$ has a positive lower bound is required, which is guaranteed by the requirement that the initial data of the deformation tensor matrix is a small perturbation of the identity matrix.
 We remark that the deformation tensor ${\bf F}^\varepsilon$ plays an essential role here. It is not clear how to obtain the uniform conormal estimates of the normal derivatives  for the tangential velocity   $u^\varepsilon$ without the viscoelasticity effect under the no-slip boundary condition.
\\[3pt]
{\bf Step II:} Estimates of $\partial_yf_2^\varepsilon$. \;
As for the estimate of $\partial_yf_2^\varepsilon$, the equation of $u^\varepsilon$ will be used. However, notice that there is also a second-order normal derivative term of $\varepsilon\mu\partial^2_yu^\varepsilon$ in the equation of $u^\varepsilon$, as a consequence we  need to estimate $\rho^\varepsilon(1+f_4^\varepsilon)\partial_yf^\varepsilon_2+\varepsilon\mu\partial^2_yu^\varepsilon$ instead of $\rho^\varepsilon(1+f_4^\varepsilon)\partial_yf^\varepsilon_2$.
Due to the conormal derivatives terms on the right hand side of the equation, taking the $L^2$-norm on both sides will produce a mixed term of $2\mu\varepsilon\rho^\varepsilon(1+f_4^\varepsilon)\partial_yf^\varepsilon_2\partial^2_yu^\varepsilon$. 
To handle the mixed term, we apply the operator $\partial_y$ on the equation of $f_2^\varepsilon$
and multiply this equation by $2\mu\varepsilon\partial_yf_2^\varepsilon$, then we  produce the same mixed term with the opposite sign.
 Adding these two estimates together will cancel the mixed terms and achieve the $L^2$ estimates of $\rho^\varepsilon(1+f_4^\varepsilon)\partial_yf^\varepsilon_2$ and $\varepsilon\mu\partial^2_yu^\varepsilon$. Here the {\it a priori} assumption that $1+f_4^\varepsilon$ and $\rho^\varepsilon$ have  positive lower bounds is required, which is guaranteed by the requirement that the initial data of  the deformation tensor matrix is a small perturbation of the identity matrix and the density is a small perturbation of the constant $1$.\\[3pt]
{\bf Step III:} Estimates of $\partial_yp^\varepsilon$. \;
By the similar arguments to those in Step II, we use the equations of $v^\varepsilon$ and $\rho^\varepsilon$ in the following manner:
\begin{align*}
\partial_yp^\varepsilon-(2\mu+\lambda)\varepsilon\partial_y^2v^\varepsilon=..., \qquad
\partial_tp^\varepsilon+\gamma p^\varepsilon\partial_yv^\varepsilon=...
\end{align*}
Moreover, the following relationship will be essentially used:
\begin{align*}
\rho^\varepsilon(1+f_4^\varepsilon)\partial_yf_4^\varepsilon=-\partial_y\rho^\varepsilon+...
\end{align*}
due to $\partial_x(\rho^\varepsilon f_2^\varepsilon)+\partial_y(\rho^\varepsilon (1+f_4^\varepsilon))=0$, which is guaranteed by imposing the same constraint on the initial data. This relationship is used to change the term involving $\partial_yf_4^\varepsilon$ in the equation of $v^\varepsilon$ to the form of $\partial_y p^\varepsilon$, then it can be merged into $\partial_yp^\varepsilon$ on the left hand side.
In this way, the $L^2$ estimates of $\left(1+\frac{(1+f_4^\varepsilon)^2}{\gamma(\rho^\varepsilon)^{\gamma-1}}\right)\partial_yp^\varepsilon$ and $(2\mu+\lambda)\varepsilon\partial_y^2v^\varepsilon$ are established.\\[3pt]
{\bf Step IV:} Estimates of $\partial_yf_3^\varepsilon, \partial_yf_4^\varepsilon$ and $\partial_yf_1^\varepsilon$. \;
In this paper, the initial data of ${\bf F_0}$ and $\rho_0$ are required to satisfy the   natural constraints \eqref{rF},
then the smooth solutions also satisfy the same relationship (c.f. \cite{HW2}).
Consequently, it follows that
\begin{align*}
\partial_yf^\varepsilon_3=\frac{1}{\rho^\varepsilon}\{-\partial_x(\rho^\varepsilon(1+f_1^\varepsilon))-f^\varepsilon_3\partial_y\rho^\varepsilon\}, \quad
\partial_yf^\varepsilon_4=\frac{1}{\rho^\varepsilon}\{-\partial_x(\rho^\varepsilon f^\varepsilon_2)-(1+f^\varepsilon_4)\partial_y\rho^\varepsilon\}.
\end{align*}
Thus, the estimates of $\partial_yf_3^\varepsilon$ and $\partial_yf_4^\varepsilon$ can be derived.
As for the estimate of $\partial_yf_1^\varepsilon$, it can be directly deduced from the following equation:
\begin{align*}
\partial_yf^\varepsilon_1=&\frac{1}{(1+f^\varepsilon_4)}\left\{\partial_y\left(\frac{1}{\rho^\varepsilon}\right)+\partial_y(f^\varepsilon_2f^\varepsilon_3)-(1+f^\varepsilon_1)\partial_yf^\varepsilon_4\right\}\\
=&\frac{1}{(1+f^\varepsilon_4)}\left\{\partial_y\left(\frac{1}{\rho^\varepsilon}\right)+f^\varepsilon_3\partial_yf^\varepsilon_2-\frac{f^\varepsilon_2}{\rho^\varepsilon}
(f^\varepsilon_3\partial_y\rho^\varepsilon+\partial_x(\rho^\varepsilon f^\varepsilon_1))+\frac{f^\varepsilon_1}{\rho^\varepsilon}(f^\varepsilon_4\partial_y\rho^\varepsilon+\partial_x(\rho^\varepsilon f_2^\varepsilon))\right\},
\end{align*}
using the property: $\rho^\varepsilon\det({\bf F^\varepsilon})=1$.

With the above four steps we obtain the estimates of the first order normal derivatives. Finally, to close the energy estimates, it suffice to control $Q(t)$ by the conormal energy estimates. According to Lemma \ref{L2}, in order to estimate $Q(t)$, we still need to derive the conormal estimates of the second order normal derivatives.  We repeat the above four steps for the second order normal derivatives to complete the energy estimate procedure, and then justify the inviscid limit of \eqref{1.1} under the no-slip boundary condition.

The paper is organized as follows. In Section 2, we give some preliminaries and technical lemmas. Section 3 is devoted to deriving the uniform conormal energy estimates of solutions to (\ref{1.1})-(\ref{1.2}). In Section 4, we establish the conormal estimates for the normal derivatives of solutions to (\ref{1.1})-(\ref{1.2}). Based on the uniform estimates established in Sections 3 and 4, we prove main Theorem \ref{Thm1}   in Section 5.

\bigskip

\section{Preliminary}

In this section, we  shall present some technical lemmas that will be used frequently in the analysis of the paper later.

We first recall the following generalized Sobolev-Gagliardo-Nirenberg-Moser inequality in the conormal Sobolev spaces (see  \cite{Gues} and  the proof):
\begin{lem}
\label{L1}
For the functions $f, g\in L^\infty([0,t]\times\mathbb{R}^2_+)\cap H_{co}^m([0,t]\times\mathbb{R}^2_+)$, it holds that
\begin{align*}
\int_0^t\|(Z^\alpha fZ^\beta g)(s)\|^2ds\lesssim \|f\|_{L_{t,{\bf x}}^\infty}^2\int_0^t\|g(s)\|_m^2ds+\|g\|_{L_{t,{\bf x}}^\infty}^2\int_0^t\|f(s)\|_m^2ds\quad \hbox{for }\ |\alpha|+|\beta|\leq m.
\end{align*}
\end{lem}
Here we note that the notation $A\lesssim B$ means $A\leq CB$ for some generic constant $C$ and $\|\cdot\|\triangleq\|\cdot\|_{L^2(\mathbb{R}^2_+)}$.

Then we recall   the following anisotropic Sobolev embedding property in the conormal Sobolev spaces (see \cite{Paddick} and the proof):
\begin{lem}
\label{L2}
Let $f(t, {\bf x})\in H^3_{co}([0,t]\times\mathbb{R}^2_+)$ and $\partial_yf(t, {\bf x})\in H^2_{co}([0,t]\times\mathbb{R}^2_+)$, then
\begin{align*}
\|f\|_{L_{t,{\bf x}}^\infty}^2\lesssim \|f(0)\|_2^2+\|\partial_yf(0)\|_1^2+\int_0^t\left(\|f(t)\|_3^2+\|\partial_yf(t)\|_2^2\right)ds.
\end{align*}
\end{lem}

To handle the commutators, it is helpful to introduce the following formula   (see   \cite{Paddick} and the proof): 
\begin{lem}
\label{L3}
There exist two families of bounded smooth functions $\{\phi_{k,m}(y)\}_{0\leq k\leq m-1}$ and \\
$\{\phi^{k,m}(y)\}_{0\leq k\leq m-1}$, such that
\begin{align*}
[Z_2^m, \partial_y]=\sum\limits_{k=0}^{m-1}\phi_{k,m}(y)Z_2^k\partial_y=\sum\limits_{k=0}^{m-1}\phi^{k,m}(y)\partial_y Z_2^k.
\end{align*}
\end{lem}
Based on Lemma \ref{L3}, the following lemma holds true.
\begin{lem}
\label{L4}
 There exists a generic constant $C>1$, such that
\begin{align*}
C^{-1}\sum\limits_{k=0}^m\|\nabla Z^ku\|^2_{L^2}\leq \|\nabla u\|^2_m\leq C\sum\limits_{k=0}^m\|\nabla Z^ku\|^2_{L^2}.
\end{align*}
\end{lem}
\begin{proof}
Denote $k=(k_0, k_1, k_2)$, then
\begin{align*}
\sum\limits_{k=0}^m\|\nabla Z^k u\|=&\sum\limits_{k=0}^m\|Z^k \nabla u\|+\sum\limits_{k=0}^m\|Z_0^{k_0}Z^{k_1}_1[\partial_y, Z_2^{k_2}] u\|, \\ 
\leq& \|\nabla u\|_m+\sum\limits_{j=0}^{k_2-1}\|\phi_{j,k_2}(y)Z_0^{k_0}Z_1^{k_1}Z^j_2\partial_y u\|\leq C\|\nabla u\|_m,
\end{align*}
where the commutator   $[\partial_x, Z^k]=0$ is used. The other inequality can be proved similarly.
\end{proof} 

\begin{lem}
\label{L5}
There exist two families of bounded smooth functions $\{\phi_{1, k,m}(y), \phi_{2, k,m}(y)\}_{0\leq k\leq m-1}$ and $\{\phi^{1, k,m}(y), \phi^{2, k,m}(y)\}_{0\leq k\leq m-1}$, such that
\begin{align*}
[Z_2^m, \partial^2_y]=\sum\limits_{k=0}^{m-1}\left(\phi_{1, k,m}(y)Z_2^k\partial_y+ \phi_{2, k,m}(y)Z_2^k\partial^2_y\right)
=\sum\limits_{k=0}^{m-1}\left(\phi^{1, k,m}(y)\partial_y Z_2^k+\phi^{2, k,m}(y)\partial^2_y Z_2^k\right).
\end{align*}
\end{lem}

\begin{lem}
\label{L6}
There exists a family of bounded smooth functions $\{\psi_{k,m}(y)\}_{0\leq k\leq m-1}$, such that
\begin{align*}
[Z_2^m, 1/\varphi (y)]f=\sum\limits_{k=0}^{m-1}\psi_{k,m}(y)Z_2^k(f/\varphi); 
\end{align*}
and there exists a family of bounded smooth functions $\{\psi^{k,m}(y)\}_{0\leq k\leq m-1}$, such that
\begin{align*}
[Z_2^m, \varphi(y)]f=\sum\limits_{k=0}^{m-1}\psi^{k,m}(y)Z_2^k(\varphi(y)f).
\end{align*}
\end{lem}
The above two Lemmas \ref{L5} and \ref{L6}  and the proofs can also  be found in \cite{Paddick}.

\begin{pro}
\label{P1}
Assume that  $(\rho^\varepsilon, {\bf u^\varepsilon}, {\bf F^\varepsilon})$ is a smooth solution to (\ref{1.1})-(\ref{1.2}). Then, the following identities
\begin{align}
\rho^\varepsilon\det({\bf F^\varepsilon})=1,
\end{align}
and
\begin{align}
\hbox{\rm div}(\rho^\varepsilon {\bf F^{\varepsilon \top}})=0
\end{align}
hold for $t\in [0, T]$, provided that these constraints are satisfied initially.
\end{pro}
The proof of Proposition \ref{P1} can be found in \cite{HW2}.

\bigskip

\section{Conormal Energy Estimates}

In this section we shall derive the uniform conormal estimates of solutions to (\ref{1.1})-(\ref{1.2}). Firstly, we set
$$\rho^\varepsilon=1+\tilde{\rho}^\varepsilon.$$ 
For simplicity of presentation, we omit the symbols $\varepsilon$ and ``$\sim$" in this section without causing any confusion. It is convenient to rewrite system (1.1) as the following:
\begin{align}
\label{3.1}
\left\{
\begin{array}{ll}
\partial_t\rho+\nabla\cdot((1+\rho){\bf u})=0,\\
(1+\rho)\partial_t{\bf u}+((1+\rho){\bf u}\cdot\nabla){\bf u}-((1+\rho)({\bf G_1}+{\bf e_1})\cdot\nabla){\bf G_1}\\
\qquad-((1+\rho)({\bf G_2}+{\bf e_2})\cdot\nabla){\bf G_2}+\nabla p=-\varepsilon\mu\nabla\times \omega+\varepsilon(2\mu+\lambda)\nabla\hbox{div}{\bf u},\\
\partial_t{\bf G_1}+({\bf u}\cdot\nabla){\bf G_1}=(({\bf G_1}+{\bf e_1})\cdot\nabla){\bf u},\\
\partial_t{\bf G_2}+({\bf u}\cdot\nabla){\bf G_2}=(({\bf G_2}+{\bf e_2})\cdot\nabla){\bf u},
\end{array}
\right.
\end{align}
with
\begin{align*}
{\bf u}=(u, v),\; \omega=\partial_yu-\partial_xv,\;  {\bf G_1}=(f_1, f_3),\;  {\bf G_2}=(f_2, f_4),\;  {\bf e_1}=(1, 0), \;  {\bf e_2}=(0, 1),
\end{align*}
and $\nabla\times=(-\partial_y,\ \partial_x)$.

The boundary conditions are imposed as the following:
\begin{align}
\label{3.2a}
{\bf u}(t,x,0)=f_3(t,x,0)=0.
\end{align}
Note that the boundary condition $f_3(t,x,0)=0$ is a consequence of the initial condition on the boundary \eqref{F0} (c.f. \eqref{BCF}).
We will establish the following uniform conormal energy estimates in this section.
\begin{pro}
\label{P3.1}
Under the assumptions in Theorem \ref{Thm1}, there exists a sufficiently small $\varepsilon_0>0$, such that for any $0<\varepsilon<\varepsilon_0$, the smooth solutions $(\rho, {\bf u}, {\bf G_1}, {\bf G_2})$ to (\ref{3.1})-(\ref{3.2a}) satisfy the following {\it a priori} estimates:
\begin{align}
\label{3.3a}
&\|(p-1, {\bf u}, {\bf G_1}, {\bf G_2})(t)\|_m^2+\varepsilon\int_0^t\|\nabla{\bf u}(\tau)\|_{m}^2d\tau\nonumber\\
\lesssim&\|(p-1, {\bf u}, {\bf G_1}, {\bf G_2})(0)\|_m^2+\delta \int_0^t\|\nabla p(\tau)\|_{m-1}^2d\tau+\delta\varepsilon^2\int_0^t\|\nabla^2{\bf u}(\tau)\|_{m-1}^2d\tau\nonumber\\
&+(1+P(Q(t)))\int_0^t\left(\|\nabla {\bf u}(\tau)\|_{m-1}^2+\|\nabla {\bf G_1}(\tau)\|_{m-1}^2+\|\nabla {\bf G_2}(\tau)\|_{m-1}^2\right)d\tau\\
&+(1+P(Q(t)))\int_0^t(\|\rho(\tau)\|_m^2+\|{\bf u}(\tau)\|_m^2+\| {\bf G_1}(\tau)\|_m^2+\| {\bf G_2}(\tau)\|_m^2)d\tau\nonumber,
\end{align}
for some small $\delta>0$ to be determined later, where
\begin{align*}
Q(t)=\sup\limits_{0\leq \tau\leq t}\{\|p-1, {\bf u}, {\bf G_1}, {\bf G_2}\|_{1, \infty}+\|\nabla p, \nabla {\bf u}, \nabla {\bf G_1}, \nabla {\bf G_2}\|_{1, \infty}\},
\end{align*}
and $P(\cdot)$ is a polynomial.
\end{pro}

\begin{proof} 
Applying the conormal derivative operator $Z^\alpha (|\alpha|\leq m)$ to the system (\ref{3.1}) yields the following system:
\begin{equation}\label{3.2}
 \begin{cases}
\partial_tZ^\alpha\rho+Z^\alpha\nabla\cdot((1+\rho){\bf u})=0,\\
(1+\rho)\partial_t Z^\alpha{\bf u}+((1+\rho){\bf u}\cdot\nabla)Z^\alpha{\bf u}-((1+\rho)({\bf G_1}+{\bf e_1})\cdot\nabla)Z^\alpha{\bf G_1}\\
\quad \; -((1+\rho)({\bf G_2}+{\bf e_2})\cdot\nabla)Z^\alpha{\bf G_2}+Z^\alpha\nabla p=-\varepsilon\mu Z^\alpha\nabla\times\omega+\varepsilon(2\mu+\lambda)Z^\alpha\nabla\hbox{div}{\bf u}+\sum\limits_{i=1}^3\mathcal{C}_i^\alpha,\\
\partial_tZ^\alpha{\bf G_1}+({\bf u}\cdot\nabla)Z^\alpha{\bf G_1}=(({\bf G_1}+{\bf e_1})\cdot\nabla)Z^\alpha{\bf u}+\mathcal{C}_4^\alpha,\\
\partial_tZ^\alpha{\bf G_2}+({\bf u}\cdot\nabla)Z^\alpha{\bf G_2}=(({\bf G_2}+{\bf e_2})\cdot\nabla)Z^\alpha{\bf u}+\mathcal{C}_5^\alpha,
\end{cases}
\end{equation}
with
\begin{equation*}
\left\{
\begin{split}
 \mathcal{C}_1^\alpha&=-[Z^\alpha, (1+\rho)]{\bf u}_t=-\sum\limits_{|\beta|\geq 1, \beta+\kappa=\alpha}C_{\alpha\beta}Z^\beta\rho Z^\kappa{\bf u}_t,\\
\mathcal{C}_2^\alpha&=-[Z^\alpha, (1+\rho){\bf u}\cdot\nabla]{\bf u}\\
 &=-\sum\limits_{|\beta|\geq 1, \beta+\kappa=\alpha}C_{\alpha\beta}Z^\beta((1+\rho){\bf u}) Z^\kappa\nabla{\bf u}-(1+\rho){\bf u}\cdot[Z^\alpha, \nabla]{\bf u},\\
\mathcal{C}_3^\alpha&=[Z^\alpha, (1+\rho)({\bf G_1}+{\bf e_1})\cdot\nabla]{\bf G_1}+[Z^\alpha, (1+\rho)({\bf G_2}+{\bf e_2})\cdot\nabla]{\bf G_2}\\
 &=\sum\limits_{|\beta|\geq 1, \beta+\kappa=\alpha}C_{\alpha\beta}Z^\beta((1+\rho)({\bf G_1}+{\bf e_1})) Z^\kappa\nabla{\bf G_1}+(1+\rho)({\bf G_1}+{\bf e_1})\cdot[Z^\alpha, \nabla]{\bf G_1}\\
 &\quad +\sum\limits_{|\beta|\geq 1, \beta+\kappa=\alpha}C_{\alpha\beta}Z^\beta((1+\rho)({\bf G_2}+{\bf e_2})) Z^\kappa\nabla{\bf G_2}+(1+\rho)({\bf G_2}+{\bf e_2})\cdot[Z^\alpha, \nabla]{\bf G_2},
\end{split}
\right.
\end{equation*}
and
\begin{equation*}
\left\{
\begin{split}
\mathcal{C}_4^\alpha=&-[Z^\alpha, {\bf u}\cdot\nabla]{\bf G_1}+[Z^\alpha, ({\bf G_1}+{\bf e_1})\cdot\nabla]{\bf u}\\
 =&-\sum\limits_{|\beta|\geq 1, \beta+\kappa=\alpha}C_{\alpha\beta}Z^\beta{\bf u} Z^\kappa\nabla{\bf G_1}-{\bf u}\cdot[Z^\alpha, \nabla]{\bf G_1}\\
 &+\sum\limits_{|\beta|\geq 1, \beta+\kappa=\alpha}C_{\alpha\beta}Z^\beta({\bf G_1}+{\bf e_1}) Z^\kappa\nabla{\bf u}-({\bf G_1}+{\bf e_1})\cdot[Z^\alpha, \nabla]{\bf u},\\
\mathcal{C}_5^\alpha=&-[Z^\alpha, {\bf u}\cdot\nabla]{\bf G_2}+[Z^\alpha, ({\bf G_2}+{\bf e_2})\cdot\nabla]{\bf u}\\
 =&-\sum\limits_{|\beta|\geq 1, \beta+\kappa=\alpha}C_{\alpha\beta}Z^\beta{\bf u} Z^\kappa\nabla{\bf G_2}-{\bf u}\cdot[Z^\alpha, \nabla]{\bf G_2}\\
 &+\sum\limits_{|\beta|\geq 1, \beta+\kappa=\alpha}C_{\alpha\beta}Z^\beta({\bf G_2}+{\bf e_2}) Z^\kappa\nabla{\bf u}-({\bf G_2}+{\bf e_2})\cdot[Z^\alpha, \nabla]{\bf u}.
\end{split}
\right.
\end{equation*}
Multiplying the second equation in (\ref{3.2}) by $Z^\alpha {\bf u}$, the third equation by $(1+\rho) Z^\alpha {\bf G_1}$, and the fourth equation by $(1+\rho) Z^\alpha {\bf G_2}$, adding the resulting equations together, and integrating them over $\mathbb{R}^2_+$, we have
\begin{align}
\label{3.3}
&\frac{d}{dt}\int\frac12(1+\rho)(|Z^\alpha{\bf u}|^2+|Z^\alpha{\bf G_1}|^2+|Z^\alpha{\bf G_2}|^2)d{\bf x}+\int Z^\alpha\nabla p\cdot Z^\alpha {\bf u} d{\bf x}\nonumber\\
=&-\mu\varepsilon\int Z^\alpha\nabla\times \omega\cdot Z^\alpha{\bf u}d{\bf x}+(2\mu+\lambda)\varepsilon\int Z^\alpha\nabla\hbox{div}{\bf u}\cdot Z^\alpha{\bf u}d{\bf x}\\
&+\int(\mathcal{C}_1^\alpha+\mathcal{C}_2^\alpha+\mathcal{C}_3^\alpha)\cdot Z^\alpha{\bf u}d{\bf x}+\int(1+\rho) \mathcal{C}_4^\alpha\cdot Z^\alpha{\bf G_1}d{\bf x}+\int(1+\rho)  \mathcal{C}_5^\alpha\cdot Z^\alpha{\bf G_2}d{\bf x}\nonumber,
\end{align}
where the  integration by parts, the boundary conditions (\ref{3.2a})  and the following facts are used:
\begin{align*}
\partial_t\rho+\nabla\cdot((1+\rho){\bf u})=0,\qquad \hbox{div}((1+\rho){\bf F}^T)=0,
\end{align*}
due to the first equation in (\ref{3.1}) and Proposition \ref{P1}.  

Notice that
\begin{align*}
&-\varepsilon\int Z^\alpha\nabla\times \omega\cdot Z^\alpha{\bf u}d{\bf x}\\
=&-\varepsilon\int \nabla\times Z^\alpha\omega\cdot Z^\alpha{\bf u}d{\bf x}-\varepsilon\int [Z^\alpha, \nabla\times]\omega\cdot Z^\alpha{\bf u}d{\bf x}\\
\leq& -\varepsilon\int Z^\alpha\omega \nabla\times Z^\alpha{\bf u}d{\bf x}+C\varepsilon\|\nabla^2{\bf u}\|_{m-1}\|{\bf u}\|_m\\
=& -\varepsilon\int Z^\alpha\omega Z^\alpha\nabla\times {\bf u}d{\bf x}-\varepsilon\int Z^\alpha\omega [Z^\alpha, \nabla\times]{\bf u}d{\bf x}+C\varepsilon\|\nabla^2{\bf u}\|_{m-1}\|{\bf u}\|_m\\
\leq& -\varepsilon\int |Z^\alpha\omega|^2d{\bf x}+C\varepsilon\|\nabla^2{\bf u}\|_{m-1}\|{\bf u}\|_m+C\varepsilon\|\nabla{\bf u}\|_{m-1}\|\nabla{\bf u}\|_m\\
\leq& -\varepsilon \|\nabla\times Z^\alpha{\bf u}\|^2+\delta\varepsilon^2\|\nabla^2{\bf u}\|^2_{m-1}+\delta\varepsilon\|\nabla{\bf u}\|^2_m+C_\delta\varepsilon\|\nabla{\bf u}\|^2_{m-1}+C_\delta\|{\bf u}\|^2_m,
\end{align*}
for some small $\delta>0$ to be determined later, where for the first and second inequalities Lemmas \ref{L3} and \ref{L4} are used.
Similarly, one has
\begin{align*}
&\varepsilon\int Z^\alpha\nabla\hbox{div}{\bf u}\cdot Z^\alpha{\bf u}d{\bf x}\\
= &\varepsilon\int \nabla Z^\alpha\hbox{div}{\bf u}\cdot Z^\alpha{\bf u}d{\bf x}+\varepsilon\int [Z^\alpha, \nabla] \hbox{div}{\bf u}\cdot Z^\alpha{\bf u}d{\bf x}\\
\leq &-\varepsilon\int Z^\alpha\hbox{div}{\bf u}\cdot \hbox{div}Z^\alpha{\bf u}d{\bf x}+\varepsilon\|\nabla^2{\bf u}\|_{m-1}\|{\bf u}\|_{m}\\
\leq &-\varepsilon\| \hbox{div}Z^\alpha{\bf u}\|^2+\delta\varepsilon^2\|\nabla^2{\bf u}\|^2_{m-1}+\delta\varepsilon\|\nabla{\bf u}\|^2_m+C_\delta(\varepsilon\|\nabla{\bf u}\|_{m-1}^2+\|{\bf u}\|_{m}^2).
\end{align*}
Combining (\ref{3.3}) and the following inequality
\begin{align*}
2c_1\|\nabla Z^\alpha{\bf u}\|_{L^2}^2\lesssim \mu\|\nabla\times Z^\alpha{\bf u}\|_{L^2}^2+(2\mu+\lambda)\|\hbox{div} Z^\alpha{\bf u}\|_{L^2}^2,
\end{align*}
where $c_1$ is a generic constant, we obtain
\begin{align}
\label{3.4}
&\int\frac12(1+\rho)(|Z^\alpha{\bf u}|^2+|Z^\alpha{\bf G_1}|^2+|Z^\alpha{\bf G_2}|^2)d{\bf x}+c_1\varepsilon\int_0^t\|\nabla Z^\alpha{\bf u}\|_{L^2}^2d\tau \nonumber\\
&\qquad +\int_0^t\int Z^\alpha\nabla p\cdot Z^\alpha {\bf u} d{\bf x}d\tau\nonumber\\
\lesssim&\int\frac12(1+\rho_0)(|Z^\alpha{\bf u}|^2(0)+|Z^\alpha{\bf G_1}|^2(0)+|Z^\alpha{\bf G_2}(0)|^2)d{\bf x}\\
&\qquad+\delta\varepsilon^2\int_0^t\|\nabla^2{\bf u}(\tau)\|_{m-1}^2d\tau+\delta\varepsilon\int_0^t\|\nabla{\bf u}(\tau)\|_{m}^2d\tau+C_\delta\varepsilon\int_0^t\|\nabla{\bf u}(\tau)\|_{m-1}^2d\tau \nonumber\\
&\qquad+C_\delta\int_0^t\|{\bf u}(\tau)\|_{m}^2d\tau+\int_0^t(\|\mathcal{C}_1^\alpha\|^2+\|\mathcal{C}_2^\alpha\|^2+\|\mathcal{C}_3^\alpha\|^2+\|\mathcal{C}_4^\alpha\|^2+\|\mathcal{C}_5^\alpha\|^2)d\tau \nonumber\\
&\qquad +C\int_0^t(\| Z^\alpha{\bf u}\|^2+\| Z^\alpha{\bf G_1}\|^2 +\|Z^\alpha{\bf G_2}\|)^2  d\tau\nonumber,
\end{align}
where the {\it a priori} assumption of $\|\rho\|_{L^\infty}\leq 1/2$ is used.

Next, we handle the term  involving the pressure. First,
\begin{align*}
&\int_0^t\int Z^\alpha\nabla p\cdot Z^\alpha{\bf u}d{\bf x}d\tau=\int_0^t\int Z^\alpha\nabla (p-1)\cdot Z^\alpha{\bf u}d{\bf x}d\tau\\
=&\int_0^t\int \nabla Z^\alpha (p-1)\cdot Z^\alpha{\bf u}d{\bf x}d\tau+\int_0^t\int [Z^\alpha, \nabla] (p-1)\cdot Z^\alpha{\bf u}d{\bf x}d\tau\\
\geq &-\int_0^t\int Z^\alpha (p-1)\cdot \hbox{div}Z^\alpha{\bf u}d{\bf x}d\tau-\int_0^t\int \|{\bf u}\|_m\|\nabla p\|_{m-1}d\tau\\
\geq &-\int_0^t\int Z^\alpha (p-1)\cdot Z^\alpha\hbox{div}{\bf u}d{\bf x}d\tau-\delta\int_0^t\int \|\nabla p\|^2_{m-1}d\tau\\
&-C_\delta\int_0^t(\|p-1\|_m^2+\| {\bf u}\|_m^2+\|\nabla{\bf u}\|_{m-1}^2)d\tau.
\end{align*}
Then, it follows from the first equation in (\ref{3.1}) that
\begin{align*}
\hbox{div}{\bf u}=-\frac{p_t}{\gamma p}-\frac{{\bf u}}{\gamma p}\cdot \nabla p=-\frac{(p-1)_t}{\gamma p}-\frac{{\bf u}}{\gamma p}\cdot \nabla (p-1).
\end{align*}
Applying the operator $Z^\alpha (|\alpha|\leq m)$ on the above equation gives
\begin{align*}
Z^\alpha\hbox{div}{\bf u}=&-\frac{Z^\alpha (p-1)_t}{\gamma p}-\frac{{\bf u}}{\gamma p}\cdot Z^\alpha\nabla (p-1)\\
&-\sum\limits_{|\beta|\geq 1, \beta+\kappa=\alpha}C_{\alpha\beta}Z^\beta\left(\frac{1}{\gamma p}\right)Z^\kappa (p-1)_t\\
&-\sum\limits_{|\beta|\geq 1, \beta+\kappa=\alpha}C_{\alpha\beta}Z^\beta\left(\frac{{\bf u}}{\gamma p}\right)Z^\kappa \nabla (p-1).
\end{align*}
Now we deal with the above right hand side term by term as follows. For the first term, one has,
\begin{align*}
&\int_0^t\int Z^\alpha (p-1)\cdot \frac{Z^\alpha (p-1)_t}{\gamma p}d{\bf x}d\tau\\
=&\int_0^t\int \left(\frac{|Z^\alpha (p-1)|^2}{2\gamma p}\right)_td{\bf x}d\tau-\int_0^t\int |Z^\alpha (p-1)|^2\left(\frac{1}{2\gamma p}\right)_td{\bf x}d\tau\\
\geq& \int \frac{|Z^\alpha (p-1)(t)|^2}{2\gamma p(t)}d{\bf x}-\int \frac{|Z^\alpha (p-1)(0)|^2}{2\gamma p(0)}d{\bf x}-C\left\|\frac{p_t}{p^2}\right\|_{L^\infty}\int_0^t\|Z^\alpha(p-1)\|^2d\tau,\\
\geq& \int \frac{|Z^\alpha (p-1)(t)|^2}{2\gamma p(t)}d{\bf x}-\int \frac{|Z^\alpha (p-1)(0)|^2}{2\gamma p(0)}d{\bf x}-(1+P(Q(t)))\int_0^t\|Z^\alpha(p-1)\|^2d\tau,
\end{align*}
where and hereafter we use the {\it a priori} assumption that $\|\rho\|_{L^\infty}\leq 1/2$, which will be justified later by choosing $\sigma_0$ in Theorem \ref{Thm1} suitably small; for the second term,
\begin{align*}
& \int_0^t\int Z^\alpha (p-1)\frac{{\bf u}}{\gamma p}\cdot Z^\alpha \nabla (p-1) d{\bf x}d\tau\\
=&\int_0^t\int Z^\alpha (p-1) \nabla Z^\alpha (p-1)\cdot \frac{{\bf u}}{\gamma p} d{\bf x}d\tau+\int_0^t\int Z^\alpha (p-1) \frac{{\bf u}}{\gamma p} [\nabla, Z^\alpha ](p-1) d{\bf x}d\tau\\
\geq&-\int\nabla\cdot\frac{{\bf u}}{2\gamma p} |Z^\alpha (p-1)(t)|^2d{\bf x}-\delta \|\nabla p\|_{m-1}^2d\tau-C_\delta\left\|\frac{{\bf u}}{p}\right\|^2_{L^\infty}\int_0^t\|(p-1)(\tau)\|^2_md\tau\\
\geq&-\delta \int_0^t\|\nabla p\|_{m-1}^2d\tau-C_\delta(1+P(Q(t)))\int_0^t\|(p-1)(\tau)\|^2_md\tau,
\end{align*}
for the third term, from direct calculations we have
\begin{align*}
&\sum\limits_{|\beta|\geq 1, \beta+\kappa=\alpha}C_{\alpha\beta}\int_0^t\int Z^\alpha (p-1) Z^\beta\left(\frac{1}{\gamma p}\right)Z^\kappa (p-1)_t\\
\geq  &-C\|p\|_{1,\infty}\int_0^t\|(p-1)(\tau)\|_m\|Z^\alpha (p-1)(\tau)\|_{L^2}d\tau\\
\geq & -(1+P(Q(t)))\int_0^t\|(p-1)(\tau)\|_m\|Z^\alpha (p-1)(\tau)\|_{L^2}d\tau\\
\geq & -(1+P(Q(t)))\int_0^t\|(p-1)(\tau)\|_m^2d\tau,
\end{align*}
where in the first inequality Lemma \ref{L1} is used;
and similarly for the fourth term,
\begin{align*}
&\sum\limits_{|\beta|\geq 1, \beta+\kappa=\alpha}C_{\alpha\beta}\int_0^t\int Z^\alpha (p-1) Z^\beta\left(\frac{{\bf u}}{\gamma p}\right)Z^\kappa \nabla (p-1)\\
\geq &-C\|\frac{{\bf u}}{\gamma p}\|_{1,\infty}\int_0^t\|\nabla p(\tau)\|_{m-1}\|Z^\alpha (p-1)(\tau)\|_{L^2}d\tau\\
&-C\|p\|_{1,\infty}\int_0^t\|\frac{{\bf u}}{p}(\tau)\|_{m}\|Z^\alpha (p-1)(\tau)\|_{L^2}d\tau\\
\geq  & -\delta \int_0^t\|\nabla p\|_{m-1}^2d\tau-C_\delta(1+P(Q(t)))\int_0^t(\|{\bf u}(\tau)\|_m^2+\|(p-1)(\tau)\|^2_m)d\tau.
\end{align*}
Next, we  estimate the terms involving $\mathcal{C}_i^\alpha\ (i=1,...,5)$ in \eqref{3.4} as follows. First,  we have the following estimates,
\begin{align*}
\int_0^t\|\mathcal{C}_1^\alpha\|^2d\tau\lesssim& \sum\limits_{|\beta|\geq 1, \beta+\kappa=\alpha}\int_0^t\|Z^\beta(1+\rho) Z^\kappa {\bf u}_t\|^2d\tau\\
\lesssim &\|\rho\|^2_{1,\infty}\int_0^t\|{\bf u}_t\|_{m-1}^2d\tau+\|{\bf u}_t\|^2_{L^\infty}\int_0^t\|\rho\|_{m}^2d\tau\\
\lesssim &(1+P(Q(t)))\int_0^t(\|\rho\|_m^2+\|{\bf u})\|_m^2)d\tau,
\end{align*}
and
\begin{align*}
\int_0^t\|\mathcal{C}_2^\alpha\|^2d\tau\lesssim& \sum\limits_{|\beta|\geq 1, \beta+\kappa=\alpha}\int_0^t\|Z^\beta((1+\rho){\bf u}) Z^\kappa \nabla {\bf u}\|^2d\tau+\|(1+\rho){\bf u}\|^2_{L^\infty}\int_0^t\|\nabla {\bf u}\|_{m-1}^2d\tau\\
\lesssim &(1+P(Q(t)))\int_0^t(\|\nabla {\bf u}(\tau)\|_{m-1}^2+\|\rho(\tau)\|_m^2+\| {\bf u}(\tau)\|_m^2)d\tau.
\end{align*}
Similarly, one has,
\begin{align*}
\int_0^t\|\mathcal{C}_3^\alpha\|^2d\tau \lesssim &(1+P(Q(t)))\int_0^t(\|\nabla {\bf G_1}(\tau)\|_{m-1}^2+\| \nabla {\bf G_2}(\tau)\|_{m-1}^2)d\tau\\
& +(1+P(Q(t)))\int_0^t(\|\rho(\tau)\|_m^2+\|{\bf G_1}(\tau)\|_m^2+ \|{\bf G_2}(\tau)\|_m^2)d\tau,
\end{align*}
and
\begin{align*}
&\int_0^t(\|\mathcal{C}_4^\alpha\|^2+\|\mathcal{C}_5^\alpha\|^2)d\tau\\
\lesssim &(1+P(Q(t)))\int_0^t(\|\nabla {\bf u}(\tau)\|_{m-1}^2+\|\nabla {\bf G_1}(\tau)\|_{m-1}^2+\|\nabla {\bf G_2}(\tau)\|_{m-1}^2)d\tau\\
&+(1+P(Q(t)))\int_0^t(\|{\bf u}(\tau)\|_m^2+\| {\bf G_1}(\tau)\|_m^2+\| {\bf G_2}(\tau)\|_m^2)d\tau.
\end{align*}

Substituting  all of the above estimates into \eqref{3.4}, we obtain
\begin{align}
\label{aaa}
&\int\frac12(1+\rho)(|Z^\alpha{\bf u}|^2+|Z^\alpha{\bf G_1}|^2+|Z^\alpha{\bf G_2}|^2)d{\bf x}\nonumber\\
&\qquad +\int \frac{|Z^\alpha (p-1)(t)|^2}{2\gamma p(t)}d{\bf x}+c_1\varepsilon\int_0^t\|\nabla Z^\alpha{\bf u}\|_{L^2}^2d\tau\nonumber\\
\lesssim&\int\frac12(1+\rho_0)(|Z^\alpha{\bf u}|^2(0)+|Z^\alpha{\bf G_1}|^2(0)+|Z^\alpha{\bf G_2}(0)|^2)d{\bf x}+\int \frac{|Z^\alpha (p-1)(0)|^2}{2\gamma p(0)}d{\bf x}\nonumber\\
&+\delta \int_0^t\|\nabla p(\tau)\|_{m-1}^2d\tau+C_\delta(1+P(Q(t)))\int_0^t\|(p-1)(\tau)\|^2_md\tau\\
&+\delta\varepsilon^2\int_0^t\|\nabla^2{\bf u}(\tau)\|_{m-1}^2d\tau+\delta\varepsilon\int_0^t\|\nabla{\bf u}(\tau)\|_{m}^2d\tau
\nonumber\\
&+(1+P(Q(t)))\int_0^t(\|\nabla {\bf u}(\tau)\|_{m-1}^2+\|\nabla {\bf G_1}(\tau)\|_{m-1}^2+\|\nabla {\bf G_2}(\tau)\|_{m-1}^2d\tau\nonumber\\
&+(1+P(Q(t)))\int_0^t(\|\rho(\tau)\|_m^2+\|{\bf u}(\tau)\|_m^2+\| {\bf G_1}(\tau)\|_m^2+\| {\bf G_2}(\tau)\|_m^2)d\tau\nonumber.
\end{align}
Summing (\ref{aaa}) over $|\alpha|\leq m$, choosing $\delta$ suitably small and using mathematical induction arguments, we have
\begin{align}
\label{bbb}
&\|(p-1, {\bf u}, {\bf G_1}, {\bf G_2})(t)\|_m^2+\varepsilon\int_0^t\|\nabla{\bf u}(\tau)\|_{m}^2d\tau\nonumber\\
\lesssim&\|(p-1, {\bf u}, {\bf G_1}, {\bf G_2})(0)\|_m^2+\delta \int_0^t\|\nabla p(\tau)\|_{m-1}^2d\tau+\delta\varepsilon^2\int_0^t\|\nabla^2{\bf u}(\tau)\|_{m-1}^2d\tau\\
&+(1+P(Q(t)))\int_0^t(\|\nabla {\bf u}(\tau)\|_{m-1}^2+\|\nabla {\bf G_1}(\tau)\|_{m-1}^2+\|\nabla {\bf G_2}(\tau)\|_{m-1}^2d\tau\nonumber\\
&+(1+P(Q(t)))\int_0^t(\|\rho(\tau)\|_m^2+\|{\bf u}(\tau)\|_m^2+\| {\bf G_1}(\tau)\|_m^2+\| {\bf G_2}(\tau)\|_m^2)d\tau\nonumber,
\end{align}
where the following fact of equivalence is used:
\begin{align}
\label{PPP}
C^{-1}\|\rho\|_m^2\leq \|p-1\|_m^2 \leq C\|\rho\|_m^2
\end{align}
holds for some generic constant $C>1$, due to (\ref{P}) and the {\it a priori} assumption that $\|\rho\|_{L^\infty}\leq 1/2$.
Therefore, the proof of Proposition \ref{P3.1} is completed.
\end{proof}  

 To close the energy estimates, it  suffices to derive the estimates of $Q(t),\ \|\nabla({\bf u}, {\bf G_1}, {\bf G_2})\|_{m-1}$ and $\|\nabla p\|_{m-1}$, which is the main task  in the next  section.

\bigskip

\section{Estimates of Normal Derivatives}

To estimate $\|\nabla({\bf u}, {\bf G_1}, {\bf G_2}, p)\|_{m-1}$, it suffices to estimate $\|\partial_y({\bf u}, {\bf G_1}, {\bf G_2}, p)\|_{m-1}$,  since
$\|\partial_x({\bf u}, {\bf G_1}, {\bf G_2}, p)\|_{m-1}\leq \|({\bf u}, {\bf G_1}, {\bf G_2}, p-1)\|_{m}$ as $\partial_x=Z_1$,.
In this  section, we focus on the estimates of the normal derivatives for $({\bf u}, {\bf G_1}, {\bf G_2})$ and $p$. We will derive the conormal estimates for both the first and second order normal derivatives of each component for $({\bf u}, {\bf G_1}, {\bf G_2})$ and $p$ in the subsequent subsections.

\begin{pro}
\label{P4.1}
Under the assumptions in Theorem \ref{Thm1}, there exists a sufficiently small $\varepsilon_0>0$, such that for any $0<\varepsilon<\varepsilon_0$, the smooth solution $(\rho, {\bf u}, {\bf G_1}, {\bf G_2})$ to (\ref{3.1})-(\ref{3.2a}) satisfies the following {\it a priori} estimate:
\begin{align}
\label{4.3a}
&\|(p-1, {\bf u}, {\bf G_1}, {\bf G_2})(t)\|_m^2+\varepsilon(\|\partial_yf_2(t)\|_{m-1}^2+\|\partial_y^2f_2(t)\|_{m-2}^2+\|\partial_yp(t)\|_{m-1}^2+\|\partial_y^2p(t)\|_{m-2}^2)\nonumber\\
&+\varepsilon\int_0^t\|\nabla{\bf u}(\tau)\|_{m}^2d\tau+\int_0^t(\|\partial_y p(\tau)\|_{m-1}^2+\|\partial_y {\bf u}(\tau)\|_{m-1}^2+\|\partial_y {\bf G_1}(\tau)\|_{m-1}^2+\|\partial_y {\bf G_2}(\tau)\|_{m-1}^2)d\tau\nonumber\\
&+\int_0^t(\|\partial^2_y p(\tau)\|_{m-2}^2+\|\partial^2_y {\bf u}(\tau)\|_{m-2}^2+\|\partial^2_y {\bf G_1}(\tau)\|_{m-2}^2+\|\partial^2_y {\bf G_2}(\tau)\|_{m-2}^2)d\tau\nonumber\\
&+\varepsilon^2\int_0^t(\|\partial^2_y u(\tau)\|_{m-1}^2+\|\partial^3_y u(\tau)\|_{m-2}^2+\|\partial^2_y v(\tau)\|_{m-1}^2+\|\partial^3_y v(\tau)\|_{m-2}^2)d\tau\nonumber\\
\lesssim& \|(p-1, {\bf u}, {\bf G_1}, {\bf G_2})(0)\|_m^2+\  \varepsilon(\|\partial_yf_2(0)\|_{m-1}^2+\|\partial_y^2f_2(0)\|_{m-2}^2+\|\partial_yp(0)\|_{m-1}^2+\|\partial_y^2p(0)\|_{m-2}^2)\nonumber\\
&+(1+P(Q(t)))\int_0^t\big(P(\|\rho(\tau)\|_m)+\|{\bf u}(\tau)\|_m^2+P(\| {\bf G_1}(\tau)\|_m)+P(\| {\bf G_2}(\tau)\|_m)\big)d\tau,
\end{align}
where
\begin{align*}
Q(t)=\sup\limits_{0\leq \tau\leq t}\{\|(p-1, {\bf u}, {\bf G_1}, {\bf G_2})(\tau)\|_{1, \infty}+\|(\nabla p, \nabla {\bf u}, \nabla {\bf G_1}, \nabla {\bf G_2})(\tau)\|_{1, \infty}\},
\end{align*}
and $P(\cdot)$ is a polynomial.
\end{pro}
\subsection{Estimates of $\partial_y v$ and $\partial^2_y v$}
To estimate the normal derivatives of each component, it is convenient to rewrite the equations of $({\bf G_1}, {\bf G_2})$ in (\ref{3.1}) as
\begin{align}
\label{3.7}
\left\{
\begin{array}{ll}
\partial_tf_1+u\partial_xf_1+v\partial_yf_1=(1+f_1)\partial_xu+f_3\partial_yu,\\
\partial_tf_2+u\partial_xf_2+v\partial_yf_2=f_2\partial_xu+(1+f_4)\partial_yu,\\
\partial_tf_3+u\partial_xf_3+v\partial_yf_3=(1+f_1)\partial_xv+f_3\partial_yv,\\
\partial_tf_4+u\partial_xf_4+v\partial_yf_4=f_2\partial_xv+(1+f_4)\partial_yv.\\
\end{array}
\right.
\end{align}
From the fourth equation in (\ref{3.7}), we have
\begin{align}
\label{3.8}
\partial_yv=\frac{1}{1+f_4}(\partial_tf_4+u\partial_xf_4+v\partial_yf_4-f_2\partial_xv).
\end{align}

\subsubsection*{Step 1}
Applying the operator $Z^\alpha\ (|\alpha|\leq m-1)$ on (\ref{3.8}), we get
\begin{align*}
Z^\alpha\partial_yv=Z^\alpha\left\{\frac{1}{1+f_4}(\partial_tf_4+u\partial_xf_4+v\partial_yf_4-f_2\partial_xv)\right\}.
\end{align*}
Notice that
\begin{align}
\label{A1}
&\left\|Z^\alpha\left(\frac{v}{1+f_4}\partial_yf_4\right)\right\|=\left\|Z^\alpha\left(\frac{1}{1+f_4}\frac{v}{\varphi(y)}\varphi(y)\partial_yf_4\right)\right\|\nonumber\\
\lesssim &\left\|\frac{1}{1+f_4}\frac{v}{\varphi(y)}\right\|_{L^\infty}\|f_4\|_m+\|\varphi(y)\partial_yf_4\|_{L^\infty}\left\|\frac{1}{1+f_4}\frac{v}{\varphi(y)}\right\|_{m-1}\nonumber\\
\lesssim &\|\partial_yv\|_{L^\infty}\|f_4\|_m+\|f_4\|_{1, \infty}(\|\partial_yv\|_{m-1}+\|\partial_yv\|_{L^\infty}P(\|f_4\|_{m-1}))\nonumber\\
\lesssim &(1+P(Q(t)))P(\|f_4\|_m)+\|f_4\|_{1, \infty}\|\partial_yv\|_{m-1},
\end{align}
and 
\begin{align*}
&\left\|Z^\alpha\left(\frac{\partial_tf_4}{1+f_4}\right)\right\|+\left\|Z^\alpha\left(\frac{u}{1+f_4}\partial_xf_4\right)\right\|+\left\|Z^\alpha\left(\frac{f_2}{1+f_4}\partial_xv\right)\right\|\\
\lesssim&\left\|\frac{1}{1+f_4}\right\|_{L^\infty}\|f_4\|_m+\|\partial_tf_4\|_{L^\infty}\left\|Z^\alpha\left(\frac{1}{1+f_4}\right)\right\|
+\left\|\frac{u}{1+f_4}\right\|_{L^\infty}\|f_4\|_m\\
&+\|\partial_xf_4\|_{L^\infty}\left\|Z^\alpha\left(\frac{u}{1+f_4}\right)\right\|
+\left\|\frac{f_2}{1+f_4}\right\|_{L^\infty}\|v\|_m+\|\partial_xv\|_{L^\infty}\left\|Z^\alpha\left(\frac{f_2}{1+f_4}\right)\right\|\\
\lesssim&(1+P(Q(t)))(P(\|f_4\|_m)+\|v\|_m+\|f_2\|_{m-1}+\|u\|_{m-1}),
\end{align*}
where we used the {\it a priori} assumption of $\|f_4\|_{L^\infty}\leq 1/2$.
Summing all of above inequalities over $|\alpha|\leq m-1$ and using the {\it a priori} assumption of $\|f_4\|_{1,\infty}\leq C_0\sigma_0$ with $\sigma_0$ being in Theorem \ref{Thm1} and $C_0$ being a suitably large constant independent of $\sigma_0$ and $\varepsilon$ to be determined later,
we obtain the following estimate by choosing $\sigma_0$ sufficiently small once $C_0$ is fixed, 
\begin{align}
\label{3.9}
\|\partial_yv\|_{m-1}\lesssim (1+P(Q(t)))(P(\|f_4\|_m)+\|v\|_m+\|f_2\|_{m-1}+\|u\|_{m-1}).
\end{align}

\subsubsection*{Step 2}
To control $\|\nabla v\|_{1, \infty}$ in $Q(t)$, it is necessary to derive the conormal estimates of $\partial_y^2v$. Applying the operator $Z^\alpha\partial_y\ (|\alpha|\leq m-2)$ on the equation (\ref{3.8}) gives
\begin{align*}
Z^\alpha\partial^2_yv=Z^\alpha\partial_y\left\{\frac{1}{1+f_4}(\partial_tf_4+u\partial_xf_4+v\partial_yf_4-f_2\partial_xv)\right\}.
\end{align*}
Then
\begin{align*}
\|Z^\alpha\partial^2_yv\|\leq &\left\|Z^\alpha\partial_y\left(\frac{\partial_tf_4}{1+f_4}\right)\right\|+\left\|Z^\alpha\partial_y\left(\frac{u\partial_xf_4}{1+f_4}\right)\right\| \\
&+\left\|Z^\alpha\partial_y\left(\frac{v\partial_yf_4}{1+f_4}\right)\right\|+\left\|Z^\alpha\partial_y\left(\frac{f_2\partial_xv}{1+f_4}\right)\right\|.
\end{align*}
Now, we estimate each of the terms  on the right hand side as follows. Firstly,
\begin{align*}
&\left\|Z^\alpha\partial_y\left(\frac{\partial_tf_4}{1+f_4}\right)\right\|\leq \left\|Z^\alpha\left(\frac{1}{1+f_4}\partial_y\partial_tf_4\right)\right\|+\left\|Z^\alpha\left(\partial_tf_4\partial_y\left(\frac{1}{1+f_4}\right)\right)\right\|\\
\lesssim&\|\partial_yf_4\|_{1, \infty}P(\|f_4\|_{m-2})+\|\partial_yf_4\|_{m-1}\\
&+\|\partial_tf_4\|_{L^\infty}(\|\partial_yf_4\|_{m-2}+\|\partial_yf_4\|_{L^\infty}P(\|f_4\|_{m-2}))+\|\partial_yf_4\|_{L^\infty}\|f_4\|_{m-1}\\
\lesssim& (1+P(Q(t)))(P(\|f_4\|_{m-1})+\|\partial_yf_4\|_{m-1}),
\end{align*}
where the {\it a priori} estimate of $\|f_4\|_{L^\infty}\leq 1/2$ is used again; secondly,
\begin{align*}
&\left\|Z^\alpha\partial_y\left(\frac{u\partial_xf_4}{1+f_4}\right)\right\|\leq \left\|Z^\alpha\left(\frac{u}{1+f_4}\partial_y\partial_xf_4\right)\right\|+\left\|Z^\alpha\left(\partial_xf_4\partial_y\left(\frac{u}{1+f_4}\right)\right)\right\|\\
\lesssim&(1+P(Q(t)))\|\partial_yf_4\|_{1, \infty}(\|u\|_{m-2}+P(\|f_4\|_{m-2}))+\|u\|_{L^\infty}\|\partial_yf_4\|_{m-1}\\
&+(1+P(Q(t)))\|\partial_xf_4\|_{L^\infty}(\|\partial_y u\|_{m-2}+\|\partial_y  f_4\|_{m-2})+\|\partial_y(u, f_4)\|_{L^\infty}\|f_4\|_{m-1}\\
\lesssim& (1+P(Q(t)))(\|u\|_{m-2}+P(\|f_4\|_{m-1})+\|\partial_yf_4\|_{m-1}+\|\partial_yu\|_{m-2}).
\end{align*}
Similarly,
\begin{align*}
&\left\|Z^\alpha\partial_y\left(\frac{f_2\partial_xv}{1+f_4}\right)\right\|\leq \left\|Z^\alpha\left(\frac{f_2}{1+f_4}\partial_y\partial_xv\right)\right\|+\left\|Z^\alpha\left(\partial_xv\partial_y\left(\frac{f_2}{1+f_4}\right)\right)\right\|\\
\lesssim& (1+P(Q(t)))(P(\|f_4\|_{m-2})+\|f_2\|_{m-2}+\|v\|_{m-1}+\|\partial_yf_4\|_{m-2}+\|\partial_yf_2\|_{m-2}+\|\partial_yv\|_{m-1}).
\end{align*}
Moreover,
\begin{align*}
&\left\|Z^\alpha\partial_y\left(\frac{v\partial_yf_4}{1+f_4}\right)\right\|\leq \left\|Z^\alpha\left(\frac{v}{1+f_4}\partial_y\partial_yf_4\right)\right\|+\left\|Z^\alpha\left(\partial_yf_4\partial_y\left(\frac{v}{1+f_4}\right)\right)\right\|\\
=&\left\|Z^\alpha\left(\frac{1}{(1+f_4)}\frac{v}{\varphi(y)}\varphi(y)\partial_y\partial_yf_4\right)\right\|+\left\|Z^\alpha\left(\partial_yf_4\partial_y\left(\frac{v}{1+f_4}\right)\right)\right\|\\
\lesssim&(1+P(Q(t)))(\|\partial_yv\|_{m-2}+P(\|f_4\|_{m-2})+\|\partial_yf_4\|_{m-1})\\
&+(1+P(Q(t)))(\|\partial_y v\|_{m-2}+\|\partial_y f_4\|_{m-2}+\|v\|_{m-2}+P(\|f_4\|_{m-2}))\\
\lesssim& (1+P(Q(t)))(P(\|f_4\|_{m-2})+\|v\|_{m-2}+\|\partial_yf_4\|_{m-1}+\|\partial_yv\|_{m-2}).
\end{align*}
Consequently, summing all of above inequalities over $|\alpha|\leq m-2$ yields that
\begin{align}
\label{3.10}
\|\partial^2_yv\|_{m-2}\lesssim&(1+P(Q(t)))(P(\|f_4\|_{m-1})+\|f_2\|_{m-2}+\|u\|_{m-2}+\|v\|_{m-1})\nonumber\\
&+(1+P(Q(t)))(\|\partial_yf_4\|_{m-1}+\|\partial_yf_2\|_{m-2}+\|\partial_yv\|_{m-1}+\|\partial_yu\|_{m-2})\nonumber\\
\lesssim&(1+P(Q(t)))(P(\|f_4\|_{m})+\|f_2\|_{m-1}+\|u\|_{m-1}+\|v\|_{m})\nonumber\\
&+(1+P(Q(t)))(\|\partial_yf_4\|_{m-1}+\|\partial_yf_2\|_{m-2}+\|\partial_yu\|_{m-2}),
\end{align}
where (\ref{3.9}) is used in the second inequality.
\subsection{Estimates of $\partial_y u$ and $\partial^2_y u$}
From the second equation in (\ref{3.7}), we have
\begin{align}
\label{3.11}
\partial_yu=\frac{1}{1+f_4}\{\partial_tf_2+u\partial_xf_2+v\partial_yf_2-f_2\partial_xu\}.
\end{align}
A similar argument to (\ref{3.9}) yields that
\begin{align*}
\|\partial_yu\|_{m-1}\lesssim (1+P(Q(t)))(\|(u, f_2)\|_m+P(\| f_4\|_{m-1})+\|\partial_yv\|_{m-1}).
\end{align*}
Then, by using (\ref{3.9}), we have
\begin{align}
\label{3.12}
\|\partial_yu\|_{m-1}\lesssim (1+P(Q(t)))(\|(u, v, f_2)\|_m+P(\|f_4\|_m)).
\end{align}
Applying the operator $Z^\alpha\partial_y\ (|\alpha|\leq m-2)$ on (\ref{3.11}) gives
\begin{align}
\label{3.13}
Z^\alpha\partial^2_yu=Z^\alpha\partial_y\left\{\frac{1}{1+f_4}(\partial_tf_2+u\partial_xf_2+v\partial_yf_2-f_2\partial_xu)\right\}.
\end{align}
Similar arguments to  (\ref{3.10}) give
\begin{align}
\label{3.14}
&\|\partial^2_yu\|_{m-2}\nonumber\\
\lesssim& (1+P(Q(t)))(\|v\|_{m-2}+P(\|f_4\|_{m-2})+\|(u, f_2)\|_{m-1}+\|\partial_y(v, f_4)\|_{m-2}+\|\partial_y(u, f_2)\|_{m-1})\nonumber\\
\lesssim& (1+P(Q(t)))(\|(u, v, f_2)\|_{m}+P(\|f_4\|_{m})+\|\partial_yf_2\|_{m-1}+\|\partial_yf_4\|_{m-2}),
\end{align}
where in the second inequality both (\ref{3.9}) and (\ref{3.12}) are used.

\subsection{Estimate of $\partial_y f_2$}
It is convenient to rewrite the momentum equations in (\ref{3.1}) as  the following:
\begin{equation}\label{3.15}
\begin{cases}
(1+\rho)\partial_t u+(1+\rho) u\partial_xu+(1+\rho) v\partial_yu-(1+\rho)(1+f_1)\partial_xf_1-(1+\rho) f_3\partial_yf_1\\
\quad-(1+\rho) f_2\partial_xf_2 -(1+\rho) (1+f_4)\partial_yf_2-\mu\varepsilon\partial_y^2u-\mu\varepsilon\partial_x^2u\\
\quad -(\mu+\lambda)\varepsilon\partial_x(u_x+v_y)+\partial_xp=0,\\
(1+\rho)\partial_t v+(1+\rho) u\partial_xv+(1+\rho) v\partial_yv-(1+\rho)(1+f_1)\partial_xf_3-(1+\rho) f_3\partial_yf_3\\
\quad-(1+\rho) f_2\partial_xf_4-(1+\rho) (1+f_4)\partial_yf_4-\mu\varepsilon\partial_y^2v-\mu\varepsilon\partial_x^2v\\
\quad -(\mu+\lambda)\varepsilon\partial_y(u_x+v_y)+\partial_yp=0.
\end{cases}
\end{equation}

\subsubsection*{Step 1}
According to the first equation in (\ref{3.15}), we have
\begin{align}
\label{3.16}
&(1+\rho)(1+f_4)\partial_yf_2+\mu\varepsilon\partial^2_yu \nonumber \\
=&(1+\rho) \partial_tu+(1+\rho) u\partial_xu+(1+\rho) v\partial_yu-(1+\rho)(1+f_1)\partial_xf_1-(1+\rho)f_3\partial_yf_1 \\
&-(1+\rho) f_2\partial_xf_2-\mu\varepsilon\partial_x^2u-(\mu+\lambda)\varepsilon\partial_x(u_x+v_y)+\partial_xp. \nonumber
\end{align}
Applying the operator $Z^\alpha\ (|\alpha|\leq m-1)$ on the both sides of (\ref{3.16}), one has,
\begin{align}
\label{3.17}
&(1+\rho)(1+f_4)Z^\alpha\partial_yf_2+\varepsilon\mu Z^\alpha\partial^2_yu \nonumber\\
=&Z^\alpha\{(1+\rho)\partial_t  u+(1+\rho) u\partial_xu+(1+\rho) v\partial_yu\}\nonumber\\
&+Z^\alpha\{-(1+\rho)(1+f_1)\partial_xf_1-(1+\rho) f_3\partial_yf_1-(1+\rho) f_2\partial_xf_2\}\nonumber\\
&+Z^\alpha\{-\mu\varepsilon\partial_x^2u-(\mu+\lambda)\varepsilon\partial_x(u_x+v_y)+\partial_xp\}-[Z^\alpha, (1+\rho)(1+f_4)]\partial_yf_2.
\end{align}
Taking the $L^2$ inner product over $\mathbb{R}^2_+$ on the both sides of the above equality yields that
\begin{align}
\label{A3}
&\|(1+\rho)(1+f_4)Z^\alpha\partial_yf_2\|^2+\mu^2\varepsilon^2\|Z^\alpha\partial_y^2 u\|^2+2\mu\varepsilon \int(1+\rho)(1+f_4) Z^\alpha\partial_yf_2\cdot Z^\alpha\partial_y^2 u d{\bf x}\nonumber\\
\lesssim&\|(1+\rho) \partial_tu\|^2_{m-1}+\|(1+\rho)u\partial_x u\|^2_{m-1}+\|(1+\rho) v \partial_y u\|^2_{m-1}\nonumber\\
&+\|(1+\rho)(1+f_1)\partial_xf_1\|^2_{m-1}+\|(1+\rho) f_3\partial_yf_1\|^2_{m-1}+\|(1+\rho) f_2\partial_xf_2\|^2_{m-1}\nonumber \\
&+\varepsilon^2\|\partial_x u\|^2_{m}+\varepsilon^2\|\partial_yv\|^2_{m}+\|\partial_xp\|^2_{m-1}\\
&+\|Z((1+\rho)(1+f_4))\|^2_{L^\infty}\|\partial_yf_2\|^2_{m-2}+\|\partial_yf_2\|^2_{L^\infty}\|Z((1+\rho)(1+f_4))\|^2_{m-2} \nonumber \\
\lesssim& (1+P(Q(t)))\left(\|( u, f_1, f_2)\|_{m}^2+\|(\rho, v, f_3, f_4)\|_{m-1}^2+\|\partial_y(u, f_1)\|_{m-1}^2+\|\partial_yf_2\|_{m-2}^2\right)\nonumber\\
&+\|\partial_xp\|_{m-1}^2+\varepsilon^2\|(\partial_x u, \partial_yv)\|_m^2\nonumber\\
\lesssim& (1+P(Q(t)))\left(\|( u, v, f_1, f_2)\|_{m}^2+\|(\rho, f_3)\|_{m-1}^2+P(\|f_4\|_m)+\|\partial_y f_1\|_{m-1}^2+\|\partial_yf_2\|_{m-2}^2\right)\nonumber\\
&+\|\partial_xp\|_{m-1}^2+\varepsilon^2\|(\partial_x u, \partial_yv)\|_m^2, \nonumber
\end{align}
where (\ref{3.12}) is used in the last inequality.

\subsubsection*{Step 2}
It remains to handle the mixed term  $2\mu\varepsilon \int(1+\rho)(1+f_4) Z^\alpha\partial_yf_2\cdot Z^\alpha\partial_y^2 u d{\bf x}$ on the left hand side  of (\ref{A3}).
From the second equation in (\ref{3.7}), we have
\begin{align}
\label{3.19}
\frac{1}{(1+f_4)}\partial_tf_2-\partial_yu=\frac{1}{(1+f_4)}(f_2\partial_xu-u\partial_xf_2-v\partial_yf_2).
\end{align}
Applying the operator $Z^\alpha\partial_y\ (|\alpha|\leq m-1)$ on the equation (\ref{3.19}) leads to
\begin{align}
\label{3.20}
&\frac{1}{(1+f_4)}\partial_tZ^\alpha \partial_yf_2-Z^\alpha \partial^2_yu\nonumber\\
=&Z^\alpha \partial_y\left\{\frac{f_2}{1+f_4}\partial_xu\right\}-\frac{u}{(1+f_4)}\partial_xZ^\alpha \partial_y f_2-\frac{v}{(1+f_4)}\partial_yZ^\alpha \partial_yf_2+\mathcal{C}_6^\alpha,
\end{align}
where
\begin{align} \label{C6}
\mathcal{C}_6^\alpha=-\left[Z^\alpha \partial_y, \frac{1}{(1+f_4)}\partial_t\right]f_2-\left[Z^\alpha \partial_y, \frac{u}{(1+f_4)}\partial_x\right]f_2-\left[Z^\alpha \partial_y, \frac{v}{(1+f_4)}\partial_y\right]f_2.
\end{align}
Multiplying  (\ref{3.20}) by $2\mu\varepsilon (1+\rho)(1+f_4) Z^\alpha\partial_yf_2$ and integrating the resulting equation over $\mathbb{R}^2_+$ give that
\begin{align} \label{418}
&\mu\varepsilon\frac{d}{dt}\|\sqrt{(1+\rho)}Z^\alpha\partial_y f_2\|^2-2\mu\varepsilon\int (1+\rho)(1+f_4)Z^\alpha\partial_yf_2 Z^\alpha\partial_y^2ud{\bf x}\\
=&2\mu\varepsilon\int (1+\rho)(1+f_4) Z^\alpha\partial_yf_2 Z^\alpha\partial_y\left(\frac{f_2}{1+f_4}\partial_xu\right)d{\bf x}+2\mu\varepsilon\int (1+\rho)(1+f_4) Z^\alpha\partial_yf_2 \mathcal{C}_6^\alpha d{\bf x},\nonumber
\end{align}
where the equation of $\partial_t\rho+\partial_x((1+\rho)u)+\partial_y((1+\rho) v)=0$ is used.
For the terms on the right hand side of \eqref{418}, by the Cauchy-Schwarz inequality, we have
\begin{align*}
&\left|2\mu\varepsilon\int (1+\rho)(1+f_4) Z^\alpha\partial_yf_2 Z^\alpha\partial_y\left(\frac{f_2}{1+f_4}\partial_xu\right)d{\bf x}\right|\\
&+\left|2\mu\varepsilon\int (1+\rho)(1+f_4) Z^\alpha\partial_yf_2 \mathcal{C}_6^\alpha d{\bf x}\right|\\
=&\left|2\mu\varepsilon\int (1+\rho)(1+f_4)  Z^\alpha\partial_yf_2 Z^\alpha\left(\partial_y\left(\frac{f_2}{1+f_4}\right)\partial_xu+\frac{f_2}{1+f_4}\partial_x\partial_yu\right)d{\bf x}\right|\\
&+\left|2\mu\varepsilon\int (1+\rho)(1+f_4)  Z^\alpha\partial_yf_2 \mathcal{C}_6^\alpha d{\bf x}\right|\\
\leq & \delta\|(1+\rho)(1+f_4)Z^\alpha\partial_yf_2\|^2\\
&+C_\delta\varepsilon^2\times\left(\left\|Z^\alpha\left(\partial_y\left(\frac{f_2}{1+f_4}\right)\partial_xu+\frac{f_2}{1+f_4}\partial_x\partial_yu\right)\right\|^2
+\|\mathcal{C}_6^\alpha\|^2\right),
\end{align*}
for some small constant $\delta>0$ to be determined later.
 Note that
\begin{align*}
&\left\|Z^\alpha\left(\partial_y\left(\frac{f_2}{1+f_4}\right)\partial_xu+\frac{f_2}{1+f_4}\partial_x\partial_yu\right)\right\|^2\\
\lesssim &(1+P(Q(t)))\|u\|_m^2+\|\partial_xu\|_{L^\infty}^2\left\|\partial_y\left(\frac{f_2}{1+f_4} \right)\right\|_{m-1}^2\\
&+(1+P(Q(t)))\|\partial_yu\|_m^2+\|\partial_yu\|_{1, \infty}^2\left\|\frac{f_2}{1+f_4}\right\|^2_{m-1}\\
\lesssim &(1+P(Q(t)))(\|u\|_m^2+\|\partial_yu\|_m^2+\|\partial_y(f_2, f_4)\|_{m-1}^2+\|f_2\|_{m-1}^2+P(\|f_4\|_{m-1})).
\end{align*}

\subsubsection*{Step 3}
For the second term on the right hand side of \eqref{418}, we need to  estimate the commutator $\mathcal{C}_6^\alpha$ defined in \eqref{C6}. First, we have
\begin{align*}
&\left\|\left[Z^\alpha \partial_y, \frac{1}{(1+f_4)}\partial_t\right]f_2\right\|^2\\
=&\left\|Z^\alpha\left(\partial_y\left(\frac{1}{1+f_4}\right)\partial_tf_2\right)+\left[Z^\alpha, \frac{1}{1+f_4}\right]\partial_t\partial_y f_2\right\|^2\\
\lesssim &(1+P(Q(t)))\|f_2\|_m^2+\|\partial_tf_2\|_{L^\infty}^2\left\|\partial_y\left(\frac{1}{1+f_4}\right)\right\|_{m-1}^2\\
&+(1+P(Q(t)))\|\partial_yf_2\|_{m-1}^2+\|\partial_yf_2\|_{1, \infty}^2\left\|Z\left(\frac{1}{1+f_4}\right)\right\|_{m-2}^2\\
\lesssim &(1+P(Q(t)))(\|f_2\|_m^2+\|\partial_y(f_2, f_4)\|_{m-1}^2+P(\| f_4\|_{m-1})),
\end{align*}
and similarly,
\begin{align*}
&\left\|\left[Z^\alpha \partial_y, \frac{u}{(1+f_4)}\partial_x\right]f_2\right\|^2\\
=&\left\|Z^\alpha\left(\partial_y\left(\frac{u}{1+f_4}\right)\partial_xf_2\right)+\left[Z^\alpha, \frac{u}{(1+f_4)}\right]\partial_x\partial_y f_2\right\|^2\\
\lesssim &(1+P(Q(t)))\|f_2\|_m^2+\|\partial_xf_2\|_{L^\infty}^2\left\|\partial_y\left(\frac{u}{1+f_4}\right)\right\|_{m-1}^2\\
&+(1+P(Q(t)))\|\partial_yf_2\|_{m-1}^2+\|\partial_yf_2\|_{1, \infty}^2\left\|Z\left( \frac{u}{1+f_4}\right)\right\|_{m-2}^2\\
\lesssim &(1+P(Q(t)))(\|f_2\|_m^2+\|u\|_{m-1}^2+P(\|f_4\|_{m-1})+\|\partial_y(u, f_2, f_4)\|_{m-1}^2).
\end{align*}
Next we  notice that
\begin{align}
\label{TT}
&\left[Z^\alpha \partial_y, \frac{v}{(1+f_4)}\partial_y\right]f_2\nonumber\\
=&Z^\alpha \left\{\partial_y\left(\frac{v}{1+f_4}\right)\partial_yf_2\right\}+\left[Z^\alpha, \frac{v}{(1+f_4)}\right]\partial_y\partial_yf_2+\frac{v}{1+f_4}[Z^\alpha, \partial_y]\partial_yf_2
\end{align}
with
\begin{align*}
\left[Z^\alpha, \frac{v}{(1+f_4)}\right]\partial_y\partial_yf_2=\sum\limits_{|\beta|\geq 1, \beta+\kappa=\alpha}C_{\alpha\beta}Z^\beta\left(\frac{v}{1+f_4}\right)Z^\kappa\partial_y\partial_yf_2.
\end{align*}
The first term in (\ref{TT}) can be estimated as the following:
\begin{align*}
&\left\|Z^\alpha \left\{\partial_y\left(\frac{v}{1+f_4}\right)\partial_yf_2\right\}\right\|^2\\
\lesssim& (1+P(Q(t)))(\|\partial_yv\|^2_{m-1}+\|\partial_yf_2\|^2_{m-1}+\|\partial_yf_4\|^2_{m-1}+P(\|f_4\|_{m-1})+\|v\|^2_{m-1}).
\end{align*}

\subsubsection*{Step 4}
We now estimate the second and the third terms, that is,  the  two commutators in (\ref{TT}). For the first commutator, we have the following computation,
\begin{align*}
&Z^\beta\left(\frac{v}{1+f_4}\right)Z^\kappa\partial_y\partial_yf_2\\
=&\left(Z^\beta\left(\frac{v}{(1+f_4)\varphi(y)}\right)+\left[Z^\beta, \frac{1}{\varphi(y)}\right]\frac{v}{1+f_4}\right)\left(Z^\kappa\varphi(y)\partial_y\partial_yf_2+[Z^\kappa, \varphi(y)]\partial_y\partial_yf_2\right),
\end{align*}
where
\begin{align*}
\left[Z^\beta, \frac{1}{\varphi(y)}\right]\frac{v}{1+f_4}=\sum\limits_{\eta=0}^{\beta-1}\psi_{\eta,\beta}(y)Z^\eta\left(\frac{v}{\varphi(y)(1+f_4)}\right)
\end{align*}
for  some  bounded smooth functions $\psi_{\eta,\beta}(y)$ due to Lemma \ref{L6}, and similarly,
\begin{align*}
[Z^\kappa, \varphi(y)]\partial_y\partial_yf_2=\sum\limits_{\theta=0}^{\kappa-1}\psi_{\theta,\kappa}(y)Z^\theta(\varphi(y)\partial_y\partial_yf_2)
\end{align*}
for some  bounded smooth functions  $\psi_{\theta,\kappa}(y)$.
Then, according to Lemma \ref{L1}, we have  the following estimate:
\begin{align*}
&\left\| \left[Z^\alpha, \frac{v}{(1+f_4)}\right]\partial_y\partial_yf_2\right\|^2\\
\lesssim& \sum\limits_{|\beta|\geq 1, \beta+\kappa=\alpha}\left\|Z^\beta\left(\frac{v}{1+f_4}\right)Z^\kappa\partial_y\partial_yf_2\right\|^2\\
\lesssim&(1+P(Q(t)))(\|\partial_yv\|^2_{m-1}+\|\partial_yf_2\|^2_{m-1}+P(\|f_4\|_{m-1})).
\end{align*}
For the second commutator in (\ref{TT}), we write
\begin{align*}
[Z^\alpha, \partial_y]\partial_yf_2=\sum\limits_{\theta=0}^{m-2}\phi_{\theta,\alpha}(y)\partial_yZ^\theta\partial_yf_2
\end{align*}
with $\phi_{\theta,\alpha}(y)$ being bounded smooth functions due to Lemma \ref{L3}.
Then,
\begin{align*}
&\left\|\frac{v}{1+f_4}[Z^\alpha, \partial_y]\partial_yf_2\right\|^2\\
=&\left\|\frac{v}{(1+f_4)\varphi(y)}\varphi(y)[Z^\alpha, \partial_y]\partial_yf_2\right\|^2\\
\lesssim&(1+P(Q(t)))\|\partial_yf_2\|^2_{m-1}.
\end{align*}
Consequently,
\begin{align*}
&\left\|\left[Z^\alpha \partial_y, \frac{v}{(1+f_4)}\partial_y\right]f_2\right\|^2\\
\lesssim &(1+P(Q(t)))(\|\partial_yv\|^2_{m-1}+\|\partial_yf_2\|^2_{m-1}+\|\partial_yf_4\|^2_{m-1}+P(\|f_4\|_{m-1})+\|v\|^2_{m-1}).
\end{align*}
Substituting all of the above estimates into \eqref{418} yields the following estimate:
\begin{align}
\label{3.22}
&\mu\varepsilon\frac{d}{dt}\|\sqrt{(1+\rho)}Z^\alpha\partial_y f_2\|^2-2\mu\varepsilon\int (1+\rho)(1+f_4)Z^\alpha\partial_yf_2 Z^\alpha\partial_y^2ud{\bf x}\nonumber\\
\leq&\delta\|(1+\rho)(1+f_4)Z^\alpha\partial_yf_2\|^2\\
&+C_\delta\varepsilon^2(1+P(Q(t)))\Big(\|u\|_m^2+\|f_2\|_m^2+P(\|f_4\|_{m-1})+\|v\|_{m-1}^2 \nonumber \\
&\qquad\qquad\qquad\qquad\qquad +\|\partial_yu\|_m^2+\|\partial_yv\|_{m-1}^2+\|\partial_yf_2\|_{m-1}^2+\|\partial_yf_4\|_{m-1}^2\Big).\nonumber
\end{align}

\subsubsection*{Step 5}
Combining (\ref{A3}) and (\ref{3.22}) together and choosing $\delta$ suitably small, we obtain
\begin{align}
\label{A4}
&\mu\varepsilon\frac{d}{dt}\|\sqrt{(1+\rho)}Z^\alpha\partial_y f_2\|^2+\|(1+\rho)(1+f_4)Z^\alpha\partial_yf_2\|^2+\mu^2\varepsilon^2\|Z^\alpha\partial_y^2 u\|^2\nonumber\\
\lesssim& (1+P(Q(t)))\left(\|( u, v, f_1, f_2)\|_{m}^2+\|(\rho, f_3)\|_{m-1}^2+\|\partial_y f_1\|_{m-1}^2+\|\partial_yf_2\|_{m-2}^2+P(\|f_4\|_m)\right)\nonumber\\
&+\|\partial_xp\|_{m-1}^2+\varepsilon^2\|(\partial_x u, \partial_yv)\|_m^2\\
&+C\varepsilon^2(1+P(Q(t)))\left(\|\partial_yu\|_m^2+\|\partial_yv\|_{m-1}^2+\|\partial_yf_2\|_{m-1}^2+\|\partial_yf_4\|_{m-1}^2\right).\nonumber
\end{align}
Choosing $\varepsilon_0$ to be sufficiently small and for  $0<\varepsilon<\varepsilon_0$ summing the above inequalities over $|\alpha|\leq m-1$  lead  to
\begin{align}
\label{A5}
&\mu\varepsilon\frac{d}{dt}\|\partial_y f_2\|_{m-1}^2+\|\partial_yf_2\|_{m-1}^2+\varepsilon^2\|\partial_y^2 u\|_{m-1}^2\nonumber\\
\lesssim& (1+P(Q(t)))\left(\|( u, v, f_1, f_2)\|_{m}^2+\|(\rho, f_3)\|_{m-1}^2+\|\partial_yf_1\|_{m-1}^2+P(\|f_4\|_m)\right)\nonumber\\
&+\|\partial_xp\|_{m-1}^2+\varepsilon^2\|(\partial_x u, \partial_yv)\|_m^2\\
&+C\varepsilon^2(1+P(Q(t)))\left(\|\partial_yu\|_m^2+\|\partial_yf_4\|_{m-1}^2\right), \nonumber
\end{align}
where the mathematical induction arguments and the following {\it a priori} estimates are used:
\begin{align*}
\|\rho\|_{L^\infty}\leq 1/2,\qquad \|f_4\|_{L^\infty}\leq 1/2,\quad Q(t)\leq C;
\end{align*}
more precisely, notice that the order of conormal derivatives is up to $m-1$ on the left hand side  of (\ref{A4}), and there exist terms of $\|\partial_yf_2\|_{m-2}^2$ and $\varepsilon^2\|\partial_yf_2\|_{m-1}^2$ on the right hand side of (\ref{A4}), then the first term is absorbed by using the mathematical induction arguments, and the second term is absorbed by choosing $\varepsilon$ sufficiently small and the {\it a priori} assumption of $Q(t)\leq C$. And (\ref{3.9}) is also used in deriving (\ref{A5}).

\subsection{Estimate of $\partial_y^2f_2$}
Next, we will derive the conormal energy estimates of $\partial_y^2f_2$.


\subsubsection*{Step 1}
Applying the operator $Z^\alpha\partial_y\ (|\alpha|\leq m-2)$ on the both sides of (\ref{3.16}) yields
\begin{align}
\label{3.23}
&(1+\rho)(1+f_4)Z^\alpha\partial^2_yf_2+\varepsilon\mu Z^\alpha\partial^3_yu\nonumber \\
&\quad=Z^\alpha\partial_y\{(1+\rho)\partial_t u+(1+\rho) u\partial_xu+(1+\rho) v\partial_yu\}\nonumber\\
&\qquad+Z^\alpha\partial_y\{-(1+\rho)(1+f_1)\partial_xf_1-(1+\rho) f_3\partial_yf_1-(1+\rho) f_2\partial_xf_2\}\\
&\qquad+Z^\alpha\partial_y\{-\mu\varepsilon\partial_x^2u-(\mu+\lambda)\varepsilon\partial_x(u_x+v_y)+\partial_xp\}-[Z^\alpha\partial_y, (1+\rho)(1+f_4)]\partial_yf_2. \nonumber
\end{align}
Taking the $L^2$ inner product on the both sides of the above equality, we obtain
\begin{align}
\label{3.24}
&\|(1+\rho)(1+f_4)Z^\alpha\partial^2_yf_2\|^2+\mu^2\varepsilon^2\|Z^\alpha\partial_y^3 u\|^2+2\mu\varepsilon \int(1+\rho)(1+f_4) Z^\alpha\partial^2_yf_2\cdot Z^\alpha\partial_y^3 u d{\bf x}\nonumber\\
\lesssim&\|Z^\alpha\partial_y ((1+\rho)\partial_t u)\|^2+\|Z^\alpha \partial_y((1+\rho) u\partial_x u)\|^2+\|Z^\alpha \partial_y((1+\rho) v \partial_y u)\|^2 \nonumber \\
&+\|Z^\alpha \partial_y((1+\rho)(1+f_1)\partial_xf_1)\|^2+\|Z^\alpha \partial_y((1+\rho) f_3\partial_yf_1)\|^2 \\
&+\|Z^\alpha \partial_y((1+\rho) f_2\partial_xf_2)\|^2 +\varepsilon^2\|Z^\alpha \partial_y\partial_x^2 u\|^2+\varepsilon^2\|Z^\alpha \partial_y\partial^2_{xy}v\|^2\nonumber\\
& +\|Z^\alpha \partial_y\partial_xp\|^2+\|[Z^\alpha\partial_y,  (1+\rho)(1+f_4)]\partial_yf_2\|^2\nonumber.
\end{align}
Here we only need to deal several typical terms in \eqref{3.24} since other terms can be handled similarly. First, we have
\begin{align*}
&\|Z^\alpha\partial_y ((1+\rho)\partial_t u)\|^2\\
\leq& \|Z^\alpha (\partial_y\rho\partial_t u)\|^2+\|Z^\alpha((1+\rho)\partial_t \partial_y u)\|^2\\
\leq& \|\partial_y\rho\|^2_{L^\infty}\|u\|_{m-1}^2+\|\partial_t u\|_{L^\infty}^2\|\partial_y \rho\|_{m-2}^2+(1+\|\rho\|_{L^\infty}^2)\|\partial_yu\|_{m-1}^2+\|\partial_yu\|_{1,\infty}^2\|\rho\|_{m-2}^2\\
\lesssim &(1+P(Q(t)))(\|u\|_{m-1}^2+\|\rho\|_{m-2}^2+\|\partial_y \rho\|_{m-2}^2+\|\partial_yu\|_{m-1}^2)\\
\lesssim &(1+P(Q(t)))(\|u, v, f_2\|_{m}^2+\|\rho\|_{m-2}^2+P(\|f_4\|_m)+\|\partial_y \rho\|_{m-2}^2),
\end{align*}
where the estimate (\ref{3.12}) is used in the last inequality.
By the same argument, we get
\begin{align*}
&\|Z^\alpha \partial_y((1+\rho) u\partial_x u)\|^2+\|Z^\alpha \partial_y((1+\rho)(1+f_1)\partial_xf_1)\|^2+\|Z^\alpha \partial_y((1+\rho) f_2\partial_xf_2)\|^2\\
\lesssim&(1+P(Q(t)))(\|(u, f_1, f_2)\|_{m-1}^2+\|\rho\|_{m-2}^2+\|\partial_y \rho\|_{m-2}^2+\|\partial_y (u, f_1, f_2)\|_{m-1}^2)\\
\lesssim&(1+P(Q(t)))(\|(u, v, f_2)\|_{m}^2+\|f_1\|^2_{m-1}+\|\rho\|_{m-2}^2+P(\|f_4\|_m)+\|\partial_y \rho\|_{m-2}^2+\|\partial_y (f_1, f_2)\|_{m-1}^2).
\end{align*}
Next,
\begin{align*}
Z^\alpha \partial_y((1+\rho) v \partial_y u)=Z^\alpha(\partial_y\rho v \partial_y u)+Z^\alpha ((1+\rho) \partial_y v \partial_y u)+Z^\alpha((1+\rho) v \partial_y\partial_y u),
\end{align*}
where
\begin{align*}
&\|Z^\alpha(\partial_y\rho v \partial_y u)\|^2+\|Z^\alpha ((1+\rho) \partial_y v \partial_y u)\|^2\\
\lesssim &(1+P(Q(t)))(\|(\rho, v)\|_{m-2}^2+\|\partial_y (\rho, u, v)\|_{m-2}^2\\
\lesssim& (1+P(Q(t)))(\|(u, v, f_2)\|_m^2+\|\rho\|_{m-2}^2+P(\|f_4\|_m)+\|\partial_y \rho\|_{m-2}^2,
\end{align*}
and
\begin{align*}
&\|Z^\alpha ((1+\rho) v \partial^2_y u)\|^2
=\left\|Z^\alpha \left((1+\rho) \frac{v}{\varphi(y)} \varphi(y)\partial_y\partial_y u\right)\right\|^2\\
\lesssim&(1+P(Q(t)))(\|\rho\|_{m-2}^2+\|\partial_y v\|_{m-2}^2+\|\partial_yu\|_{m-1}^2)\\
\lesssim&(1+P(Q(t)))(\|\rho\|_{m-2}^2+\|(u, v, f_2)\|_{m}^2+P(\|f_4\|_{m})).
\end{align*}
Similarly,
\begin{align*}
Z^\alpha \partial_y((1+\rho) f_3\partial_yf_1)=Z^\alpha(\partial_y\rho f_3\partial_yf_1)+Z^\alpha ((1+\rho) \partial_y f_3\partial_yf_1)+Z^\alpha((1+\rho) f_3\partial_y\partial_yf_1),
\end{align*}
where
\begin{align*}
&\|Z^\alpha(\partial_y\rho f_3\partial_yf_1)\|^2+\|Z^\alpha ((1+\rho) \partial_y f_3\partial_yf_1)\|^2\\
\lesssim&(1+P(Q(t)))(\|\rho\|_{m-2}^2+\|f_3\|_{m-2}^2+\|\partial_y\rho\|_{m-2}^2+\|\partial_yf_1\|_{m-2}^2+\|\partial_y f_3\|_{m-2}^2),
\end{align*}
and
\begin{align*}
&\|Z^\alpha((1+\rho) f_3\partial_y\partial_yf_1)\|^2
=\left\|Z^\alpha \left((1+\rho) \frac{f_3}{\varphi(y)} \varphi(y)\partial_y\partial_y f_1\right)\right\|^2\\
\lesssim&(1+P(Q(t)))(\|\rho\|_{m-2}^2+\|\partial_y f_3\|_{m-2}^2+\|\partial_yf_1\|_{m-1}^2).
\end{align*}
Now we deal with the last term of commutator in \eqref{3.24}.
Note that
\begin{align*}
[Z^\alpha\partial_y, (1+\rho)(1+f_4)]\partial_yf_2=Z^\alpha(\partial_y((1+\rho)(1+f_4))\partial_yf_2)+[Z^\alpha, (1+\rho)(1+f_4)]\partial^2_yf_2,
\end{align*}
where
\begin{align*}
\|Z^\alpha(\partial_y((1+\rho)(1+f_4))\partial_yf_2)\|^2\lesssim (1+P(Q(t)))(\|(\rho, f_4)\|_{m-2}^2+\|(\partial_y\rho, \partial_yf_2, \partial_yf_4)\|_{m-2}^2),
\end{align*}
and
\begin{align*}
&\|[Z^\alpha, (1+\rho)(1+f_4)]\partial^2_yf_2\|^2\\
\lesssim& \sum\limits_{|\beta|\geq 1, \beta+\kappa=\alpha}\|Z^\beta((1+\rho)(1+f_4))Z^\kappa(\partial^2_yf_2)\|^2\\
=&\sum\limits_{1\leq|\beta|\leq |\alpha|/2, \beta+\kappa=\alpha}\|Z^\beta((1+\rho)(1+f_4))Z^\kappa(\partial^2_yf_2)\|^2\\
&+\sum\limits_{|\beta|>|\alpha|/2, \beta+\kappa=\alpha}\|Z^\beta((1+\rho)(1+f_4))Z^\kappa(\partial^2_yf_2)\|^2\\
\lesssim&\sum\limits_{1\leq|\beta|\leq |\alpha|/2, \beta+\kappa=\alpha}\|Z^\beta((1+\rho)(1+f_4))\|^2_{L_{x,y}^\infty}\|Z^\kappa(\partial^2_yf_2)\|^2_{L^2_{{\bf x}}}\\
&+\sum\limits_{|\beta|>|\alpha|/2, \beta+\kappa=\alpha}\|Z^\beta((1+\rho)(1+f_4))\|^2_{L^2_x(L^\infty_y)}\|Z^\kappa(\partial^2_yf_2)\|^2_{L^\infty_{x}(L^2_y)}\\
\lesssim&\sum\limits_{1\leq|\beta|\leq |\alpha|/2, \beta+\kappa=\alpha}\|Z^\beta((1+\rho)(1+f_4))\|_{L_x^\infty(L^2_y)}\|\partial_yZ^\beta((1+\rho)(1+f_4))\|_{L_x^\infty(L^2_y)}\|Z^\kappa(\partial^2_yf_2)\|^2_{L^2_{{\bf x}}}\\
&+\sum\limits_{|\beta|>|\alpha|/2, \beta+\kappa=\alpha}\|Z^\beta((1+\rho)(1+f_4))\|_{L^2_{{\bf x}}}\|\partial_yZ^\beta((1+\rho)(1+f_4))\|_{L^2_{{\bf x}}}\|Z^\kappa(\partial^2_yf_2)\|^2_{L^\infty_{x}(L^2_y)}\\
\lesssim &(1+P(Q(t)))(\|\rho\|^2_{m-2}+\|f_4\|^2_{m-2}+\|\partial_y\rho\|^2_{m-2}+\|\partial_yf_4\|^2_{m-2})\|\partial^2_yf_2\|_{m-3}^2,
\end{align*}
provided that $m>8$, where Lemma \ref{L4} is used in the last inequality.

Consequently, from the above estimates and \eqref{3.24} we arrive at
\begin{align}
\label{A6}
&\|(1+\rho)(1+f_4)Z^\alpha\partial^2_yf_2\|^2+\mu^2\varepsilon^2\|Z^\alpha\partial_y^3 u\|^2+2\mu\varepsilon \int(1+\rho)(1+f_4) Z^\alpha\partial^2_yf_2\cdot Z^\alpha\partial_y^3 u d{\bf x}\nonumber\\
\lesssim& (1+P(Q(t)))(\|(u, v, f_2)\|_{m}^2+\|f_1\|_{m-1}^2+P(\|f_4\|_m)+\|(\rho, f_3)\|_{m-2}^2)\nonumber\\
&+(1+P(Q(t)))(\|\partial_y(f_1,f_2)\|_{m-1}^2+\|\partial_y(\rho, f_3, f_4)\|_{m-2}^2)\nonumber\\
&+\|\partial_yp\|_{m-1}^2+\varepsilon^2\|\partial_y u\|_m^2+\varepsilon^2\|\partial_y^2v\|_{m-1}^2 \\
&+(1+P(Q(t)))(\|\rho\|^2_{m-2}+\|f_4\|^2_{m-2}+\|\partial_y\rho\|^2_{m-2}+\|\partial_yf_4\|^2_{m-2})\|\partial^2_yf_2\|_{m-3}^2. \nonumber
\end{align}

\subsubsection*{Step 2}
Applying the operator $Z^\alpha\partial^2_y\ (|\alpha|\leq m-2)$ on the equation $(1+f_4)\times(\ref{3.19})$ gives that
\begin{align}
\label{3.25}
\partial_tZ^\alpha \partial^2_yf_2-(1+f_4)Z^\alpha \partial^3_yu
=Z^\alpha \partial^2_y(f_2\partial_xu)-u\partial_xZ^\alpha \partial_y^2 f_2-v\partial_yZ^\alpha \partial_y^2f_2+\mathcal{C}_7^\alpha,
\end{align}
where
\begin{align}\label{C7}
\mathcal{C}_7^\alpha=[Z^\alpha \partial^2_y, (1+f_4)]\partial_y u-[Z^\alpha \partial^2_y, u\partial_x]f_2-[Z^\alpha \partial^2_y, v\partial_y]f_2.
\end{align}
Multiplying (\ref{3.25}) by $2\mu\varepsilon (1+\rho) Z^\alpha\partial_y^2f_2$ and integrating the resulting equality over $\mathbb{R}^2_+$ give
\begin{align}
\label{3.26}
&\mu\varepsilon\frac{d}{dt}\|(1+\rho)Z^\alpha\partial^2_y f_2\|^2-2\mu\varepsilon\int (1+\rho)(1+f_4)Z^\alpha\partial^2_yf_2 Z^\alpha\partial_y^3ud{\bf x}\\
=&2\mu\varepsilon\int (1+\rho)Z^\alpha\partial^2_yf_2 Z^\alpha\partial^2_y(f_2\partial_xu)d{\bf x}+2\mu\varepsilon\int (1+\rho) Z^\alpha\partial^2_yf_2 \mathcal{C}_7^\alpha d{\bf x}.\nonumber
\end{align}
Since
$$\partial^2_y\left(f_2\partial_xu\right)
= \partial_y^2f_2\partial_xu+2\partial_yf_2\partial_x\partial_yu+f_2\partial_x\partial^2_yu,$$
by the Cauchy-Schwarz inequality, we get
\begin{align*}
&\left|2\mu\varepsilon\int (1+\rho) Z^\alpha\partial_y^2f_2 Z^\alpha\partial^2_y\left(f_2\partial_xu\right)d{\bf x}+2\mu\varepsilon\int (1+\rho) Z^\alpha\partial_y^2f_2 \mathcal{C}_7^\alpha d{\bf x}\right|\\
\leq & \delta\|(1+\rho)(1+f_4)Z^\alpha\partial_y^2f_2\|^2+C_\delta\varepsilon^2\|\mathcal{C}_7^\alpha\|^2\\
&+C_\delta\varepsilon^2
\left\|Z^\alpha\left(\partial_y^2f_2\partial_xu+2\partial_yf_2\partial_x\partial_yu+f_2\partial_x\partial^2_yu\right)\right\|^2.
\end{align*}
Notice that
\begin{align*}
&\left\|Z^\alpha\left(\partial_y^2f_2\partial_xu\right)\right\|^2+\left\|Z^\alpha\left(\partial_yf_2\partial_x\partial_yu\right)\right\|^2\\
=&\left\|Z^\alpha\left(\frac{\partial_xu}{\varphi(y)}\varphi(y)\partial_y\partial_yf_2\right)\right\|^2+\left\|Z^\alpha\left(\partial_yf_2\partial_x\partial_yu\right)\right\|^2\\
\lesssim&\|\partial_yu\|^2_{1, \infty}\left\|\partial_yf_2\right\|^2_{m-1}+\left\|\partial_yf_2\right\|^2_{1,\infty}\|\partial_yu\|^2_{m-1}\\
\lesssim&(1+P(Q(t)))(\|\partial_yu\|^2_{m-1}+\|\partial_yf_2\|^2_{m-1}),
\end{align*}
and
\begin{align*}
&\left\|Z^\alpha(f_2\partial_x\partial^2_yu)\right\|^2\\
\lesssim&\sum\limits_{|\beta|\leq |\alpha|/2, \beta+\kappa=\alpha}\left\|Z^\beta f_2Z^\kappa\partial_x\partial^2_yu \right\|^2
+\sum\limits_{|\beta|>|\alpha|/2, \beta+\kappa=\alpha}\left\|Z^\beta f_2 Z^\kappa\partial_x\partial^2_yu \right\|^2\\
\lesssim&\sum\limits_{|\beta|\leq |\alpha|/2, \beta+\kappa=\alpha}\left \|Z^\beta f_2 \right\|^2_{L_{x,y}^\infty}\left\|Z^\kappa\partial_x\partial^2_yu \right \|^2_{L^2_{{\bf x}}} \\
&+\sum\limits_{|\beta|>|\alpha|/2, \beta+\kappa=\alpha}\left\|Z^\beta f_2 \right\|^2_{L^2_{x}(L^\infty_y)}\|Z^\kappa\partial_x\partial^2_yu \|^2_{L_{x}^\infty(L^2_y)}\\
\lesssim&\sum\limits_{|\beta|\leq |\alpha|/2, \beta+\kappa=\alpha}\left \|Z^\beta f_2 \right\|_{L_{x}^\infty(L^2_y)}\left \|\partial_yZ^\beta f_2 \right\|_{L_{x}^\infty(L^2_y)}\left\|Z^\kappa\partial_x\partial^2_yu \right \|^2_{L^2_{{\bf x}}} \\
&+\sum\limits_{|\beta|>|\alpha|/2, \beta+\kappa=\alpha}\left\|Z^\beta f_2 \right\|_{L^2_{{\bf x}}}\left\|\partial_yZ^\beta f_2 \right\|_{L^2_{{\bf x}}}\|Z^\kappa\partial_x\partial^2_yu \|^2_{L_{x}^\infty(L^2_y)}\\
\lesssim&(1+P(Q(t)))(\|f_2\|_{m-2}^2+\|\partial_yf_2\|^2_{m-2})\|\partial^2_yu\|_{m-1}^2,
\end{align*}
provided that $m>4$.

\subsubsection*{Step 3}
Next, we deal with the estimates for commutator of $\mathcal{C}_7^\alpha$ defined in \eqref{C7}. First,
\begin{align*}
&\left\|\left[Z^\alpha \partial^2_y, (1+f_4)\right]\partial_y u\right\|^2\\
\lesssim&\left\|Z^\alpha\left( \partial^2_y(1+f_4)\partial_yu\right)\right\|^2+\left\|Z^\alpha\left( \partial_y(1+f_4)\partial_y\partial_yu\right)\right\|^2
+\left\|\left[Z^\alpha, (1+f_4)\right]\partial^2_y\partial_yu\right\|^2.
\end{align*}
Since
\begin{align*}
\partial_yf_4=\frac{1}{1+\rho}\{-\partial_x((1+\rho) f_2)-(1+f_4)\partial_y\rho\}
\end{align*}
due to $\partial_x((1+\rho) f_2)+\partial_y((1+\rho) (1+f_4))=0$,  then
\begin{align*}
&\left\|Z^\alpha\left( \partial^2_y(1+f_4)\partial_yu\right)\right\|^2\\
=&\left\|Z^\alpha\left( \partial_y\left(\frac{1}{1+\rho}\{-\partial_x((1+\rho) f_2)-(1+f_4)\partial_y\rho\}\right)\partial_yu\right)\right\|^2\\
\lesssim&\left\|Z^\alpha\left( \partial_y\left(\frac{\partial_x((1+\rho) f_2)}{(1+\rho)}\right)\partial_yu\right)\right\|^2+\left\|Z^\alpha\left( \partial_y\left(\frac{(1+f_4)\partial_y\rho}{(1+\rho)}\right)\partial_yu\right)\right\|^2\\
\lesssim&\left\|Z^\alpha\left( \partial_y\left(\frac{\partial_x((1+\rho) f_2)}{(1+\rho)}\right)\partial_yu\right)\right\|^2+\left\|Z^\alpha\left( \partial_y\left(\frac{(1+f_4)}{(1+\rho)}\right)\partial_y\rho\partial_yu\right)\right\|^2\\
&+\left\|Z^\alpha\left(\frac{(1+f_4)\partial_yu}{(1+\rho)}\partial^2_y\rho\right)\right\|^2\\
\lesssim&(1+P(Q(t)))\Big(\|\rho\|_{m-2}^2+\|\partial_y\rho\|_{m-2}^2+\|\partial_yu\|_{m-2}^2+\|\partial^2_yu\|^2_{m-2}\\
&\qquad\qquad\qquad\qquad+\|f_4\|_{m-2}^2+\|\partial_yf_4\|^2_{m-2}\Big)\|\partial^2_y\rho\|_{m-2}^2\\
&+(1+P(Q(t)))\Big(\|\partial_y(\rho, f_2)\|_{m-1}^2+\|(\rho, f_2)\|_{m-1}^2+\|f_4\|_{m-2}^2+\|\partial_y (u, f_4)\|_{m-2}^2\Big),
\end{align*}
where  the following estimates are used in the last inequality:
\begin{align*}
&\left\|Z^\alpha\left( \partial_y\left(\frac{\partial_x((1+\rho) f_2)}{(1+\rho)}\right)\partial_yu\right)\right\|^2+\left\|Z^\alpha\left( \partial_y\left(\frac{(1+f_4)}{(1+\rho)}\right)\partial_y\rho\partial_yu\right)\right\|^2\\
\lesssim&(1+P(Q(t)))\Big(\|\partial_y(\rho, f_2)\|_{m-1}^2+\|(\rho, f_2)\|_{m-1}^2+\|f_4\|_{m-2}^2+\|\partial_y (u, f_4)\|_{m-2}^2\Big),
\end{align*}
and
\begin{align*}
&\left\|Z^\alpha\left(\frac{(1+f_4)\partial_yu}{(1+\rho)}\partial^2_y\rho\right)\right\|^2\\
\lesssim&\sum\limits_{|\beta|\leq |\alpha|/2, \beta+\kappa=\alpha}\left\|Z^\beta\left(\frac{(1+f_4)\partial_yu}{(1+\rho)}\right)Z^\kappa\partial^2_y\rho \right\|^2\\
&+\sum\limits_{|\beta|>|\alpha|/2, \beta+\kappa=\alpha}\left\|Z^\beta\left(\frac{(1+f_4)\partial_yu}{(1+\rho)}\right)Z^\kappa\partial^2_y\rho \right\|^2\\
\lesssim&\sum\limits_{|\beta|\leq |\alpha|/2, \beta+\kappa=\alpha}\left \|Z^\beta\left(\frac{(1+f_4)\partial_yu}{(1+\rho)}\right)\right\|_{L_{x,y}^\infty}\left\|Z^\kappa\partial^2_y\rho  \right \|^2_{L^2_{{\bf x}}} \\
&+\sum\limits_{|\beta|>|\alpha|/2, \beta+\kappa=\alpha}\left\|Z^\beta\left(\frac{(1+f_4)\partial_yu}{(1+\rho)}\right)\right\|_{L^2_{x}(L^\infty_y)}\|Z^\kappa\partial^2_y\rho \|^2_{L_{x}^\infty(L^2_y)}\\
\lesssim&\sum\limits_{|\beta|\leq |\alpha|/2, \beta+\kappa=\alpha}\left \|Z^\beta\left(\frac{(1+f_4)\partial_yu}{(1+\rho)}\right)\right\|_{L_{x}^\infty(L^2_y)}\left \|\partial_yZ^\beta\left(\frac{(1+f_4)\partial_yu}{(1+\rho)}\right)\right\|_{L_{x}^\infty(L^2_y)}\left\|Z^\kappa\partial^2_y\rho \right \|^2_{L^2_{{\bf x}}} \\
&+\sum\limits_{|\beta|>|\alpha|/2, \beta+\kappa=\alpha}\left\|Z^\beta\left(\frac{(1+f_4)\partial_yu}{(1+\rho)}\right)\right\|_{L^2}\left\|\partial_yZ^\beta\left(\frac{(1+f_4)\partial_yu}{(1+\rho)}\right)\right\|_{L^2}
\|Z^\kappa\partial^2_y\rho \|^2_{L_{x}^\infty(L^2_y)}\\
\lesssim&(1+P(Q(t)))\Big(\|\rho\|_{m-2}^2+\|\partial_y\rho\|_{m-2}^2+\|\partial_yu\|_{m-2}^2+\|\partial^2_yu\|^2_{m-2}\\
&\qquad\qquad\qquad\qquad+\|f_4\|_{m-2}^2+\|\partial_yf_4\|^2_{m-2}\Big)\|\partial^2_y\rho\|_{m-2}^2,
\end{align*}
provided that $m>6$.

Moreover,
\begin{align*}
&\left\|Z^\alpha\left( \partial_y(1+f_4)\partial_y\partial_yu\right)\right\|^2\\
\lesssim &(1+P(Q(t)))\|\partial_yf_4\|_{m-2}^2(\|\partial^2_yu\|_{m-2}^2+\|\partial^3_yu\|_{m-2}^2),
\end{align*}
and
\begin{align}
\label{B1}
&\left\|\left[Z^\alpha, (1+f_4)\right]\partial^2_y\partial_yu\right\|^2\nonumber\\
\lesssim &\sum\limits_{1\leq|\beta|\leq |\alpha|/2, \beta+\kappa=\alpha}\left\|Z^\beta\left(1+f_4\right)Z^\kappa \partial^3_yu \right\|^2
+\sum\limits_{|\beta|>|\alpha|/2, \beta+\kappa=\alpha}\left\|Z^\beta\left(1+f_4\right)Z^\kappa \partial^3_yu \right\|^2\nonumber\\
\lesssim&\sum\limits_{1\leq|\beta|\leq |\alpha|/2, \beta+\kappa=\alpha}\left \|Z^\beta\left(1+f_4\right)\right\|_{L_{x}^\infty(L^2_y)}\left \|\partial_yZ^\beta\left(1+f_4\right)\right\|_{L_{x}^\infty(L^2_y)}\left\|Z^\kappa \partial^3_yu \right \|^2_{L^2_{{\bf x}}} \nonumber\\
&+\sum\limits_{|\beta|>|\alpha|/2, \beta+\kappa=\alpha}\left\|Z^\beta\left(1+f_4 \right)\right\|_{L^2_{{\bf x}}}\left\|\partial_yZ^\beta\left(1+f_4\right)\right\|_{L^2_{{\bf x}}}\|Z^\kappa \partial^3_yu \|^2_{L_{x}^\infty(L^2_y)}\nonumber\\
\lesssim&(1+P(Q(t)))(\|f_4\|^2_{m-2}+\|\partial_yf_4\|^2_{m-2})\|\partial^3_yu\|_{m-3}^2,
\end{align}
provided that $m>6$.

Similarly,
\begin{align}
\label{B2}
&\left\|\left[Z^\alpha\partial^2_y, u\partial_x\right]f_2\right\|^2\nonumber\\
\lesssim &\left\|Z^\alpha(\partial^2_yu\partial_xf_2)\right\|^2+\left\|Z^\alpha(\partial_yu\partial_x\partial_yf_2)\right\|^2+\left\|[Z^\alpha, u/\varphi(y)]\partial_x\varphi(y)\partial_y^2f_2)\right\|^2\nonumber\\
\lesssim &(\|\partial_y^2u\|_{m-2}^2+\|\partial_y^3u\|_{m-2}^2)\|f_2\|_{m-1}^2+(\|\partial_yu\|_{m-2}^2+\|\partial_y^2u\|_{m-2}^2)\|\partial_yf_2\|_{m-1}^2
\end{align}
and
\begin{align}
\label{B3}
&\left\|\left[Z^\alpha\partial^2_y, v\partial_y\right]f_2\right\|^2\nonumber\\
\lesssim &\left\|Z^\alpha(\partial^2_yv\partial_yf_2)\right\|^2+\left\|Z^\alpha(\partial_yv\partial_y\partial_yf_2)\right\|^2+\left\|[Z^\alpha, v/\varphi(y)]\varphi(y)\partial_y\partial_y^2f_2)\right\|^2\nonumber\\
\lesssim &(\|\partial_y^2v\|_{m-2}^2+\|\partial_y^3v\|_{m-2}^2)\|\partial_yf_2\|_{m-2}^2+(\|\partial_yv\|_{m-2}^2+\|\partial_y^2v\|_{m-2}^2)\|\partial_y^2f_2\|_{m-2}^2.
\end{align}
\subsubsection*{Step 4}
Consequently, from \eqref{3.26} we have
\begin{align}
\label{A7}
&\mu\varepsilon\frac{d}{dt}\|(1+\rho)Z^\alpha\partial^2_y f_2\|^2-2\mu\varepsilon\int (1+\rho)(1+f_4)Z^\alpha\partial^2_yf_2 Z^\alpha\partial_y^3ud{\bf x}\nonumber\\
\lesssim&\delta\|(1+\rho)(1+f_4)Z^\alpha\partial_y^2f_2\|^2+\varepsilon^2(\|\partial_yu\|_{m-2}^2+\|\partial_y^2u\|_{m-2}^2)\|\partial_yf_2\|_{m-1}^2\nonumber\\
&+\varepsilon^2\{(\|\partial_y^2v\|_{m-2}^2+\|\partial_y^3v\|_{m-2}^2)\|\partial_yf_2\|_{m-2}^2+(\|\partial_yv\|_{m-2}^2+\|\partial_y^2v\|_{m-2}^2)\|\partial_y^2f_2\|_{m-2}^2\}\nonumber\\
&+C_\delta\varepsilon^2(1+P(Q(t)))(\|\partial_y(\rho, u, f_2))\|_{m-1}^2+\|(\rho, f_2)\|_{m-1}^2+\|f_4\|_{m-2}^2+\|\partial_yf_4\|_{m-2}^2)\nonumber\\
&+C_\delta\varepsilon^2(1+P(Q(t)))(\|f_2\|_{m-2}^2+\|\partial_yf_2\|_{m-2}^2)\|\partial_y^2u\|_{m-1}^2\\
&+C_\delta\varepsilon^2(1+P(Q(t)))(\|\rho\|_{m-2}^2+\|\partial_y\rho\|_{m-2}^2+\|\partial_yu\|_{m-2}^2+\|\partial_y^2u\|_{m-2}^2\nonumber\\
&+\|f_4\|_{m-2}^2+\|\partial_yf_4\|_{m-2}^2)\times\|\partial_y^2\rho\|_{m-2}^2\nonumber\\
&+C_\delta\varepsilon^2(\|f_4\|_{m-2}^2+\|f_2\|_{m-1}^2+\|\partial_yf_4\|_{m-2}^2)(\|\partial_y^2u\|_{m-2}^2+\|\partial_y^3u\|_{m-2}^2)\nonumber.
\end{align}
Combining (\ref{A6}) and (\ref{A7}) and choosing $\delta$ suitably small, we have
\begin{align}
\label{A8}
&\mu\varepsilon\frac{d}{dt}\|(1+\rho)Z^\alpha\partial^2_y f_2\|^2+\|(1+\rho)(1+f_4)Z^\alpha\partial^2_yf_2\|^2+\mu^2\varepsilon^2\|Z^\alpha\partial_y^3 u\|^2\nonumber\\
\lesssim& (1+P(Q(t)))(\|(u, v, f_2)\|_{m}^2+\|f_1\|_{m-1}^2+P(\|f_4\|_m)+\|(\rho, f_3)\|_{m-2}^2)\nonumber\\
&+(1+P(Q(t)))(\|\partial_y(f_1,f_2)\|_{m-1}^2+\|\partial_y(\rho, f_3, f_4)\|_{m-2}^2)\nonumber\\
&+\|\partial_yp\|_{m-1}^2+\varepsilon^2\|\partial_y u\|_m^2+\varepsilon^2\|\partial_y^2v\|_{m-1}^2+\varepsilon^2(\|\partial_yu\|_{m-2}^2+\|\partial_y^2\|_{m-2}^2)\|\partial_yf_2\|_{m-1}^2 \nonumber\\
&+(1+P(Q(t)))(\|\rho\|^2_{m-2}+\|f_4\|^2_{m-2}+\|\partial_y\rho\|^2_{m-2}+\|\partial_yf_4\|^2_{m-2})\|\partial^2_yf_2\|_{m-3}^2. \nonumber\\
&+\varepsilon^2\{(\|\partial_y^2v\|_{m-2}^2+\|\partial_y^3v\|_{m-2}^2)\|\partial_yf_2\|_{m-2}^2+(\|\partial_yv\|_{m-2}^2+\|\partial_y^2v\|_{m-2}^2)\|\partial_y^2f_2\|_{m-2}^2\}\nonumber\\
&+C_\delta\varepsilon^2(1+P(Q(t)))(\|\partial_y(\rho, u, f_2))\|_{m-1}^2+\|(\rho, f_2)\|_{m-1}^2+\|f_4\|_{m-2}^2+\|\partial_yf_4\|_{m-2}^2)\nonumber\\
&+C_\delta\varepsilon^2(1+P(Q(t)))(\|f_2\|_{m-2}^2+\|\partial_yf_2\|_{m-2}^2)\|\partial_y^2u\|_{m-1}^2\\
&+C_\delta\varepsilon^2(1+P(Q(t)))(\|\rho\|_{m-2}^2+\|\partial_y\rho\|_{m-2}^2+\|\partial_yu\|_{m-2}^2+\|\partial_y^2u\|_{m-2}^2\nonumber\\
&+\|f_4\|_{m-2}^2+\|\partial_yf_4\|_{m-2}^2)\times\|\partial_y^2\rho\|_{m-2}^2\nonumber\\
&+C_\delta\varepsilon^2(\|f_4\|_{m-2}^2+\|f_2\|_{m-1}^2+\|\partial_yf_4\|_{m-2}^2)(\|\partial_y^2u\|_{m-2}^2+\|\partial_y^3u\|_{m-2}^2)\nonumber.
\end{align}
\subsection{Estimate of $\partial_y p$}
From the second equation in (\ref{3.15}), we have
\begin{align*}
\partial_yp-(2\mu+\lambda)\varepsilon\partial_y^2v=-(1+\rho) \partial_t v-(1+\rho) u\partial_xv-(1+\rho) v\partial_yv+(1+\rho)(1+f_1)\partial_xf_3\nonumber\\
\qquad+(1+\rho) f_3\partial_yf_3+(1+\rho) f_2\partial_xf_4+(1+\rho) (1+f_4)\partial_yf_4+\mu\varepsilon\partial_x^2v+(\mu+\lambda)\varepsilon\partial_yu_x.
\end{align*}
Since $\partial_x((1+\rho)f_2)+\partial_y((1+\rho)(1+f_4))=0$, then
\begin{align*}
(1+\rho)(1+f_4)\partial_yf_4=&(1+\rho)(1+f_4)\frac{1}{(1+\rho)}\{-\partial_x((1+\rho) f_2)-(1+f_4)\partial_y\rho\}\\
=&-(1+f_4)\partial_x((1+\rho) f_2)-(1+f_4)^2\partial_y\rho\\
=&-(1+f_4)\partial_x((1+\rho) f_2)-\frac{(1+f_4)^2}{\gamma(1+\rho)^{\gamma-1}}\partial_yp,
\end{align*}
and  consequently,
\begin{align}
\label{3.27}
&\left(1+\frac{(1+f_4)^2}{\gamma(1+\rho)^{\gamma-1}}\right)\partial_yp-(2\mu+\lambda)\varepsilon\partial_y^2v\nonumber\\
=&-(1+\rho)\partial_tv-(1+\rho) u\partial_xv-(1+\rho) v\partial_yv\\
&+(1+\rho)(1+f_1)\partial_xf_3+(1+\rho) f_3\partial_yf_3\nonumber\\
&+(1+\rho) f_2\partial_xf_4-(1+f_4)\partial_x((1+\rho)f_2)+\mu\varepsilon\partial_x^2v+(\mu+\lambda)\varepsilon\partial_yu_x. \nonumber
\end{align}

\subsubsection*{Step 1}
Applying the operator $Z^\alpha (|\alpha|\leq m-1)$ on (\ref{3.27}) leads to
\begin{align}
\label{3.28}
&\left(1+\frac{(1+f_4)^2}{\gamma(1+\rho)^{\gamma-1}}\right)Z^\alpha\partial_yp-(2\mu+\lambda)\varepsilon Z^\alpha \partial_y^2v\nonumber\\
=&Z^\alpha\left\{-(1+\rho) \partial_t v-(1+\rho)u\partial_xv-(1+\rho)v\partial_yv+(1+\rho)(1+f_1)\partial_xf_3+(1+\rho)f_3\partial_yf_3\right\}\nonumber \\
&+Z^\alpha\left\{(1+\rho) f_2\partial_xf_4-(1+f_4)\partial_x((1+\rho) f_2)+\mu\varepsilon\partial_x^2v+(\mu+\lambda)\varepsilon\partial_yu_x\right\}
+\mathcal{C}_8^\alpha,
\end{align}
where
\begin{align*}
\mathcal{C}_8^\alpha=\left[Z^\alpha, \left(1+\frac{(1+f_4)^2}{\gamma(1+\rho)^{\gamma-1}}\right)\right]\partial_y p.
\end{align*}
Taking $L^2$ inner product on the both sides of (\ref{3.28}) gives
\begin{align}
\label{3.38a}
&\left\|\left(1+\frac{(1+f_4)^2}{\gamma(1+\rho)^{\gamma-1}}\right)Z^\alpha\partial_yp\right\|^2+(2\mu+\lambda)^2\varepsilon^2\|Z^\alpha\partial_y^2v\|^2\nonumber\\
&\quad -2(2\mu+\lambda)\varepsilon\int\left(1+\frac{(1+f_4)^2}{\gamma(1+\rho)^{\gamma-1}}\right)Z^\alpha\partial_ypZ^\alpha\partial_y^2vd{\bf x}\nonumber\\
\leq&\|Z^\alpha\{-(1+\rho) \partial_t v-(1+\rho) u\partial_xv-(1+\rho) v\partial_yv+(1+\rho)(1+f_1)\partial_xf_3+(1+\rho)f_3\partial_yf_3\}\|^2\nonumber\\
&+\|Z^\alpha\{(1+\rho) f_2\partial_xf_4-(1+f_4)\partial_x((1+\rho) f_2)+\mu\varepsilon\partial_x^2v+(\mu+\lambda)\varepsilon\partial_yu_x\}\|^2+\|\mathcal{C}_8^\alpha\|^2\nonumber\\
\leq&\left\|Z^\alpha\{-(1+\rho) \partial_t v-(1+\rho) u\partial_xv-(1+\rho) \frac{v}{\varphi(y)}\varphi(y)\partial_yv+(1+\rho)(1+f_1)\partial_xf_3\}\right\|^2\nonumber\\
&+\left\|Z^\alpha\{(1+\rho) \frac{f_3}{\varphi(y)}\varphi(y)\partial_yf_3\}\right\|^2\\
&+\|Z^\alpha\{(1+\rho) f_2\partial_xf_4-(1+f_4)\partial_x((1+\rho) f_2)+\mu\varepsilon\partial_x^2v+(\mu+\lambda)\varepsilon\partial_yu_x\}\|^2+\|\mathcal{C}_8^\alpha\|^2\nonumber \\
\lesssim&(1+P(Q(t)))(\|(\rho, v, f_2, f_3, f_4)\|_m^2+\|(u, f_1)\|_{m-1}^2+\|\partial_y(v, f_3)\|_{m-1}^2)\nonumber\\
&+\varepsilon^2\|(\partial_xv, \partial_yu)\|_{m}^2+\|\mathcal{C}_8^\alpha\|^2\nonumber\\
\lesssim&(1+P(Q(t)))(\|(\rho, v, f_2, f_3, f_4)\|_m^2+\|(u, f_1)\|_{m-1}^2+P(\|\rho\|_{m-1}))\nonumber\\
&+(1+P(Q(t)))(\|\partial_y f_3\|_{m-1}^2+\|\partial_y p\|_{m-2}^2)+\varepsilon^2\|(\partial_xv, \partial_yu)\|_{m}^2,   \nonumber
\end{align}
where we have used  (\ref{3.9})  and the following estimates in the last inequality:
\begin{align*}
\|\mathcal{C}_8^\alpha\|^2=&\left\|\left[Z^\alpha, \left(1+\frac{(1+f_4)^2}{\gamma(1+\rho)^{\gamma-1}}\right)\right]\partial_y p\right\|^2\\
\lesssim&\left\|Z\left(1+\frac{(1+f_4)^2}{\gamma(1+\rho)^{\gamma-1}}\right)\right\|_{L^\infty}^2\|\partial_y p\|_{m-2}^2+\|\partial_y p\|_{L^\infty}^2\left\|Z\left(1+\frac{(1+f_4)^2}{\gamma(1+\rho)^{\gamma-1}}\right)\right\|_{m-2}^2\\
\lesssim& (1+P(Q(t)))(\|\partial_y p\|_{m-2}^2+P(\|\rho\|_{m-1})+\|f_4\|_{m-1}^2).
\end{align*}

\subsubsection*{Step 2}
From the equation of conservation of mass, we have
\begin{align}
\label{3.30}
\frac{1}{\gamma p}\partial_t p+\partial_yv=-\frac{1}{\gamma p}(u\partial_x p+v\partial_y  p)-\partial_xu.
\end{align}
Applying the operator $Z^\alpha\partial_y (|\alpha|\leq m-1)$ on (\ref{3.30}) yields
\begin{align}
\label{3.31}
\frac{1}{\gamma p}\partial_t Z^\alpha\partial_yp+Z^\alpha\partial^2_yv=-\frac{1}{\gamma p}(u\partial_x Z^\alpha\partial_y p+v\partial_y  Z^\alpha\partial_yp)-Z^\alpha\partial_y\partial_xu+\mathcal{C}_9^\alpha,
\end{align}
with
\begin{align*}
\mathcal{C}_9^\alpha=-\left[Z^\alpha\partial_y, \frac{1}{\gamma p}\partial_t\right]p-\left[Z^\alpha\partial_y, \frac{u}{\gamma p}\partial_x\right]p-\left[Z^\alpha\partial_y, \frac{v}{\gamma p}\partial_y\right]p.
\end{align*}
Multiplying (\ref{3.31}) by
$$2(2\mu+\lambda)\varepsilon\left(1+ \frac{(1+f_4)^2}{\gamma(1+\rho)^{\gamma-1}}\right) Z^\alpha\partial_y p\triangleq 2a(\rho, f_4)\varepsilon Z^\alpha\partial_y p$$
 and integrating the resulting equality over $\mathbb{R}^2_+$, we obtain
\begin{align}
\label{3.32}
&\frac{d}{dt}\varepsilon\left\|\sqrt{\frac{a(\rho, f_4)}{\gamma p}}Z^\alpha\partial_yp\right\|^2+2(2\mu+\lambda)\varepsilon\int\left(1+\frac{(1+f_4)^2}{\gamma(1+\rho)^{\gamma-1}}\right)Z^\alpha\partial_ypZ^\alpha\partial_y^2vd{\bf x}\nonumber\\
=&\varepsilon\int\left(\left(\frac{a(\rho, f_4)}{\gamma p}\right)_t+\left(\frac{u a(\rho, f_4)}{\gamma p}\right)_x+\left(\frac{v a(\rho, f_4)}{\gamma p}\right)_y\right)(Z^\alpha\partial_yp)^2d{\bf x}
\nonumber\\
&-\varepsilon\int 2a(\rho, f_4) Z^\alpha\partial_y pZ^\alpha\partial_y\partial_xu d{\bf x}+\varepsilon\int 2a(\rho, f_4) Z^\alpha\partial_y p\mathcal{C}_9^\alpha d{\bf x}.
\end{align}
First,
\begin{align*}
&\varepsilon\left|\int\left(\left(\frac{a(\rho, f_4)}{\gamma p}\right)_t+\left(\frac{u a(\rho, f_4)}{\gamma p}\right)_x+\left(\frac{v a(\rho, f_4)}{\gamma p}\right)_y\right)(Z^\alpha\partial_yp)^2d{\bf x}\right|\\
\lesssim& \varepsilon(1+P(Q(t)))\|Z^\alpha\partial_yp\|^2,
\end{align*}
and
\begin{align*}
&\left|\varepsilon\int a(\rho, f_4) Z^\alpha\partial_y pZ^\alpha\partial_y\partial_xu d{\bf x}\right|\\
\lesssim& \varepsilon(1+P(Q(t)))\|Z^\alpha\partial_y p\|\|Z^\alpha\partial_y\partial_xu\|\\
\lesssim& \delta\|Z^\alpha\partial_y p\|^2+C_\delta\varepsilon^2(1+P(Q(t)))\|Z^\alpha\partial_y\partial_xu\|^2.
\end{align*}

\subsubsection*{Step 3}
The commutator $\mathcal{C}_9^\alpha$ is estimated as follows. We note that
\begin{align*}
\varepsilon\left|\int a(\rho, f_4) Z^\alpha\partial_y p\mathcal{C}_9^\alpha d{\bf x}\right|\leq \delta\|Z^\alpha\partial_y p\|^2+C_\delta\varepsilon^2\|a(\rho, f_4) \|_{L^\infty}^2\|\mathcal{C}_9^\alpha\|^2.
\end{align*}
For the first term in $\mathcal{C}_9^\alpha$, one has
\begin{align*}
\left\|\left[Z^\alpha\partial_y, \frac{1}{\gamma p}\partial_t\right]p\right\|\leq \left\|Z^\alpha\left(\partial_y\left(\frac{1}{\gamma p}\right)\partial_tp\right)\right\|+\left\|\left[Z^\alpha, \frac{1}{\gamma p}\right]\partial_t\partial_yp\right\|,
\end{align*}
where
\begin{align*}
&\left\|Z^\alpha\left(\partial_y\left(\frac{1}{\gamma p}\right)\partial_tp\right)\right\|=\left\|Z^\alpha\left(\left(\frac{\partial_y p}{\gamma p^2}\right)\partial_tp\right)\right\|\\
\lesssim&\|\partial_yp\|_{L^\infty}\|p_t\|_{m-1}+\|p_t\|_{L^\infty}(1+P(Q(t)))(\|\partial_yp\|_{m-1}+P(\|p-1\|_{m-1}))\\
\lesssim &(1+P(Q(t)))(\|\partial_yp\|_{m-1}+P(\|p-1\|_{m})),
\end{align*}
and
\begin{align*}
&\left\|\left[Z^\alpha, \frac{1}{\gamma p}\right]\partial_t\partial_yp\right\|\\
\lesssim&\sum\limits_{|\beta|\geq 1, \beta+\kappa=\alpha}\left\|Z^\beta\left(\frac{1}{\gamma p}\right)Z^\kappa\partial_t\partial_yp\right\|\\
\lesssim&\left\|Z\left(\frac{1}{\gamma p}\right)\right\|_{L^\infty}\|\partial_yp\|_{m-1}+\|\partial_yp\|_{1, \infty}P(\|p-1\|_{m-1})\\
\lesssim &(1+P(Q(t)))(P(\|p-1\|_{m-1})+\|\partial_yp\|_{m-1}).
\end{align*}
Similarly, we have
\begin{align*}
&\left\|\left[Z^\alpha\partial_y, \frac{u}{\gamma p}\partial_x\right]p\right\|=\left\|Z^\alpha\left(\partial_y\left(\frac{u}{\gamma p}\right)\partial_x p\right)+\left[Z^\alpha, \frac{u}{\gamma p}\right]\partial_x\partial_yp\right\|\\
\lesssim&\left\|\partial_y\left(\frac{u}{\gamma p}\right)\right\|_{L^\infty}\|\partial_xp\|_{m-1}+\left\|\partial_y\left(\frac{u}{\gamma p}\right)\right\|_{m-1}\|\partial_xp\|_{L^\infty}+\sum\limits_{|\beta|\geq 1, \beta+\kappa=\alpha}\left\|Z^\beta\left(\frac{u}{\gamma p}\right)Z^\kappa\partial_x\partial_yp\right\|\\
\lesssim&\left\|\partial_y\left(\frac{u}{\gamma p}\right)\right\|_{L^\infty}\|\partial_xp\|_{m-1}+\left\|\partial_y\left(\frac{u}{\gamma p}\right)\right\|_{m-1}\|\partial_xp\|_{L^\infty}\\
&+\sum\limits_{|\beta|\geq 1, \beta+\kappa=\alpha}\left\|Z\left(\frac{u}{\gamma p}\right)\right\|_{L^\infty}\|\partial_yp\|_{m-1}+\|\partial_yp\|_{1, \infty}\left\|\frac{u}{\gamma p}\right\|_{m-1}\\
\lesssim &(1+P(Q(t)))(P(\|p-1\|_{m})+\|u\|_{m-1}+\|\partial_yp\|_{m-1}+\|\partial_yu\|_{m-1}),
\end{align*}
and
\begin{align*}
&\left\|\left[Z^\alpha\partial_y, \frac{v}{\gamma p}\right]\partial_yp\right\|\\
=&\left\|Z^\alpha\left(\partial_y\left(\frac{v}{\gamma p}\right)\partial_y p\right)+\left[Z^\alpha, \frac{v}{\gamma p}\right]\partial_y\partial_yp\right\|\\
=&\left\|Z^\alpha\left(\partial_y\left(\frac{v}{\gamma p}\right)\partial_y p\right)+\left[Z^\alpha, \frac{1}{\gamma p}\frac{v}{\varphi(y)}\right]\varphi(y)\partial_y\partial_yp\right\|\\
\lesssim&\left\|Z^\alpha\left(\partial_y\left(\frac{v}{\gamma p}\right)\partial_y p\right)\right\|+\sum\limits_{|\beta|\geq 1, \beta+\kappa=\alpha}\left\|Z^\beta\left(\frac{1}{\gamma p}\frac{v}{\varphi(y)}\right)Z^\kappa\left(\varphi(y)\partial_y\partial_yp\right)\right\|\\
\lesssim&(1+P(Q(t)))(\|\partial_y p\|_{m-1}+ \|\partial_y v\|_{m-1}+P(\|p-1\|_{m-1})+\|v\|_{m-1}).
\end{align*}
Substituting  all of the above inequalities into \eqref{3.32} gives
\begin{align}
\label{A10}
&\frac{d}{dt}\varepsilon\left\|\sqrt{\frac{a(\rho, f_4)}{\gamma p}}Z^\alpha\partial_yp\right\|^2+2(2\mu+\lambda)\varepsilon\int\left(1+\frac{(1+f_4)^2}{\gamma(1+\rho)^{\gamma-1}}\right)Z^\alpha\partial_ypZ^\alpha\partial_y^2vd{\bf x}\nonumber\\
\lesssim&C_\delta\varepsilon^2(1+P(Q(t)))(\|\partial_yp\|^2_{m-1}+\|\partial_yu\|^2_{m-1}+\|\partial_y v\|^2_{m-1}+\|u\|^2_{m-1}+\|v\|_{m-1}+P(\|p-1\|_{m}))\nonumber\\
&+\delta\|Z^\alpha\partial_yp\|^2+\varepsilon(1+P(Q(t)))\|Z^\alpha\partial_yp\|^2+C_\delta\varepsilon^2(1+P(Q(t)))\|\partial_yu\|^2_m.
\end{align}

Combining (\ref{3.38a}) and (\ref{A10}) and choosing $\delta$ and $\varepsilon$ suitably small, we have
\begin{align}
\label{A111}
&\frac{d}{dt}\varepsilon\left\|\sqrt{\frac{a(\rho, f_4)}{\gamma p}}Z^\alpha\partial_yp\right\|^2+\left\|\left(1+\frac{(1+f_4)^2}{\gamma(1+\rho)^{\gamma-1}}\right)Z^\alpha\partial_yp\right\|^2+(2\mu+\lambda)^2\varepsilon^2\|Z^\alpha\partial_y^2v\|^2\nonumber\\
\lesssim&(1+P(Q(t)))(\|(v, f_2, f_3, f_4)\|_m^2+\|(u, f_1)\|_{m-1}^2+P(\|p-1\|_{m})+\|\partial_y f_3\|_{m-1}^2+\|\partial_yp\|_{m-2}^2)\nonumber\\
&+\varepsilon^2\|(\partial_xv, \partial_yu)\|_{m}^2+C_\delta\varepsilon^2(1+P(Q(t)))(\|\partial_yp\|^2_{m-1}+\|\partial_yu\|^2_{m-1}+\|\partial_y v\|^2_{m-1}).
\end{align}
We remark that in order to derive \eqref{A111} we have used    the equivalence between $\|\rho\|_{m}$ and  $\|p-1\|_{m}$, and the {\it a priori} assumption that $\|\rho\|_{L^\infty}\leq 1/2, \|f_4\|_{L^\infty}\leq 1/2$ and $Q(t)\leq C$; moreover, the smallness of $\varepsilon$ is required, which is used to absorb the term of $\varepsilon(1+P(Q(t)))\|Z^\alpha\partial_yp\|$ on the right hand side.

Summing\ (\ref{A111}) over $|\alpha|\leq m-1$, choosing $\varepsilon$ suitably small and using the mathematical induction argument yield that
\begin{align}
\label{A11}
&\frac{d}{dt}\varepsilon\left\|\sqrt{\frac{a(\rho, f_4)}{\gamma p}}\partial_yp\right\|_{m-1}^2+\left\|\left(1+\frac{(1+f_4)^2}{\gamma(1+\rho)^{\gamma-1}}\right)\partial_yp\right\|_{m-1}^2+(2\mu+\lambda)^2\varepsilon^2\|Z^\alpha\partial_y^2v\|_{m-1}^2\nonumber\\
\lesssim&(1+P(Q(t)))(\|(u, v, f_2, f_3, f_4)\|_m^2+\|f_1\|_{m-1}^2+P(\|p-1\|_{m})+\|\partial_y f_3\|_{m-1}^2)\\
&+\varepsilon^2\|(\partial_xv, \partial_yu)\|_{m}^2. \nonumber
\end{align}
where (\ref{3.9}) and (\ref{3.12}) are used.
\subsection{Estimate of $\partial_y^2 p$} We now derive the estimates on $\partial_y^2 p$.

\medskip

\subsubsection*{Step 1}
Applying the operator $Z^\alpha\partial_y\ (|\alpha|\leq m-2)$ on (\ref{3.27}) leads to
\begin{align}
\label{3.34}
&\left(1+\frac{(1+f_4)^2}{\gamma(1+\rho)^{\gamma-1}}\right)Z^\alpha\partial^2_yp-(2\mu+\lambda)\varepsilon Z^\alpha \partial_y^3v\nonumber\\
=&Z^\alpha\partial_y\{-(1+\rho) \partial_t v-(1+\rho) u\partial_xv-(1+\rho) v\partial_yv+(1+\rho)(1+f_1)\partial_xf_3+(1+\rho) f_3\partial_yf_3\}\nonumber\\
&+Z^\alpha\partial_y\{(1+\rho) f_2\partial_xf_4-(1+f_4)\partial_x((1+\rho) f_2)+\mu\varepsilon\partial_x^2v+(\mu+\lambda)\varepsilon\partial_yu_x\}+\mathcal{C}_{10}^\alpha
\end{align}
with
\begin{align*}
\mathcal{C}_{10}^\alpha=\left[Z^\alpha\partial_y, (1+\frac{(1+f_4)^2}{\gamma(1+\rho)^{\gamma-1}})\right]\partial_y p.
\end{align*}
Taking $L^2$ inner product on both sides of (\ref{3.34}), we arrive at
\begin{align}
\label{3.35}
&\left\|\left(1+\frac{(1+f_4)^2}{\gamma(1+\rho)^{\gamma-1}}\right)Z^\alpha\partial^2_yp\right\|^2+(2\mu+\lambda)^2\varepsilon^2\|Z^\alpha\partial_y^3v\|^2\nonumber\\
&-2(2\mu+\lambda)\varepsilon\int\left(1+\frac{(1+f_4)^2}{\gamma(1+\rho)^{\gamma-1}}\right)Z^\alpha\partial^2_ypZ^\alpha\partial_y^3vd{\bf x}\nonumber\\
\leq&\|Z^\alpha\partial_y\{-(1+\rho)\partial_t v-(1+\rho) u\partial_xv-(1+\rho) v\partial_yv+(1+\rho)(1+f_1)\partial_xf_3+(1+\rho)f_3\partial_yf_3\}\|^2\nonumber\\
&+\|Z^\alpha\partial_y\{(1+\rho) f_2\partial_xf_4-(1+f_4)\partial_x((1+\rho) f_2)+\mu\varepsilon\partial_x^2v+(\mu+\lambda)\varepsilon\partial_yu_x\}\|^2+\|\mathcal{C}_{10}^\alpha\|^2\nonumber\\
\lesssim& (1+P(Q(t)))(\|(\rho, v, f_2, f_3, f_4)\|_{m-1}^2+\|(u, f_1)\|_{m-2}^2)\nonumber\\
&+(1+P(Q(t)))(\|\partial_y(\rho, v, f_2, f_3, f_4)\|_{m-1}^2+\| \partial_y(u, f_1)\|_{m-2}^2)+\varepsilon^2\|\partial_yv\|_{m}^2 \\
& +\varepsilon^2\|\partial^2_yu\|_{m-1}^2+\|\mathcal{C}_{10}^\alpha\|^2\nonumber\\
\lesssim& (1+P(Q(t)))(\|(\rho, f_2, f_3)\|_{m-1}^2+\|(u, f_1)\|_{m-2}^2+P(\|f_4\|_m)+\|v\|^2_m) \nonumber \\
&+(1+P(Q(t)))(\|\partial_y(\rho, f_2, f_3, f_4)\|_{m-1}^2+\| \partial_y f_1\|_{m-2}^2)+\varepsilon^2\|\partial_yv\|_{m}^2+\varepsilon^2\|\partial^2_yu\|_{m-1}^2+\|\mathcal{C}_{10}^\alpha\|^2\nonumber.
\end{align}
It  remains to estimate the term of commutator $\mathcal{C}_{10}^\alpha$. First we note that
\begin{align*}
\|\mathcal{C}_{10}^\alpha\|^2=&\left\|\left[Z^\alpha\partial_y, \left(1+\frac{(1+f_4)^2}{\gamma(1+\rho)^{\gamma-1}}\right)\right]\partial_y p\right\|^2\\
=&\left\|Z^\alpha\left(\partial_y\left(\frac{(1+f_4)^2}{\gamma(1+\rho)^{\gamma-1}}\right)\partial_yp\right)\right\|+\left\|\left[Z^\alpha, \left(1+\frac{(1+f_4)^2}{\gamma(1+\rho)^{\gamma-1}}\right)\right]\partial^2_y p\right\|\\
\lesssim&(1+P(Q(t)))(P(\|f_4\|_{m-2})+P(\|\rho\|_{m-2})+\|\partial_y(\rho, f_4)\|_{m-2}^2)\\
&+\sum\limits_{|\beta|\geq 1, \beta+\kappa=\alpha}\left\|Z^\beta\left(\frac{(1+f_4)^2}{\gamma(1+\rho)^{\gamma-1}}\right)Z^\kappa\partial_y^2p\right\|^2.
\end{align*}
By the similar arguments to  (\ref{B1}), we have
\begin{align*}
&\sum\limits_{|\beta|\geq 1, \beta+\kappa=\alpha}\left\|Z^\beta\left(\frac{(1+f_4)^2}{\gamma(1+\rho)^{\gamma-1}}\right)Z^\kappa\partial_y^2p\right\|^2\\
\lesssim &\sum\limits_{1\leq|\beta|\leq |\alpha|/2, \beta+\kappa=\alpha}\|Z^\kappa\partial_y^2p\|^2_{L^2_{{\bf x}}}\left\|Z^\beta\left(\frac{(1+f_4)^2}{\gamma(1+\rho)^{\gamma-1}}\right)\right\|_{L^\infty}^2 \\
&+\sum\limits_{|\beta|> |\alpha|/2, \beta+\kappa=\alpha} \|Z^\kappa\partial_y^2p\|^2_{L^\infty_x(L^2_y)}\left\|Z^\beta\left(\frac{(1+f_4)^2}{\gamma(1+\rho)^{\gamma-1}}\right)\right\|_{L^2_x(L^\infty_y)}^2\\
\lesssim &\sum\limits_{1\leq|\beta|\leq |\alpha|/2, \beta+\kappa=\alpha}\|Z^\kappa\partial_y^2p\|^2_{L^2_{{\bf x}}}\left\|Z^\beta\left(\frac{(1+f_4)^2}{\gamma(1+\rho)^{\gamma-1}}\right)\right\|_{L^\infty_x(L^2_y)}
\left\|\partial_yZ^\beta\left(\frac{(1+f_4)^2}{\gamma(1+\rho)^{\gamma-1}}\right)\right\|_{L^\infty_x(L^2_y)}\\
&+\sum\limits_{|\beta|> |\alpha|/2, \beta+\kappa=\alpha}\|Z^\kappa\partial_y^2p\|^2_{L^\infty_x(L^2_y)}\left\|Z^\beta\left(\frac{(1+f_4)^2}{\gamma(1+\rho)^{\gamma-1}}\right)\right\|_{L^2_{{\bf x}}}
\left\|\partial_yZ^\beta\left(\frac{(1+f_4)^2}{\gamma(1+\rho)^{\gamma-1}}\right)\right\|_{L^2_{{\bf x}}}\\
\lesssim &(1+P(Q(t)))(P(\|f_4\|_{m-2})+P(\|\rho\|_{m-2})+\|\partial_yf_4\|^2_{m-2}+\|\partial_y\rho\|^2_{m-2})\|\partial_y^2p\|^2_{m-3},
\end{align*}
provided that $m>8$.
Consequently,
\begin{align*}
\|\mathcal{C}_{10}^\alpha\|^2\lesssim& (1+P(Q(t)))(P(\|f_4\|_{m-2})+P(\|\rho\|_{m-2})+\|\partial_yf_4\|^2_{m-2}+\|\partial_y\rho\|^2_{m-2})\|\partial_y^2p\|^2_{m-3}\\
&+(1+P(Q(t)))(P(\|f_4\|_{m-2})+P(\|\rho\|_{m-2})+\|\partial_y(\rho, f_4)\|_{m-2}^2).
\end{align*}
Then, we have from \eqref{3.35},
\begin{align}
\label{A12a}
&\left\|\left(1+\frac{(1+f_4)^2}{\gamma(1+\rho)^{\gamma-1}}\right)Z^\alpha\partial^2_yp\right\|^2+(2\mu+\lambda)^2\varepsilon^2\|Z^\alpha\partial_y^3v\|^2\nonumber\\
&-2(2\mu+\lambda)\varepsilon\int\left(1+\frac{(1+f_4)^2}{\gamma(1+\rho)^{\gamma-1}}\right)Z^\alpha\partial^2_ypZ^\alpha\partial_y^3vd{\bf x}\nonumber\\
\lesssim& (1+P(Q(t)))(\|(\rho, f_2, f_3)\|_{m-1}^2+\|(u, f_1)\|_{m-2}^2+P(\|f_4\|_m)+\|v\|_m^2)\nonumber\\
&+(1+P(Q(t)))(\|\partial_y(\rho, f_2, f_3, f_4)\|_{m-1}^2+\| \partial_y f_1\|_{m-2}^2)+\varepsilon^2\|\partial_yv\|_{m}^2+\varepsilon^2\| \partial^2_yu\|_{m-1}^2\nonumber\\
&+(1+P(Q(t)))(P(\|f_4\|_{m-2})+P(\|\rho\|_{m-2})+\|\partial_yf_4\|^2_{m-2}+\|\partial_y\rho\|^2_{m-2})\|\partial_y^2p\|^2_{m-3}.
\end{align}
\medskip

\subsubsection*{Step 2}
Again, from the equation of conservation of mass, we have
\begin{align}
\label{3.36}
\partial_t p+\gamma p\partial_yv=-(u\partial_x p+v\partial_y  p)-\gamma p\partial_xu.
\end{align}
Applying the operator $Z^\alpha\partial^2_y\ (|\alpha|\leq m-2)$ on (\ref{3.36}) yields
\begin{align}
\label{3.37}
\partial_t Z^\alpha\partial^2_yp+\gamma pZ^\alpha\partial^3_yv=-u\partial_x Z^\alpha\partial^2_y p-v\partial_y  Z^\alpha\partial^2_yp-Z^\alpha\partial^2_y(\gamma p\partial_xu)+\mathcal{C}_{11}^\alpha,
\end{align}
with
\begin{align*}
\mathcal{C}_{11}^\alpha=-\left[Z^\alpha\partial^2_y, \gamma p\right]\partial_y v-\left[Z^\alpha\partial^2_y,  u \partial_x\right]p-\left[Z^\alpha\partial^2_y,  v \partial_y\right]p.
\end{align*}
Multiplying (\ref{3.37}) by $$2(2\mu+\lambda)\varepsilon\frac{1}{\gamma p}\left(1+ \frac{(1+f_4)^2}{\gamma(1+\rho)^{\gamma-1}}\right) Z^\alpha\partial^2_y p\triangleq 2\frac{a(\rho, f_4)}{\gamma p}\varepsilon Z^\alpha\partial^2_y p$$  and integrating the resulting equality over $\mathbb{R}^2_+$, we get
\begin{align}
\label{3.38}
&\frac{d}{dt}\varepsilon\left\|\sqrt{\frac{a(\rho, f_4)}{\gamma p}}Z^\alpha\partial^2_yp\right\|^2+2(2\mu+\lambda)\varepsilon\int\left(1+\frac{(1+f_4)^2}{\gamma(1+\rho)^{\gamma-1}}\right)Z^\alpha\partial^2_ypZ^\alpha\partial_y^3vd{\bf x}\nonumber\\
=&\varepsilon\int\left(\left(\frac{a(\rho, f_4)}{\gamma p}\right)_t+\left(\frac{u a(\rho, f_4)}{\gamma p}\right)_x+\left(\frac{v a(\rho, f_4)}{\gamma p}\right)_y\right)(Z^\alpha\partial^2_yp)^2d{\bf x}
\nonumber\\
&-\varepsilon\int 2\frac{a(\rho, f_4)}{\gamma p} Z^\alpha\partial^2_y pZ^\alpha\partial^2_y(\gamma p\partial_xu) d{\bf x}+\int 2\frac{a(\rho, f_4)}{\gamma p}\varepsilon Z^\alpha\partial^2_y p\mathcal{C}_{11}^\alpha d{\bf x}.
\end{align}
First,
\begin{align*}
&\varepsilon\left|\int\left(\left(\frac{a(\rho, f_4)}{\gamma p}\right)_t+\left(\frac{u a(\rho, f_4)}{\gamma p}\right)_x+\left(\frac{v a(\rho, f_4)}{\gamma p}\right)_y\right)(Z^\alpha\partial^2_yp)^2d{\bf x}\right|\\
\lesssim& \varepsilon(1+P(Q(t)))\|Z^\alpha\partial^2_yp\|^2,
\end{align*}
and
\begin{align*}
&|\varepsilon\int \frac{a(\rho, f_4)}{\gamma p} Z^\alpha\partial^2_y pZ^\alpha\partial^2_y(\gamma p\partial_xu) d{\bf x}|\\
\lesssim& \delta\|Z^\alpha\partial^2_y p\|^2+C_\delta \varepsilon^2(1+P(Q(t)))\|Z^\alpha\partial^2_y(\gamma p\partial_xu)\|^2\\
\lesssim& \delta\|Z^\alpha\partial^2_y p\|^2+C_\delta \varepsilon^2(1+P(Q(t)))(\|Z^\alpha(\partial^2_y p\partial_xu)\|^2+\|Z^\alpha(\partial_y p\partial_x\partial_yu)\|^2+\|Z^\alpha( p\partial_x\partial^2_y u)\|^2).
\end{align*}
Note that
\begin{align*}
&\left\|Z^\alpha\left(\partial_y^2p\partial_xu\right)\right\|^2+\left\|Z^\alpha\left(\partial_yp\partial_x\partial_yu\right)\right\|^2\\
=&\left\|Z^\alpha\left(\frac{\partial_xu}{\varphi(y)}\varphi(y)\partial_y\partial_yp\right)\right\|^2+\left\|Z^\alpha\left(\partial_yp\partial_x\partial_yu\right)\right\|^2\\
\lesssim&\|\partial_yu\|^2_{1, \infty}\left\|\partial_yp\right\|^2_{m-1}+\left\|\partial_yp\right\|^2_{1,\infty}\|\partial_yu\|^2_{m-1}\\
\lesssim&(1+P(Q(t)))(\|\partial_yu\|^2_{m-1}+\|\partial_yp\|^2_{m-1}),
\end{align*}
and
\begin{align*}
&\left\|Z^\alpha(p\partial_x\partial^2_yu)\right\|^2\\
\lesssim&\sum\limits_{|\beta|\leq |\alpha|/2, \beta+\kappa=\alpha}\left\|Z^\beta pZ^\kappa\partial_x\partial^2_yu \right\|^2
+\sum\limits_{|\beta|>|\alpha|/2, \beta+\kappa=\alpha}\left\|Z^\beta p Z^\kappa\partial_x\partial^2_yu \right\|^2\\
\lesssim&\sum\limits_{|\beta|\leq |\alpha|/2, \beta+\kappa=\alpha}\left \|Z^\beta p \right\|^2_{L_{x,y}^\infty}\left\|Z^\kappa\partial_x\partial^2_yu \right \|^2_{L^2_{{\bf x}}} \\
&+\sum\limits_{|\beta|>|\alpha|/2, \beta+\kappa=\alpha}\left\|Z^\beta p \right\|^2_{L^2_{x}(L^\infty_y)}\|Z^\kappa\partial_x\partial^2_yu \|^2_{L_{x}^\infty(L^2_y)}\\
\lesssim&\sum\limits_{|\beta|\leq |\alpha|/2, \beta+\kappa=\alpha}\left \|Z^\beta p \right\|_{L_{x}^\infty(L^2_y)}\left \|\partial_yZ^\beta p \right\|_{L_{x}^\infty(L^2_y)}\left\|Z^\kappa\partial_x\partial^2_yu \right \|^2_{L^2_{{\bf x}}} \\
&+\sum\limits_{|\beta|>|\alpha|/2, \beta+\kappa=\alpha}\left\|Z^\beta p \right\|_{L^2_{{\bf x}}}\left\|\partial_yZ^\beta p \right\|_{L^2_{{\bf x}}}\|Z^\kappa\partial_x\partial^2_yu \|^2_{L_{x}^\infty(L^2_y)}\\
\lesssim&(1+P(Q(t)))(1+\|p-1\|_{m-2}^2+\|\partial_yp\|^2_{m-2})\|\partial^2_yu\|_{m-1}^2,
\end{align*}
provided that $m>4$.

The commutator of $\mathcal{C}_{11}^\alpha$ is estimated as follows.
For the first term in $\mathcal{C}_{11}^\alpha$, it is noted that
\begin{align*}
\left\|\left[Z^\alpha\partial^2_y, \gamma p\right]\partial_yv\right\|\leq \left\|\gamma Z^\alpha\left(\partial_y^2p\partial_yv\right)\right\|+\left\|2\gamma Z^\alpha\left(\partial_y p\partial^2_yv\right)\right\|+\left\|[Z^\alpha, \gamma p]\partial^3_yv\right\|.
\end{align*}
For the above three terms on the right hand side, we have the following estimates:
\begin{align*}
&\left\|Z^\alpha\left(\partial_y^2 p\partial_yv\right)\right\|^2\\
=&\int \left(\sum\limits_{|\beta|\leq |\alpha|/2, \beta+\kappa=\alpha}Z^\beta(\partial_y^2p)Z^\kappa\partial_yv+\sum\limits_{|\beta|> |\alpha|/2, \beta+\kappa=\alpha}Z^\beta(\partial_y^2 p)Z^\kappa\partial_yv\right)^2d{\bf x}\\
\lesssim&\int \left(\sum\limits_{|\beta|\leq |\alpha|/2, \beta+\kappa=\alpha}Z^\beta\left(\partial_y^2p\right)Z^\kappa\partial_yv\right)^2d{\bf x} \\
&+\int\left(\sum\limits_{|\beta|> |\alpha|/2, \beta+\kappa=\alpha}Z^\beta\left(\partial_y^2 p\right)Z^\kappa\partial_yv\right)^2d{\bf x}\\
\lesssim &\sum\limits_{|\beta|\leq |\alpha|/2, \beta+\kappa=\alpha}\|Z^\kappa\partial_yv\|^2_{L^2_x(L^\infty_{y})}\left\|Z^\beta\partial_y^2 p\right\|_{L^\infty_x(L^2_{y})}^2\\
&+\sum\limits_{|\beta|> |\alpha|/2, \beta+\kappa=\alpha}\|Z^\kappa\partial_yv\|^2_{L^\infty_x(L^\infty_{y})}\left\|Z^\beta\partial_y^2 p\right\|_{L^2_x(L^2_{y})}^2\\
\lesssim &\sum\limits_{|\beta|\leq |\alpha|/2, \beta+\kappa=\alpha}\|Z^\kappa\partial_yv\|_{L^2_x(L^2_{y})}\|\partial_yZ^\kappa\partial_yv\|_{L^2_x(L^2_{y})}\left\|Z^\beta\partial_y^2 p\right\|_{L^2_x(L^2_{y})}\left\|Z^{\beta+1}\partial_y^2 p\right\|_{L^\infty_x(L^2_{y})}\\
&+\sum\limits_{|\beta|> |\alpha|/2, \beta+\kappa=\alpha}\|Z^\kappa\partial_yv\|^{1/2}_{L^2_x(L^2_{y})}\|\partial_yZ^\kappa\partial_yv\|^{1/2}_{L^2_x(L^2_{y})}\|Z^{\kappa+1}\partial_yv\|^{1/2}_{L^2_x(L^2_{y})}\\
&\qquad\qquad\qquad\qquad\qquad  \times \|\partial_yZ^{\kappa+1}\partial_yv\|^{1/2}_{L^2_x(L^2_{y})}\left\|Z^\beta\partial_y^2 p\right\|_{L^2_x(L^2_{y})}^2\\
\lesssim &(1+P(Q(t)))(\|\partial_yv\|^2_{m-2}+\|\partial^2_yv\|_{m-2}^2)\|\partial_y^2p\|_{m-2}^2,
\end{align*}
and similarly,
\begin{align*}
\left\|\left[Z^\alpha, \gamma p\right]\partial_y\partial_y^2v\right\|^2\lesssim (1+P(Q(t)))(\|p-1\|_{m-2}^2+\|\partial_yp\|_{m-2}^2)\|\partial_y^3v\|_{m-3}^2,
\end{align*}
as well as
\begin{align*}
\left\|2Z^\alpha\left(\partial_y(\gamma p)\partial_y\partial_yv\right)\right\|^2
\lesssim(1+P(Q(t)))\|\partial_yp\|^2_{m-2}(\|\partial_y^2 v\|_{m-2}^2+\|\partial_y^3v\|_{m-2}^2).
\end{align*}
Consequently, for the first term in $\mathcal{C}_{11}^\alpha$, one has
\begin{align} \label{446}
&\left\|\left[Z^\alpha\partial^2_y, \gamma p\right]\partial_yv\right\|\nonumber\\
\lesssim&(1+P(Q(t)))\big\{\|\partial_yv\|^2_{m-2}+\|\partial^2_yv\|_{m-2}^2)\|\partial_y^2p\|_{m-2}^2\big\}\\
&+(1+P(Q(t)))\big\{(\|p-1\|_{m-2}^2+\|\partial_yp\|_{m-2}^2)\|\partial_y^3v\|_{m-3}^2+\|\partial_yp\|^2_{m-2}(\|\partial_y^2 v\|_{m-2}^2+\|\partial_y^3v\|_{m-2}^2)\big\}\nonumber.
\end{align}
The second term in $\mathcal{C}_{11}^\alpha$ can be dealt with similarly as  the following:
\begin{align} \label{447}
&\|\left[Z^\alpha\partial^2_y,  u \partial_x\right]p\|^2\nonumber\\
\lesssim&(1+P(Q(t)))\big\{\|\partial_yu\|^2_{m-2}+\|\partial_yp\|_{m-1}^2\big\}\\
&+(1+P(Q(t)))\big\{(\|p-1\|_{m-1}^2+\|\partial_yp\|_{m-1}^2)\|\partial_y^2u\|_{m-2}^2+\|\partial_y^2p\|^2_{m-2}(\|u\|_{m-2}^2+\|\partial_yu\|_{m-2}^2)\big\}\nonumber.
\end{align}
For the third term in $\mathcal{C}_{11}^\alpha$, we notice that
\begin{align*}
\left[Z^\alpha\partial^2_y,  v \partial_y\right]p
=&Z^\alpha\left(\partial^2_y v\partial_yp\right)+Z^\alpha\left(2\partial_yv\partial^2_yp\right)+\left[Z^\alpha, v\right]\partial^3_yp\\
=&Z^\alpha\left(\partial^2_y v\partial_yp\right)+2 Z^\alpha(\partial_yv\partial^2_yp)+\left[Z^\alpha, \frac{v}{\varphi(y)}\right]\varphi(y)\partial_y\partial^2_yp,
\end{align*}
which  can also be handled similarly as the following:
\begin{align}\label{448}
&\left\|\left[Z^\alpha\partial^2_y,  v \partial_y\right]p\right\| \nonumber \\
\lesssim&(1+P(Q(t)))\big\{\|\partial_yp\|_{m-2}^2(\|\partial_y^2v\|_{m-2}^2+\|\partial_y^3v\|_{m-2}^2)+\|\partial_y^2p\|^2_{m-2}(\|\partial_yv\|_{m-2}^2+\|\partial_y^2v\|_{m-2}^2)\big\}.
\end{align}
Consequently, we have from \eqref{3.38}-\eqref{448},
\begin{align}
\label{A13}
&\frac{d}{dt}\varepsilon\left\|\sqrt{\frac{a(\rho, f_4)}{\gamma p}}Z^\alpha\partial^2_yp\right\|^2+2(2\mu+\lambda)\varepsilon\int\left(1+\frac{(1+f_4)^2}{\gamma(1+\rho)^{\gamma-1}}\right)Z^\alpha\partial^2_ypZ^\alpha\partial_y^3vd{\bf x}\nonumber\\
&\lesssim (\delta+(1+P(Q(t)))\varepsilon)\|Z^\alpha\partial^2_yp\|^2+C_\delta\varepsilon^2(1+P(Q(t)))(\|\partial_yu\|^2_{m-1}+\|\partial_yp\|^2_{m-1})\nonumber\\
& +C_\delta\varepsilon^2(1+P(Q(t)))(1+\|p-1\|_{m-2}^2+\|\partial_yp\|^2_{m-2})\|\partial^2_yu\|_{m-1}^2\\
& +C_\delta\varepsilon^2(1+P(Q(t)))\big\{(\|\partial_yv\|^2_{m-2}+\|\partial^2_yv\|_{m-2}^2)\|\partial_y^2p\|_{m-2}^2\big\}\nonumber\\
&+C_\delta\varepsilon^2(1+P(Q(t)))\big\{(\|p-1\|_{m-2}^2+\|\partial_yp\|_{m-2}^2)\|\partial_y^3v\|_{m-3}^2+\|\partial_yp\|^2_{m-2}(\|\partial_y^2 v\|_{m-2}^2+\|\partial_y^3v\|_{m-2}^2)\big\}\nonumber\\
&+C_\delta\varepsilon^2(1+P(Q(t)))\big\{(\|p-1\|_{m-1}^2+\|\partial_yp\|_{m-1}^2)\|\partial_y^2u\|_{m-2}^2+\|\partial_y^2p\|^2_{m-2}(\|u\|_{m-2}^2+\|\partial_yu\|_{m-2}^2)\big\}\nonumber
\end{align}

\medskip

\subsubsection*{Step 3}
Combining (\ref{A12a}) and (\ref{A13}) together and choosing $\delta$ and $\varepsilon$ suitably small lead to
\begin{align}
\label{A14}
&\frac{d}{dt}\varepsilon\left\|\sqrt{\frac{a(\rho, f_4)}{\gamma p}}Z^\alpha\partial^2_yp\right\|^2+\left\|\left(1+\frac{(1+f_4)^2}{\gamma(1+\rho)^{\gamma-1}}\right)Z^\alpha\partial^2_yp\right\|^2+(2\mu+\lambda)^2\varepsilon^2\|Z^\alpha\partial_y^3v\|^2\nonumber\\
&\lesssim (1+P(Q(t)))(\|(\rho, f_2, f_3)\|_{m-1}^2+\|(u, f_1)\|_{m-2}^2+P(\|f_4\|_m)+\|v\|^2_m)\nonumber\\
&+(1+P(Q(t)))(\|\partial_y(\rho, f_2, f_3, f_4)\|_{m-1}^2+\| \partial_y f_1\|_{m-2}^2)+\varepsilon^2\|\partial_yv\|_{m}^2+\varepsilon^2\| \partial^2_yu\|_{m-1}^2\nonumber\\
&+(1+P(Q(t)))(P(\|f_4\|_{m-2})+P(\|\rho\|_{m-2})+\|\partial_yf_4\|^2_{m-2}+\|\partial_y\rho\|^2_{m-2})\|\partial_y^2p\|^2_{m-3}\nonumber\\
&+C_\delta\varepsilon^2(1+P(Q(t)))(\|\partial_yu\|^2_{m-1}+\|\partial_yp\|^2_{m-1})\nonumber\\
& +C_\delta\varepsilon^2(1+P(Q(t)))(1+\|p-1\|_{m-2}^2+\|\partial_yp\|^2_{m-2})\|\partial^2_yu\|_{m-1}^2\\
& +C_\delta\varepsilon^2(1+P(Q(t)))\big\{(\|\partial_yv\|^2_{m-2}+\|\partial^2_yv\|_{m-2}^2)\|\partial_y^2p\|_{m-2}^2\big\}\nonumber\\
&+C_\delta\varepsilon^2(1+P(Q(t)))\big\{(\|p-1\|_{m-2}^2+\|\partial_yp\|_{m-2}^2)\|\partial_y^3v\|_{m-3}^2+\|\partial_yp\|^2_{m-2}(\|\partial_y^2 v\|_{m-2}^2+\|\partial_y^3v\|_{m-2}^2)\big\}\nonumber\\
&+C_\delta\varepsilon^2(1+P(Q(t)))\big\{(\|p-1\|_{m-1}^2+\|\partial_yp\|_{m-1}^2)\|\partial_y^2u\|_{m-2}^2+\|\partial_y^2p\|^2_{m-2}(\|u\|_{m-2}^2+\|\partial_yu\|_{m-2}^2)\big\}\nonumber
\end{align}
\subsection{Estimates of $\partial_y f_1$ and $\partial^2_y f_1$}
As for the normal derivatives of $f_1$, we use the following formulation
\begin{align*}
(1+f_1)(1+f_4)-f_2f_3=\frac{1}{1+\rho}
\end{align*}
due to $(1+\rho)\hbox{det}{\bf F}=1$. Then
\begin{align*}
\partial_yf_1=&\frac{1}{(1+f_4)}\left\{\partial_y\left(\frac{1}{1+\rho}\right)+\partial_y(f_2f_3)-(1+f_1)\partial_yf_4\right\}\\
=&\frac{1}{(1+f_4)}\left\{\partial_y\left(\frac{1}{1+\rho}\right)+f_3\partial_yf_2-\frac{f_2}{1+\rho}(f_3\partial_y\rho+\partial_x((1+\rho) (1+f_1)))\right\}\\
&+\frac{1}{(1+f_4)}\left\{\frac{1+f_1}{1+\rho}((1+f_4)\partial_y\rho+\partial_x((1+\rho) f_2))\right\}.
\end{align*}
due to
\begin{align*}
\partial_x((1+\rho) (1+f_1))+\partial_y((1+\rho) f_3)=0,\qquad \partial_x((1+\rho) f_2)+\partial_y((1+\rho) (1+f_4))=0.
\end{align*}
Then, the following two inequalities hold true:
\begin{align}\label{A20}
&\|\partial_yf_1\|^2_{m-1}\nonumber\\
\lesssim& (1+P(Q(t)))(\|\partial_y\rho\|_{m-1}^2+\|f_2\|_{m}^2+\|\partial_yf_3\|_{m-1}^2+P(\|\rho\|_{m})+\|f_1\|_{m}^2+P(\|f_4\|_{m-1}))
\end{align}
and
\begin{align}
\label{A21}
&\|\partial_y^2f_1\|^2_{m-2}\nonumber\\
\lesssim& (1+P(Q(t)))(\|f_2\|_{m}^2+P(\|\rho\|_{m})+\|f_1\|_{m}^2+P(\|f_4\|_{m-1}))\nonumber\\
&+(1+P(Q(t)))(\|\partial^2_y\rho\|_{m-2}^2+\|\partial_y\rho\|_{m-1}^2+\|\partial_yf_2\|_{m-1}^2+\|\partial_yf_3\|_{m-1}^2+\|\partial_yf_4\|_{m-2}^2),
\end{align}
where (\ref{A20}) is used.
\subsection{Estimates of $\partial^i_y f_3$ and $\partial^i_y f_4\ (i=1,2)$}
By the divergence free conditions, we have
\begin{align*}
\partial_yf_3=\frac{1}{1+\rho}\{-\partial_x((1+\rho)(1+f_1))-f_3\partial_y\rho\},
\end{align*}
and
\begin{align*}
\partial_yf_4=\frac{1}{1+\rho}\{-\partial_x((1+\rho)f_2)-(1+f_4)\partial_y\rho\}.
\end{align*}
Consequently,
\begin{align*}
\|\partial_yf_3\|_{m-1}^2\lesssim (1+P(Q(t)))(\|\rho\|_m^2+\|f_1\|_m^2)+\|\rho\|_{1,\infty}^2\|\partial_yf_3\|_{m-1}^2,
\end{align*}
where we  have used the following argument:
\begin{align*}
\|f_3\partial_y\rho\|_{m-1}=&\left\|\frac{f_3}{\varphi(y)}\varphi(y)\partial_y\rho\right\|_{m-1}\\
\lesssim&\|\partial_yf_3\|_{L^\infty}\|\rho\|_m+\|\varphi(y)\partial_y\rho\|_{L^\infty}\|\partial_yf_3\|_{m-1}.
\end{align*}
By using the {\it a priori} assumption of $\|\rho\|_{1,\infty}\leq C_0\sigma_0$ with $\sigma_0$ being sufficiently small, we have
\begin{align}\label{A15}
\|\partial_yf_3\|_{m-1}^2\lesssim (1+P(Q(t)))(\|\rho\|_m^2+\|f_1\|_m^2).
\end{align}
By similar arguments, it follows that
\begin{align}\label{A16}
\|\partial_yf_4\|_{m-1}^2\lesssim (1+P(Q(t)))(\|\rho\|_m^2+\|f_2\|_m^2+\|\partial_y\rho\|_{m-1}^2+\|f_4\|_{m-1}^2);
\end{align}
moreover,
\begin{align}\label{A17}
&\|\partial^2_yf_3\|_{m-2}^2\nonumber\\
\lesssim& (1+P(Q(t)))\|\rho\|_{m-1}^2+\|(f_1, f_3)\|_{m-1}^2+\|\partial_y(\rho, f_1)\|_{m-1}^2+\|\partial_yf_3\|_{m-2}^2),
\end{align}
and
\begin{align}\label{A18}
&\|\partial^2_yf_4\|_{m-2}^2\nonumber\\
\lesssim& (1+P(Q(t)))(\|(\rho, f_2, f_4)\|_{m-1}^2+\|\partial_y(\rho, f_2)\|_{m-1}^2+\|\partial_yf_4\|_{m-2}^2+(1+\|f_4\|_m^2)\|\partial^2_y\rho\|_{m-2}^2).
\end{align}

Finally, by combining  the estimates (\ref{bbb}), (\ref{3.9}), (\ref{3.10}), (\ref{3.12}), (\ref{3.14}), (\ref{A5}), (\ref{A8}), (\ref{A11}), (\ref{A14}), (\ref{A20})-\eqref{A18} 
we shall be able to complete the proof of Proposition \ref{P4.1}.
We remark that for this purpose  we may apply the multiplications: $(\ref{A5})\times M_0$ and $(\ref{A11})\times M_1$ with $M_0$ and $M_1$ being suitably large to cancel the terms $\varepsilon^2\|\partial_y^2v\|_{m-1}^2$ in (\ref{A8}) and $\varepsilon^2\|\partial_y^2u\|_{m-1}^2$ in (\ref{A14}), moreover, it can also cancel $\|\partial_y\rho\|_{m-1}^2$ in the right hand sides of (\ref{A20}) and (\ref{A16}) due to {\it a priori} assumption $Q(t)\leq C$, and $\delta$ in (\ref{bbb}) is chosen to be suitably small.
To derive the $L^2_{t, {\bf x}}$-norms of second order normal derivatives, we also need the following facts:
\begin{align*}
\sup\limits_{0\leq s\leq t}\|\partial_y(p, f_2, f_4)(s)\|_{m-2}^2\lesssim \|\partial_y(p, f_2, f_4)(0)\|_{m-2}^2+\int_0^t\|\partial_y(p, f_2, f_4)(s)\|_{m-1}^2ds\leq C_0\sigma_0,
\end{align*}
and
\begin{align*}
\sup\limits_{0\leq s\leq t}\varepsilon^2\|\partial_y^i(u, v)(s)\|_{m-2}^2\lesssim \varepsilon^2\|\partial_y^i(u, v)(0)\|_{m-2}^2+\int_0^t\varepsilon^2\|\partial_y^i(u, v)(s)\|_{m-1}^2ds\leq C_0\sigma_0,\ (i=1,2),
\end{align*}
due to the condition (\ref{CCCC}) and the {\it a priori} assumptions:
\begin{align*}
\int_0^t\|\partial_y(p, f_2, f_4)(s)\|_{m-1}^2ds\leq (C_0-1)\sigma_0, \quad \int_0^t\varepsilon^2\|\partial_y^i(u, v)(s)\|_{m-1}^2ds\leq (C_0-1)\sigma_0,\  (i=1,2),
\end{align*}
where the constant $C_0$ is suitably large 
and $\sigma_0$ is sufficiently small in Theorem \ref{Thm1}. Then, the $L^2_{t, {\bf x}}$-norms of the second order normal derivatives appearing on the right hand sides of (\ref{A8}) and (\ref{A14}) can be absorbed by related terms on the left hand sides due to the smallness of $\sigma_0$ and $\varepsilon$.
\bigskip

\section{Proof of Theorem \ref{Thm1}}

We are ready to prove the   estimate \eqref{THM1} in Theorem \ref{Thm1}. Combining the estimates in Propositions \ref{P3.1} and \ref{P4.1}, we have
\begin{align}
\label{A19}
&\|(p-1, {\bf u}, {\bf G_1}, {\bf G_2})(t)\|_m^2+\varepsilon\|\partial_yf_2(t)\|_{m-1}^2+\varepsilon\left\|\partial_yp\right\|_{m-1}^2+\varepsilon\|\partial_y^2f_2(t)\|_{m-2}^2+\varepsilon\left\|\partial_y^2p\right\|_{m-2}^2\nonumber\\
&+\int_0^t(\varepsilon\|\nabla{\bf u}(\tau)\|_{m}^2+\varepsilon^2(\|\partial_y^2u\|_{m-1}^2+\|\partial_y^3u\|_{m-2}^2)+\varepsilon^2(\|\partial_y^2v\|_{m-1}^2+\|\partial_y^3v\|_{m-2}^2))d\tau\nonumber\\
&+\int_0^t(\|\partial_y({\bf u}, {\bf G_1}, {\bf G_2})(\tau)\|^2_{m-1}+\|\partial^2_y({\bf u}, {\bf G_1}, {\bf G_2})(\tau)\|^2_{m-2})d\tau \nonumber\\
&+\int_0^t\left(\left\|\partial_yp\right\|^2_{m-1}+\left\|\partial^2_yp\right\|^2_{m-2}\right)d\tau\nonumber\\
\lesssim&\|(p-1, {\bf u}, {\bf G_1}, {\bf G_2})(0)\|_m^2+\varepsilon\|\partial_yf_2(0)\|_{m-1}^2+\varepsilon\left\|\partial_yp(0)\right\|_{m-1}^2+\varepsilon\|\partial^2_yf_2(0)\|_{m-2}^2\nonumber\\
&+\varepsilon\left\|\partial^2_yp(0)\right\|_{m-2}^2+(1+P(Q(t)))\int_0^t\|(\rho, {\bf u}, {\bf G_1}, {\bf G_2})(\tau)\|_m^2d\tau.
\end{align}
Set
\begin{align*}
W(t)=&\|(p-1, {\bf u}, {\bf G_1}, {\bf G_2})(t)\|_m^2+\varepsilon\|\partial_yf_2(t)\|_{m-1}^2+\varepsilon\left\|\partial_yp\right\|_{m-1}^2+\varepsilon\|\partial_y^2f_2(t)\|_{m-2}^2+\varepsilon\left\|\partial_y^2p\right\|_{m-2}^2\nonumber\\
&+\int_0^t(\varepsilon\|\nabla{\bf u}(\tau)\|_{m}^2+\varepsilon^2(\|\partial_y^2u\|_{m-1}^2+\|\partial_y^3u\|_{m-2}^2)+\varepsilon^2(\|\partial_y^2v\|_{m-1}^2+\|\partial_y^3v\|_{m-2}^2))d\tau\\
&+\int_0^t(\|\partial_y({\bf u}, {\bf G_1}, {\bf G_2})(\tau)\|^2_{m-1}+\|\partial^2_y({\bf u}, {\bf G_1}, {\bf G_2})(\tau)\|^2_{m-2})d\tau\\
&+\int_0^t\left(\left\|\partial_yp\right\|^2_{m-1}
+\left\|\partial^2_yp\right\|^2_{m-2}\right)d\tau.
\end{align*}
By Lemma \ref{L2}, we have
\begin{align*}
&\|(p-1, {\bf u}, {\bf G_1}, {\bf G_2})(t)\|^2_{1, \infty}\\
\lesssim& \|(p-1, {\bf u}, {\bf G_1}, {\bf G_2})(0)\|_{3}^2+\|\partial_y(p, {\bf u}, {\bf G_1}, {\bf G_2})(0)\|_{3}^2\\
&+\int_0^t(\|(p-1, {\bf u}, {\bf G_1}, {\bf G_2})(\tau)\|_{4}^2+\|\partial_y(p, {\bf u}, {\bf G_1}, {\bf G_2})(\tau)\|_{3}^2)d\tau\lesssim W(t)(1+t)+\sigma_0,
\end{align*}
and
\begin{align*}
&\|(\nabla p, \nabla {\bf u}, \nabla {\bf G_1}, \nabla {\bf G_2})(t)\|_{1, \infty}\\
\lesssim& \|(\nabla p, \nabla {\bf u}, \nabla {\bf G_1}, \nabla {\bf G_2})(0)\|_{3}^2+\|\partial_y(\nabla p, \nabla {\bf u}, \nabla {\bf G_1}, \nabla {\bf G_2})(0)\|_{3}^2\\
&+\int_0^t(\|(\nabla p, \nabla {\bf u}, \nabla {\bf G_1}, \nabla {\bf G_2})(\tau)\|_{4}^2+\|\partial_y(\nabla p, \nabla {\bf u}, \nabla {\bf G_1}, \nabla {\bf G_2})(\tau)\|_{3}^2)d\tau\\
\lesssim& W(t)(1+t)+\sigma_0,
\end{align*}
provided that $m> 5$. Then one has
\begin{align}
W(t)\lesssim& \|(p-1, {\bf u}, {\bf G_1}, {\bf G_2})(0)\|_m^2+\varepsilon\|\partial_yf_2(0)\|_{m-1}^2+\varepsilon\left\|\partial_yp(0)\right\|_{m-1}^2\nonumber\\
&+\varepsilon\|\partial^2_yf_2(0)\|_{m-2}^2+\varepsilon\left\|\partial^2_yp(0)\right\|_{m-2}^2+(1+P(W(t)(1+t)+\sigma_0))W(t)t.
\end{align}
Let the time $t$ and $\sigma_0$ be suitably small, it follows that
\begin{align}
W(t)=&\|(p-1, {\bf u}, {\bf G_1}, {\bf G_2})(t)\|_m^2+\varepsilon\|\partial_yf_2(t)\|_{m-1}^2+\varepsilon\left\|\partial_yp\right\|_{m-1}^2+\varepsilon\|\partial_y^2f_2(t)\|_{m-2}^2+\varepsilon\left\|\partial_y^2p\right\|_{m-2}^2\nonumber\\
&+\int_0^t(\varepsilon\|\nabla{\bf u}(\tau)\|_{m}^2+\varepsilon^2(\|\partial_y^2u\|_{m-1}^2+\|\partial_y^3u\|_{m-2}^2)+\varepsilon^2(\|\partial_y^2v\|_{m-1}^2+\|\partial_y^3v\|_{m-2}^2))d\tau\nonumber\\
&+\int_0^t(\|\partial_y({\bf u}, {\bf G_1}, {\bf G_2})(\tau)\|^2_{m-1}+\|\partial^2_y({\bf u}, {\bf G_1}, {\bf G_2})(\tau)\|^2_{m-2})d\tau\nonumber\\
&+\int_0^t\left(\left\|\partial_yp\right\|^2_{m-1}
+\left\|\partial^2_yp\right\|^2_{m-2}\right)d\tau\nonumber\\
\lesssim& \|(p-1, {\bf u}, {\bf G_1}, {\bf G_2})(0)\|_m^2+\varepsilon\|\partial_yf_2(0)\|_{m-1}^2+\varepsilon\left\|\partial_yp(0)\right\|_{m-1}^2\nonumber\\
&+\varepsilon\|\partial^2_yf_2(0)\|_{m-2}^2+\varepsilon\left\|\partial^2_yp(0)\right\|_{m-2}^2,
\end{align}
and
\begin{align}
&\|(p-1, {\bf u}, {\bf G_1}, {\bf G_2})(t)\|^2_{1, \infty}+\|(\nabla p, \nabla {\bf u}, \nabla {\bf G_1}, \nabla {\bf G_2})(t)\|_{1, \infty} \nonumber\\
\lesssim& \|(p-1, {\bf u}, {\bf G_1}, {\bf G_2})(0)\|_m^2+\varepsilon\|\partial_yf_2(0)\|_{m-1}^2+\varepsilon\left\|\partial_yp(0)\right\|_{m-1}^2 \nonumber\\
&+\varepsilon\|\partial^2_yf_2(0)\|_{m-2}^2+\varepsilon\left\|\partial^2_yp(0)\right\|_{m-2}^2+\|\partial_y(p, {\bf u}, {\bf G_1}, {\bf G_2})(0)\|_3^2\nonumber\\
&+\|\partial_y(\nabla p, \nabla{\bf u}, \nabla{\bf G_1}, \nabla{\bf G_2})(0)\|_3^2.
\end{align}
Consequently, the following {\it a priori} assumptions hold true:
\begin{align*}
\|\rho\|_{L^\infty}<1/2,\qquad \|f_4\|_{L^\infty}<1/2
\end{align*}
by letting $\sigma_0$ in Theorem \ref{Thm1} be suitably small.
In fact, the following estimates hold true:
\begin{align*}
\|\rho\|_{1, \infty}\leq \frac{C_0}{2}\sigma_0,\qquad \|f_4\|_{1, \infty}\leq \frac{C_0}{2}\sigma_0,
\end{align*}
and
\begin{align*}
\int_0^t\|\partial_y(p, f_2, f_4)(s)\|_{m-1}^2ds\leq \frac{C_0}{2}\sigma_0, \qquad \int_0^t\varepsilon^2\|\partial_y^i(u, v)(s)\|_{m-1}^2ds\leq \frac{C_0}{2}\sigma_0,\ (i=1,2),
\end{align*}
where $C_0$ is a suitably large constant.
Based on the uniform {\it a priori} estimates established above, we can achieve  the estimate \eqref{THM1} and further  verify the inviscid limit in Theorem \ref{Thm1} by the similar arguments to those in \cite{MR}. We omit the details here. The proof of Theorem \ref{Thm1} is completed.

\bigskip

 \section*{Acknowledgement}
 The research  of D. Wang was  partially supported by the National Science Foundation under grant  DMS-1907519.
F.  Xie's research was supported by National Natural Science Foundation of China No.11831003 and Shanghai Science and Technology Innovation Action Plan No. 20JC1413000.

\end{document}